\documentclass[11pt,a4paper]{article}

\usepackage{amsthm}

\theoremstyle{definition}
\newtheorem{theorem}{Theorem}[section]
\newtheorem{corollary}[theorem]{Corollary}
\newtheorem{lemma}[theorem]{Lemma}
\newtheorem{proposition}[theorem]{Proposition}
\newtheorem{definition}[theorem]{Definition}
\newtheorem{remark}[theorem]{Remark}

\usepackage[left=1.in, right=1.in, bottom=1.5in, top=1.5in]{geometry}

\usepackage[defaultsans]{droidsans}

\usepackage{amsmath,amssymb}
\usepackage{bm}
\usepackage{mathtools}

\usepackage{booktabs}
\usepackage[hang,small,bf]{caption}
\usepackage[scale=2]{ccicons}
\usepackage{colortbl,color}
\usepackage{pgfplots}
\pgfplotsset{compat=newest}
\usepgfplotslibrary{dateplot}
\usepackage{graphicx} 
\graphicspath{{figures/}}
\usepackage{subcaption}
\usepackage{xcolor,xspace}
\usepackage{here}
\usepackage{pifont}

\newcommand{\cmarkg}{{\color{green}\ding{51}}}
\newcommand{\xmarkr}{{\color{red}\ding{55}}}
\newcolumntype{G}{>{\columncolor[gray]{0.9}}l}

\usepackage{tikz}
\usetikzlibrary{arrows.meta,bbox}
\usetikzlibrary{calc}

\usepackage{fontenc}
\usepackage{newunicodechar}
\usepackage{lmodern}

\usepackage{bibunits}
\usepackage{changepage}
\usepackage{enumitem}
\usepackage{framed}
\usepackage{tcolorbox}
\usepackage[hypertexnames=false]{hyperref}
\hypersetup{
	colorlinks,
	linkcolor={blue!30!black},
	citecolor={blue!50!black},
	urlcolor={blue!80!black}
}
\usepackage{multirow}

\newcommand{\zero}{\bm{0}}

\renewcommand{\u}{\bm{u}}

\newcommand{\x}{\bm{x}}
\newcommand{\y}{\bm{y}}

\renewcommand{\H}{\bm{H}}

\renewcommand{\P}{\bm{P}}
\newcommand{\Q}{\bm{Q}}
\newcommand{\U}{\bm{U}}

\newcommand{\W}{\bm{W}}
\newcommand{\X}{\bm{X}}
\newcommand{\Y}{\bm{Y}}
\newcommand{\Z}{\bm{Z}}

\DeclareMathOperator{\cl}{cl}
\DeclareMathOperator{\interior}{int}
\DeclareMathOperator{\N}{\mathbb{N}}
\DeclareMathOperator{\R}{\mathbb{R}}
\DeclareMathOperator{\C}{\mathcal{C}}
\DeclareMathOperator{\dom}{dom}

\DeclareMathOperator{\dist}{dist}

\DeclareMathOperator{\tr}{Tr}

\DeclareMathOperator*{\argmin}{argmin}

\newcommand{\ie}{\textit{i.e.}}
\newcommand{\etal}{\textit{et al.}}

\newtheorem{assumption}{Assumption}[section]

\usepackage[noend,ruled,vlined,noline]{algorithm2e}
\SetKwInput{Input}{Input}
\SetKwInOut{Output}{output}

\title{Majorization-minimization Bregman proximal gradient algorithms for NMF with the Kullback--Leibler divergence}
\author{Shota Takahashi\thanks{Graduate School of Information Science and Technology, The University of Tokyo, Tokyo, Japan (\href{mailto:shota@mist.i.u-tokyo.ac.jp}{shota@mist.i.u-tokyo.ac.jp})} \and Mirai Tanaka\thanks{Department of Statistical Inference and Mathematics, The Institute of Statistical Mathematics, Tokyo, Japan (\href{mailto:mirai@ism.ac.jp}{mirai@ism.ac.jp}, \href{mailto:shiro@ism.ac.jp}{shiro@ism.ac.jp})}
\thanks{Graduate Institute for Advanced Studies, The Graduate University for Advanced Studies, Kanagawa, Japan}
\and Shiro Ikeda\footnotemark[2] \footnotemark[3]}

\date{\today}

\begin{document}

\maketitle

\begin{abstract}
Nonnegative matrix factorization (NMF) is a popular method in machine learning and signal processing to decompose a given nonnegative matrix into two nonnegative matrices.
In this paper, we propose new algorithms, called majorization-minimization Bregman proximal gradient algorithm (MMBPG) and MMBPG with extrapolation (MMBPGe) to solve NMF.
These iterative algorithms minimize the objective function and its potential function monotonically.
Assuming the Kurdyka--\L{}ojasiewicz property, we establish that a sequence generated by MMBPG(e) globally converges to a stationary point.
We apply MMBPG and MMBPGe to the Kullback--Leibler (KL) divergence-based NMF.
While most existing KL-based NMF methods update two blocks or each variable alternately, our algorithms update all variables simultaneously.
MMBPG and MMBPGe for KL-based NMF are equipped with a separable Bregman distance that satisfies the smooth adaptable property and that makes its subproblem solvable in closed form.
Using this fact, we guarantee that a sequence generated by MMBPG(e) globally converges to a Karush--Kuhn--Tucker (KKT) point of KL-based NMF.
In numerical experiments, we compare proposed algorithms with existing algorithms on synthetic data and real-world data.
\end{abstract}

\section{Introduction}
Nonnegative matrix factorization (NMF) has been studied in machine learning~\cite{Chavoshinejad2023-ul,Huang2023-qd,Lee2000-nl,Seyedi2023-gc,Zhou2023-nv} and signal processing~\cite{Fu2019-ak,Marmin2023-kw,Mokry2023-eb} as well as in mathematical optimization~\cite{Gillis2015-kp,Leplat2023-sx,Phan2023-co}.
Given an observed nonnegative matrix $\X\in\R_+^{m \times n}$, NMF is a method for finding two nonnegative matrices $\W\in\R_+^{m \times r}$ and $\H\in\R_+^{r \times n}$ such that $\X = \W\H$ or $\X \simeq \W\H$.
The dimension $r$ is often assumed to be smaller than both $n$ and $m$, and thus $\W$ and $\H$ have smaller sizes than $\X$.
Both exact NMF, which finds $\W$ and $\H$ such that $\X = \W\H$, and approximate NMF, $\X \simeq \W\H$, are NP-hard~\cite{Vavasis2010-ds}.
In real-world applications, computing $\W$ and $\H$ such that $\X \simeq \W\H$ is the practical goal because $\X = \W\H$ does not hold in general for $r < \min\{m, n\}$.

In order to achieve the above goal, we define the loss function to measure the similarity between $\X$ and $\W\H$ and denote it as $D(\X,\W\H)$ with $D:\R_+^{m \times n}\times\R_+^{m\times n} \to (-\infty, +\infty]$.
A minimization problem of the loss function with regularization is defined by
\begin{align}
    \min_{\W\in\R_+^{m \times r},\H\in\R_+^{r \times n}} \quad D(\X, \W\H) + g(\W, \H), \label{prob:klnmf}
\end{align}
where $g:\R_+^{m \times r}\times\R_+^{r \times n} \to(-\infty,+\infty]$ is a regularization term.
Different loss functions have been proposed in the studies of NMF, such as the Frobenius loss $D(\X,\W\H) = \|\X - \W\H\|_F^2:=\sum_{i,j}(X_{ij} - (\W\H)_{ij})^2$, the componentwise $\ell_1$ loss $D(\X,\W\H) = \|\X - \W\H\|_1 := \sum_{i,j}|X_{ij} - (\W\H)_{ij}|$, the Itakura--Saito (IS) divergence $D(\X,\W\H) = \sum_{i,j}\left(\frac{X_{ij}}{(\W\H)_{ij}} - \log \frac{X_{ij}}{(\W\H)_{ij}}-1\right)$, and the $\beta$-divergence loss, which is a generalization of the IS divergence, the Kullback--Leibler (KL) divergence, and the Frobenius loss.
See also~\cite[Chapter 5]{Gillis2020-dh} for more details of these losses.
In this paper, we focus on the KL divergence~\cite{Kullback1951-yv} as the loss function $D:\R_+^{m \times n} \times \R_{++}^{m \times n}\to\R_+$, given by
\begin{align*}
    D(\X, \W\H) := \sum_{i=1}^m\sum_{j=1}^n\left(X_{ij}\log\frac{X_{ij}}{(\W\H)_{ij}} - X_{ij} + (\W\H)_{ij}\right),
\end{align*}
where we define $0\log 0 := 0$.
In what follows, the optimization problem~\eqref{prob:klnmf} with the KL divergence is called KL-NMF.

Many different approaches have been proposed to solve NMF. Lee and Seung~\cite{Lee2000-nl} proposed the multiplicative updates (MU) for KL-NMF with $g\equiv0$, and Marmin~\etal~\cite{Marmin2023-jh} proposed MU for $\beta$-divergence loss.
Marmin~\etal~\cite{Marmin2023-kw} proposed majorization-minimization methods for $\beta$-divergence loss with $g(\W, \H) = \mu_{\W}\|\W\|_1 + \mu_{\H}\|\H\|_1$ for $\mu_{\W},\mu_{\H} > 0$.
Likewise, when $g\equiv0$, Hien and Gillis~\cite{Hien2021-sh} compared KL-NMF algorithms such as MU, the alternating direction method of multipliers, the block mirror descent, and so on.
Hsieh and Dhillon~\cite{Hsieh2011-ik} applied coordinate descent methods to NMF.
Recently, Hien~\etal~\cite{Hien2025-vc} proposed the MU with extrapolation, which is an accelerated version of MU.
Note that many algorithms update two blocks or each variable alternately and would not converge to a Karush--Kuhn--Tucker (KKT) point from any initial point when general $g \not\equiv 0$.
For example, because a KKT point is equivalent to a stationary point in this case, even when $g\not\equiv0$ is Clarke regular, it holds that $\zero\in\nabla f(\W,\H) + \partial \tilde{g}(\W,\H) \subset \nabla f(\W,\H) + \partial_{\W} \tilde{g}(\W,\H) \times \partial_{\H} \tilde{g}(\W,\H)$, where $f(\W, \H) = D(\X,\W\H)$, $\tilde{g} = g + \delta_{\R_+^{m \times r}\times\R_+^{r \times n}}$, $\partial \tilde{g}$ is the limiting subdifferential of $g$, and $\partial_{\W} \tilde{g}$ and $\partial_{\H} \tilde{g}$ are the partial limiting subdifferential with respect to $\W$ and $\H$, respectively (see also~\cite[Section 2.3]{Takahashi2023-uh}).

Let $f(\W, \H) = D(\X,\W\H)$ with $f:\R_+^{m \times r}\times\R_+^{r \times n}\to(-\infty,+\infty]$.
KL-NMF~\eqref{prob:klnmf} is a composite optimization problem.
A popular approach for solving composite optimization problems would be a proximal algorithm.
The proximal algorithm is a class of algorithms based on the proximal mapping.
It includes many algorithms, such as the proximal point algorithm~\cite{Fukushima1981-wn,Rockafellar1976-th}, the proximal gradient method~\cite{Bruck1975-et,Lions1979-id,Passty1979-sk}, and the fast iterative shrinkage-thresholding algorithm (FISTA)~\cite{Beck2009-kr}.
The proximal gradient algorithm with constant step-sizes requires the global Lipschitz continuity of $\nabla f$.
The proximal gradient algorithm with line search can be applied to problems without the global Lipschitz continuity of $\nabla f$, which is, for example, NMF with the KL divergence loss, the Frobenius loss, and the $\beta$-divergence loss.
For such problems, step-sizes may be underestimated, and the convergence of proximal gradient algorithms would be slow. On the other hand, algorithms based on the smooth adaptable property can be applied to them without line search.

Bolte \etal~\cite{Bolte2018-zt} proposed the Bregman proximal gradient algorithm (BPG), which does not require the global Lipschitz continuity.
BPG and other methods using the Bregman distance (its definition is in Section~\ref{subsec:bregman-distance}) have been actively studied in recent years.
Zhang \etal~\cite{Zhang2019-ia} proposed the Bregman proximal gradient algorithm with extrapolation (BPGe) as an acceleration version of BPG.
Takahashi \etal~\cite{Takahashi2022-ml} proposed the Bregman proximal algorithms for a difference of convex functions (DC) optimization problem.
Takahashi and Takeda~\cite{Takahashi2024-ej} proposed the approximate Bregman proximal gradient algorithm exploiting an approximate measurement of the Bregman distance.
Mukkamala and Ochs~\cite{Mukkamala2019-mk} proposed a variant of BPG for NMF with the Frobenius loss.
Wang \etal~\cite{Wang2024-zl} proposed a stochastic variant of BPG and applied it to NMF with the Frobenius loss.
These Bregman proximal algorithms have been applied to many other applications, such as phase retrieval~\cite{Bolte2018-zt,Takahashi2022-ml,Zhang2019-ia}, low-rank minimization~\cite{Dragomir2021-rv}, and blind deconvolution~\cite{Takahashi2023-uh}.

In this paper, we propose new Bregman proximal gradient algorithms, called the majorization-minimization Bregman proximal gradient algorithm (MMBPG) and the MMBPG with extrapolation (MMBPGe), which can be regarded as a BPG(e) combined with the majorization-minimization algorithm.
These methods update $\W$ and $\H$ not alternately but simultaneously.
Inspired by~\cite{Lee2000-nl}, MMBPG and MMBPGe minimize an auxiliary function majorizing the loss function, instead of the objective function.
In addition, we establish the monotonically decreasing property of the objective function for MMBPG and that of the potential function for MMBPGe (see Lemmas~\ref{lemma:decreasing-bpg} and~\ref{lemma:decreasing-bpge}, respectively).
We also establish global convergence, \ie, a sequence of MMBPG(e) converges to a stationary point from any initial point (see Theorems~\ref{theorem:covergence-mmbpg} and~\ref{theorem:convergence-mmbpge}).
To apply MMBPG(e) to KL-NMF, we define a separable Bregman distance that satisfies the smooth adaptable property (see Theorem~\ref{theorem:l-smad}), which makes the subproblem of MMBPG(e) solvable in closed form.
We provide closed-form solutions for the subproblem of MMBPG(e) when using $\ell_1$ regularization and squared Frobenius norm regularization (see Propositions~\ref{prop:closed-form-nmf-l1} and~\ref{prop:closed-form-nmf-fro}).
Using these techniques, MMBPG(e) can update all variables simultaneously.
In numerical experiments, we have compared proposed algorithms with existing algorithms on synthetic data and real-world data.

The rest of the paper is organized as follows.
Section~\ref{sec:preliminaries} summarizes important properties of KL-NMF and notation such as the Bregman distance, the smooth adaptable property, and the Kurdyka--\L{}ojasiewicz property.
These tools are used to establish the theoretical aspect of our algorithms.
Section~\ref{sec:proposed-method} proposes MMBPG and MMBPGe, and establishes their monotonically decreasing properties.
Assuming the Kurdyka--\L{}ojasiewicz property, we guarantee global convergence to a stationary point.
Section~\ref{sec:mmbpgs-kl-nmf} shows MMBPG and MMBPGe for KL-NMF.
It also provides a separable Bregman distance and closed-form solutions of their subproblems without heavy computation.
Moreover, we prove that MMBPG(e) globally converges to a KKT point.
Section~\ref{sec:numerical-experiments} shows numerical experiments on synthetic data and real-world data sets to compare MMBPG(e) with other existing algorithms.
Section~\ref{sec:conclusion} offers some conclusions.

\paragraph*{Notation}\quad
Vectors and matrices are shown in boldface.
The Frobenius inner product of $\X, \Y\in\R^{m \times n}$ is defined by $\langle\X, \Y\rangle := \tr(\X^{\mathsf{T}}\Y)$ and the Frobenius norm of $\X$ is defined by $\|\X\|_F := \sqrt{\sum_{i,j}X_{ij}^2}$.

In what follows, we use the following notation.
Let $[n] := \{1,\ldots, n\}$ for $0 < n\in\N$ be a subset of $\N$.
Let $\interior{C}$ and $\cl{C}$ be the interior and the closure of a set $C\subset\R^{m \times n}$, respectively.
For an extended-real-valued function $f:\R^{m \times n}\to[-\infty, +\infty]$, the effective domain of $f$ is defined by $\{\X\in\R^{m \times n} \mid f(\X) < +\infty\}$ and denoted by $\dom f$.
The function $f$ is said to be proper if $f(\X) > -\infty$ for all $\X\in\R^{m \times n}$ and $\dom f \neq \emptyset$.
For a proper and lower semicontinuous function $f:\R^{m \times n}\to(-\infty, +\infty]$, we also define $\dom\partial f:= \{\X\in\R^{m \times n}\mid\partial f(\X)\neq\emptyset\}$, where $\partial f(\X)$ is the limiting differential (also called the general subdifferential or the Mordukhovich subdifferential)~\cite[Definition 8.3]{Rockafellar1997-zb}.
When $f$ is locally lower semicontinuous at $\bar{\x}$ with $\bar{\x}$ finite and strictly continuous, $\partial f(\bar{\x})$ is nonempty and compact~\cite[Theorem 9.13]{Rockafellar1997-zb}.

\section{Preliminaries}\label{sec:preliminaries}
\subsection{Properties of KL-NMF}\label{subsec:properties-kl-nmf}
We introduce properties of KL-NMF~\eqref{prob:klnmf}.
Throughout this section, we assume $g\equiv0$. 
\begin{proposition}[{\cite[Proposition 2.1]{Finesso2006-hl}}]
    \label{prop:global-solution}
    Problem~\eqref{prob:klnmf} has a global optimal solution.
\end{proposition}
Proposition~\ref{prop:global-solution} for KL-NMF constrained by $\W \geq \epsilon$ and $\H \geq \epsilon$ with $\epsilon \geq 0$ was established in~\cite[Proposition 1]{Hien2021-sh}.
Problem~\eqref{prob:klnmf} has infinitely many solutions.
Let $(\W^*, \H^*)\in\R_+^{m \times r}\times\R_+^{r \times n}$ be a global optimal solution of~\eqref{prob:klnmf} and let $\Q\in\R^{r \times r}$ be any regular matrix such that $(\W^*\Q, \Q^{-1}\H^*)\in\R_+^{m \times r}\times\R_+^{r \times n}$.
The pair $(\W^*\Q, \Q^{-1}\H^*)$ is also a global optimal solution because of $D(\X, \W^*\H^*) = D(\X, \W^*\Q\Q^{-1}\H^*)$.
Problem~\eqref{prob:klnmf} has the nonuniqueness of global optimal solutions.

We define a \emph{Karush--Kuhn--Tucker (KKT) point} $(\W^*, \H^*)$ of~\eqref{prob:klnmf}.
If the objective function $f(\W, \H)$ is differentiable at $(\W^*, \H^*)$, the KKT condition is given by
\begin{align}
    W_{il}^*&\geq 0,\quad \frac{\partial f(\W^*,\H^*)}{\partial W_{il}} \geq 0,\quad
    \frac{\partial f(\W^*,\H^*)}{\partial W_{il}}W_{il}^* = 0,\quad \forall i\in[m],l\in[r],\label{cond:kkt_w}\\
    H_{lj}^*&\geq 0,\quad \frac{\partial f(\W^*,\H^*)}{\partial H_{lj}} \geq 0,\quad
    \frac{\partial f(\W^*,\H^*)}{\partial H_{lj}}H_{lj}^* = 0,\quad \forall l\in[r],j\in[n],\label{cond:kkt_h}
\end{align}
where
\begin{align*}
    \frac{\partial f(\W^*,\H^*)}{\partial W_{il}} &= \sum_{j=1}^n H_{lj}^* - \sum_{j=1}^n \frac{X_{ij}H_{lj}^*}{(\W^*\H^*)_{ij}},\\
    \frac{\partial f(\W^*,\H^*)}{\partial H_{lj}} &= \sum_{i=1}^m W_{il}^* - \sum_{i=1}^m \frac{X_{ij}W_{il}^*}{(\W^*\H^*)_{ij}}.
\end{align*}
Let $\delta_C$ be the indicator function defined by $\delta_C(\x) = 0$ for $\x\in C$ and $\delta_C(\x) = +\infty$ otherwise.
Problem~\eqref{prob:klnmf} is equivalent to minimizing $f + \delta_{\R_+^{m \times r} \times \R_+^{r \times n}}$.
KKT points of~\eqref{prob:klnmf} coincide with its stationary points satisfying
\begin{align}
     \frac{\partial f(\W^*,\H^*)}{\partial W_{il}}(W_{il} - W_{il}^*)&\geq 0, \quad \forall W_{il} \geq 0,\quad \forall i\in[m],l\in[r],\label{cond:stationary_w}\\
     \frac{\partial f(\W^*,\H^*)}{\partial H_{lj}}(H_{lj} - H_{lj}^*)&\geq 0, \quad \forall H_{lj} \geq 0,\quad \forall l\in[r],j\in[n].\label{cond:stationary_h}
\end{align}
KKT conditions~\eqref{cond:kkt_w} and~\eqref{cond:kkt_h} hold if and only if~\eqref{cond:stationary_w} and~\eqref{cond:stationary_h} hold, respectively. See also~\cite[Section 2.2]{Hien2021-sh}.

\subsection{Bregman Distances}\label{subsec:bregman-distance}
Let $C$ be a nonempty open convex subset of $\R^{m \times n}$.
We introduce the kernel generating distance and the Bregman distance.
\begin{definition}[Kernel generating distance~{\cite[Definition 2.1]{Bolte2018-zt}}]
    A function $\phi:\R^{m \times n}\to(-\infty,+\infty]$ is called a \emph{kernel generating distance} associated with $C$ if it satisfies the following conditions:
    \begin{enumerate}
        \item $\phi$ is proper, lower semicontinuous, and convex, with $\dom\phi\subset\cl{C}$ and $\dom\partial\phi = C$.
        \item $\phi$ is $\mathcal{C}^1$ on $\interior\dom\phi \equiv C$.
    \end{enumerate}
    We denote the class of kernel generating distances associated with $C$ by $\mathcal{G}(C)$.
\end{definition}
\begin{definition}[Bregman distance~\cite{Bregman1967-rf}]
    Given a kernel generating distance $\phi\in\mathcal{G}(C)$, a \emph{Bregman distance} $D_\phi:\dom\phi\times\interior\dom\phi\to\R_+$ associated with $\phi$ is defined by
    \begin{align*}
        D_\phi(\X,\Y) = \phi(\X) - \phi(\Y) - \langle\nabla\phi(\Y),\X-\Y\rangle.
    \end{align*}
\end{definition}
Note that the Bregman distance does not satisfy the symmetry and the triangle inequality in general.
Nevertheless, the Bregman distance behaves like a measure of similarity.
For example, $D_\phi$ is always nonnegative due to the convexity of $\phi$.
In addition, when $\phi$ is strictly convex, $D_\phi(\X,\Y) = 0$ holds if and only if $\X = \Y$.

The KL divergence is an example of a Bregman distance, \ie, $D(\X,\Y) = D_{\phi}(\X, \Y)$ by using $\phi(\X) = \sum_{i=1}^n\sum_{j=1}^m X_{ij}\log X_{ij}$ and $\dom\phi = \R_+^{m \times n}$.
Setting $\phi(\X) = \frac{1}{2}\|\X\|_F^2$, $D_\phi(\X,\Y) = \frac{1}{2}\|\X - \Y\|_F^2$.
See~\cite{Bauschke2017-hg,Bauschke1997-vk,Lu2018-hx} and~\cite[Table 2.1]{Dhillon2008-zz} for more examples.

The Bregman distance satisfies the three-point identity~\cite{Bolte2018-zt}, for any $\X\in\dom\phi$ and $\Y,\Z\in\interior\dom\phi$,
\begin{align}
    D_\phi(\X, \Z) - D_\phi(\X, \Y) - D_\phi(\Y, \Z) = \langle\nabla\phi(\Y) - \nabla\phi(\Z),\X-\Y\rangle.\label{eq:three-point}
\end{align}

\subsection{Smooth Adaptable Property}
The smooth adaptable property~\cite{Bolte2018-zt} is a key property for Bregman-type algorithms.
\begin{definition}[$L$-smooth adaptable property~{\cite[Definition 2.2]{Bolte2018-zt}}]\label{def:l-smad}
    Consider a pair of functions $(f, \phi)$ satisfying the following conditions:
    \begin{enumerate}
        \item $\phi\in\mathcal{G}(C)$,
        \item $f:\R^{m \times n}\to(-\infty,+\infty]$ is proper and lower semicontinuous with $\dom\phi\subset\dom f$, which is $\mathcal{C}^1$ on $C = \interior\dom\phi$.
    \end{enumerate}
    The pair $(f,\phi)$ is called $L$-\emph{smooth adaptable} ($L$-\emph{smad}) on $C$ if there exists $L > 0$ such that $L\phi - f$ and $L\phi + f$ are convex on $C$.
\end{definition}

The smooth adaptable property leads to majorizing $f$ by a kernel generating distance $\phi$ and a parameter $L > 0$.
Lu~\etal~\cite{Lu2018-hx} also defined similar properties called relative smoothness and relative strong convexity.

The smooth adaptable property provides the extended descent lemma.
\begin{lemma}[Extended descent lemma~{\cite[Lemma 2.1]{Bolte2018-zt}}]
    \label{lemma:extended-descent}
    The pair of functions $(f, \phi)$ is $L$-smad on $C$ if and only if
    \begin{align*}
        |f(\X) - f(\Y) - \langle\nabla f(\Y),\X-\Y\rangle|\leq LD_\phi(\X,\Y), \quad \forall\X,\Y\in\interior\dom\phi.
    \end{align*}
\end{lemma}
Setting $\phi(\X)=\frac{1}{2}\|\X\|_F^2$, Lemma~\ref{lemma:extended-descent} reduces to the classical descent lemma.
\subsection{Kurdyka--\L{}ojasiewicz Property}
The Kurdyka--\L{}ojasiewicz property is used for establishing global convergence.
Attouch~\etal~\cite{Attouch2010-nr} extended the \L{}ojasiewicz gradient inequality~\cite{Kurdyka1998-zd,Lojasiewicz1963-kf} to nonsmooth functions.

Given $\eta>0$, let $\Xi_{\eta}$ denote the set of all continuous concave functions $\psi:[0, \eta) \to \R_+$ that are $\mathcal{C}^1$ on $(0, \eta)$ with positive derivatives and which satisfy $\psi(0) = 0$.
Here, we define the Kurdyka--\L{}ojasiewicz property.
\begin{definition}[Kurdyka--\L{}ojasiewicz property~{\cite[Definition 7]{Attouch2010-nr}}]
    \label{def:kl}
    Let $f: \R^n \to(-\infty,+\infty]$ be a proper and lower semicontinuous function.
    \begin{enumerate}
        \item $f$ is said to have the \emph{Kurdyka--\L{}ojasiewicz property} at $\hat{\x}\in\dom\partial f$ if there exist $\eta\in(0, +\infty]$, a neighborhood $U$ of $\hat{\x}$, and a function $\psi\in\Xi_{\eta}$ such that, for all
        \begin{align*}
            \x\in U \cap\{\x\in\R^n \mid f(\hat{\x})<f(\x)<f(\hat{\x})+\eta\},
        \end{align*}
        the following inequality holds:
        \begin{align*}
            \psi'(f(\x) - f(\hat{\x})) \dist(\bm{0}, \partial f(\x)) \geq 1.
        \end{align*}
        \item If $f$ has the Kurdyka--\L{}ojasiewicz property at each point of $\dom\partial f$, then it is called a \emph{Kurdyka--\L{}ojasiewicz function}.
    \end{enumerate}
\end{definition}
Let $\psi(s) = cs^{1-\theta}$ for some $\theta\in[0,1)$ and $c > 0$.
The parameter $\theta$ is called the Kurdyka--\L{}ojasiewicz exponent, which affects convergence rates.
Li and Pong~\cite{Li2018-mf} showed calculus rules of the Kurdyka--\L{}ojasiewicz exponent.

\section{Proposed Methods: Majorization-Minimization Bregman Proximal Gradient Algorithms}\label{sec:proposed-method}
In this section, we consider the following optimization problem that includes~\eqref{prob:klnmf} as a special case:
\begin{align}
    \min_{\x\in \cl{C}} \quad \Psi(\x) &:= f(\x) + g(\x), \label{prob:general}
\end{align}
where $C\subset\R^n$ is a nonempty open convex set.
We make the following assumptions for~\eqref{prob:general}.
\begin{assumption}\label{assumption:setting}\quad
    \begin{enumerate}
        \item $\phi\in\mathcal{G}(C)$ with $\cl{C} = \cl{\dom\phi}$ on $C = \interior\dom\phi$.
        \item $f:\R^n \to (-\infty,+\infty]$ is a proper and lower semicontinuous function with $\dom \phi \subset \dom f$, which is $\C^1$ on $C$.
        \item $g:\R^n \to (-\infty, +\infty]$ is a proper and lower semicontinuous function with $\dom g \cap C \neq \emptyset$.
        \item $v := \inf_{\x\in\cl{C}}\Psi(\x) > -\infty$.
        \item The function $\phi + \lambda g$ is supercoercive for any $\lambda > 0$, \ie,
        \begin{align*}
            \lim_{\|\x\|\to+\infty}\frac{\phi(\x) + \lambda g(\x)}{\|\x\|} = +\infty.
        \end{align*}
    \end{enumerate}
\end{assumption}
Assumption~\ref{assumption:setting} is standard for Bregman-type algorithms~\cite[Assumptions A, B, and C]{Bolte2018-zt},~\cite[Assumptions 1 and 2]{Takahashi2022-ml}, and~\cite[Assumptions 9 and 11]{Takahashi2024-ej}.

\subsection{Majorization-Minimization Bregman Proximal Gradient Algorithm}
Establishing the smooth adaptable property of $(f, \phi)$ is often difficult.
In order to alleviate this fact, we introduce an auxiliary function.
A function $\hat{f}:\R^n \to (-\infty,+\infty]$ is called an \emph{auxiliary function} of $f$ at $\y \in C$ if it satisfies the following conditions:
\begin{enumerate}
    \item $\hat{f}$ is proper and lower semicontinuous with $\dom \phi \subset \dom \hat{f}$, which is $\C^1$ on $C$.
    \item $f(\x)\leq \hat{f}(\x)$ for all $\x \in C$.
    \item $f(\y) = \hat{f}(\y)$ and $\nabla f(\y) = \nabla\hat{f}(\y)$.
\end{enumerate}

The \emph{Bregman proximal gradient mapping} at $\x \in C$ is defined by
\begin{align*}
    \mathcal{T}_{\lambda}(\x) := \argmin_{\u\in\cl{C}}\left\{\langle\nabla f(\x), \u - \x\rangle + g(\u) + \frac{1}{\lambda} D_\phi(\u, \x)\right\},
\end{align*}
where $\lambda > 0$.
Bolte~\etal~\cite{Bolte2018-zt} proposed the Bregman proximal gradient algorithm (BPG) by using the Bregman proximal gradient mapping $\mathcal{T}_{\lambda}(\x)$.
The iteration of BPG is given by $\x^{k+1} = \mathcal{T}_{\lambda}(\x^k)$ for $\x^k \in C$.
When $\phi = \frac{1}{2}\|\cdot\|^2$, $\mathcal{T}_{\lambda}(\x)$ is equivalent to a proximal gradient mapping.
We make the following assumption for $\mathcal{T}_{\lambda}(\x)$.
\begin{assumption}\label{assumption:bregman-mapping-well-defined}
    For all $\x \in C$ and $\lambda > 0$, it holds that $\mathcal{T}_{\lambda}(\x)\subset C$.
\end{assumption}
Assumption~\ref{assumption:bregman-mapping-well-defined} guarantees that a sequence generated by $\mathcal{T}_{\lambda}$ belongs to $C$.
See, for more discussion about Assumption~\ref{assumption:bregman-mapping-well-defined},~\cite{Bolte2018-zt}.
We have the well-posedness of $\mathcal{T}_\lambda$. Its proof is omitted because it is the same as~\cite[Lemma 3.1]{Bolte2018-zt}.
\begin{lemma}[Well-posedness of $\mathcal{T}_\lambda$]
    Suppose that Assumptions~{\rm\ref{assumption:setting}} and~{\rm\ref{assumption:bregman-mapping-well-defined}} hold, and let $\x \in \interior \dom \phi$. Then, for any $\lambda > 0$, the set $\mathcal{T}_\lambda(\x)$ is a nonempty and compact subset of $C$.
\end{lemma}

Given an auxiliary function $\hat{f}$, we propose the majorization-minimization Bregman proximal gradient algorithm (MMBPG) to minimize the auxiliary function $\hat{f}_k:\R^n \to (-\infty,+\infty]$ at the $k$th iteration, which can be constructed from $\hat{f}$.
Now we are ready to describe our algorithm for solving~\eqref{prob:general} as Algorithm~\ref{alg:mmbpg}.

\begin{algorithm}[!t]
    \caption{Majorization-minimization Bregman proximal gradient algorithm (MMBPG)}
    \label{alg:mmbpg}
    \Input{$\phi\in\mathcal{G}(C)$ such that the smooth adaptable property for the pair $(\hat{f}_k,\phi)$ holds on $C$.
    An initial point $\x^0\in\interior\dom\phi$.
    }
    \For{$k = 0, 1, 2, \ldots,$}{
        Take $\lambda_k > 0$ and compute
        \begin{align}
            \x^{k+1}
            = \argmin_{\x\in\cl{C}}&
            \left\{\langle\nabla \hat{f}_k(\x^k), \x - \x^k\rangle + g(\x) + \frac{1}{\lambda_k} D_\phi(\x,\x^k)\right\}.\label{subprob:mmbpg}
        \end{align}
    }
\end{algorithm}

Because MMBPG minimizes the auxiliary function with a regularization term $g$, not the objective function $\Psi$, we need to show that a sequence generated by MMBPG decreases the objective function value.
\begin{assumption}\label{assumption:l-smad}
    Let $\hat{f}:\R^n \to (-\infty,+\infty]$ be an auxiliary function of $f$.
    The pair $(\hat{f}, \phi)$ is $L$-smad on $C$.
\end{assumption}
Assumption~\ref{assumption:l-smad} implies that $(\hat{f}_k, \phi)$ is $L_k$-smad for $L_k \leq L$ for any $k \geq 0$.
In Theorem~\ref{theorem:l-smad}, for KL-NMF~\eqref{prob:klnmf}, we will show that there exist $L_k$, $L$, and $\phi$ such that Assumption~\ref{assumption:l-smad} holds.
Now, we establish the monotonically decreasing property of MMBPG for~\eqref{prob:general}.
\begin{lemma}
\label{lemma:decreasing-bpg}
Suppose that Assumptions{\rm~\ref{assumption:setting},~\ref{assumption:bregman-mapping-well-defined}}, and~{\rm\ref{assumption:l-smad}} hold.
For any $\x^k\in C$, let $\x^{k+1}\in C$ be defined by~\eqref{subprob:mmbpg}.
It holds that
\begin{align}
    \lambda_k\Psi(\x^{k+1}) \leq \lambda_k\Psi(\x^k) - (1 - \lambda_k L_k)D_\phi(\x^{k+1}, \x^k).\label{ineq:decreasing-property}
\end{align}
In addition, when $0 < \lambda_k L_k < 1$ for all $k \geq 0$, the objective function value $\Psi$ is ensured to be decreasing.
\end{lemma}
\begin{proof}
From the global optimality of~\eqref{subprob:mmbpg}, we have
\begin{align*}
    \langle\nabla \hat{f}_k(\x^k), \x^{k+1} - \x^k\rangle + g(\x^{k+1}) + \frac{1}{\lambda_k} D_\phi(\x^{k+1},\x^k)\leq g(\x^k).
\end{align*}
Using Lemma~\ref{lemma:extended-descent} and the above inequality, we obtain
\begin{align}
    \hat{f}_k(\x^{k+1}) + g(\x^{k+1}) &\leq \hat{f}_k(\x^k) + \langle\nabla \hat{f}_k(\x^k), \x^{k+1} - \x^k\rangle\nonumber\\
    &\quad+ g(\x^{k+1}) + L_kD_\phi(\x^{k+1},\x^k)\nonumber\\
    &\leq \hat{f}_k(\x^k) + g(\x^k) + \left(L_k - \frac{1}{\lambda_k}\right)D_\phi(\x^{k+1},\x^k)\nonumber\\
    &= f(\x^k) + g(\x^k) - \left(\frac{1}{\lambda_k} - L_k \right)D_\phi(\x^{k+1},\x^k),\label{ineq:decreasing-hff}
\end{align}
where the last equation holds due to $\hat{f}_k(\x^k) = f(\x^k)$.
Recalling the relationship between $f$ and $\hat{f}$, it holds that $f(\x) \leq \hat{f}_k(\x)$ for any $\x\in C$.
Therefore, from~\eqref{ineq:decreasing-hff}, we obtain
\begin{align*}
    \Psi(\x^{k+1}) \leq \Psi(\x^k) - \left(\frac{1}{\lambda_k} - L_k \right)D_\phi(\x^{k+1},\x^k),
\end{align*}
which implies~\eqref{ineq:decreasing-property}.
The last statement follows $0 < \lambda_k L_k < 1$ for all $k \geq 0$.
\qed
\end{proof}

Since we have the monotonically decreasing property, we will show that MMBPG converges to a stationary point of~\eqref{prob:general}.
Inspired by Fermat's rule \cite[Theorem 10.1]{Rockafellar1997-zb}, using the limiting subdifferential, we define a stationary point, also called a critical point~\cite{Bolte2018-zt}.
\begin{definition}
    We say that $\tilde{\x} \in\R^n$ is a \emph{stationary point} of~\eqref{prob:general} if $\zero\in\partial\Psi(\tilde{\x})$.
\end{definition}
Note that KKT points of~\eqref{prob:klnmf} are stationary points of $\min_{\cl{C}} D(\X, \W\H) + g(\W, \H)$ with $\cl{C} = \R_+^{m \times r}\times\R_+^{r \times n}$ (see also Section~\ref{subsec:properties-kl-nmf}).
To establish global convergence, we make additional assumptions as follows.
\begin{assumption}\label{assumption:global-convergence}\quad
    \begin{enumerate}[label=\normalfont(\roman*)]
        \item $\phi$ is $\sigma$-strongly convex.
        \item $\nabla\hat{f}_k$ and $\nabla\phi$ are Lipschitz continuous on any compact subset of $C$.\label{assumption:lip-compact}
        \item The objective function $\Psi$ is level-bounded; \ie, for any $r \in\R$, the level sets $\{\x\in C \mid \Psi(\x)\leq r\}$ are bounded.\label{assumption:level-bounded}
    \end{enumerate}
\end{assumption}
Assumption~\ref{assumption:global-convergence} is also standard for Bregman-type algorithms~\cite[Assumption D]{Bolte2018-zt},~\cite[Assumption 4]{Takahashi2022-ml}, and~\cite[Assumption 19]{Takahashi2024-ej}.
Now, we establish global subsequential convergence and global convergence.

\begin{theorem}[Convergence of MMBPG]\label{theorem:covergence-mmbpg}
    Suppose that Assumptions~{\rm\ref{assumption:setting},~\ref{assumption:bregman-mapping-well-defined},~\ref{assumption:l-smad}}, and~{\rm\ref{assumption:global-convergence}} hold.
    Let $\{\x^k\}_k$ be a sequence generated by MMBPG with $0 < \lambda_k L_k < 1$ for solving~\eqref{prob:general}.
    Then, the following statements hold.
    \begin{enumerate}
        \item Global subsequential convergence: Any accumulation point of $\{\x^k\}_k$ is a stationary point of~\eqref{prob:general}.
        \item Global convergence: Suppose that $\Psi$ is a Kurdyka--\L{}ojasiewicz function.
        Then, the sequence $\{\x^k\}_k$ has finite length, \ie, $\sum_{k=1}^{\infty}\|\x^{k+1} - \x^k\| < \infty$, and converges to a stationary point of~\eqref{prob:general}.
    \end{enumerate}
\end{theorem}
\begin{proof}
    (i) From Lemma~\ref{lemma:decreasing-bpg}, we have $\Psi(\x^k) \leq \Psi(\x^0)$ for all $k \in \N$.
    The objective function $\Psi$ is also level-bounded due to Assumption~\ref{assumption:global-convergence}\ref{assumption:level-bounded}.
    Therefore, $\{\x^k\}_k$ is bounded.

    Using~\eqref{ineq:decreasing-property}, we have
    \begin{align}
        \Psi(\x^k) - \Psi(\x^{k+1}) &\geq\left(\frac{1}{\lambda_k} - L_k\right)D_\phi(\x^{k+1}, \x^k)\nonumber\\
        &\geq \left(\frac{1}{\lambda_k} - L_k\right)\frac{\sigma}{2}\|\x^{k+1} - \x^k\|^2,\label{ineq:decreasing-use-strongly-convex}
    \end{align}
    where the last inequality holds because of Assumption~\ref{assumption:global-convergence}(i).
    Note that there exists $c > 0$ such that $\frac{1}{\lambda_k} - L_k \geq c$ for all $k\in\N$.
    For example, we can choose $\lambda_k = \frac{1}{c + L_k} \geq \frac{1}{c + L}$ because $L_k \leq L$ holds from Assumption~\ref{assumption:l-smad}. We find that $\lambda_kL_k = \frac{L_k}{c + L_k} \leq \frac{L}{c + L} < 1$ holds because $L_k \leq L$ holds and $\frac{x}{c + x}$ is increasing over $x \geq 0$.
    Summing~\eqref{ineq:decreasing-use-strongly-convex} from $k = 0$ to $\infty$ and using Assumption~\ref{assumption:setting}(iii), we have
    \begin{align*}
        \sum_{k = 1}^\infty\frac{\sigma c}{2}\|\x^{k+1} - \x^k\|^2\leq\Psi(\x^0) - \liminf_{n\to\infty}\Psi(\x^n) \leq\Psi(\x^0) -  v < \infty,
    \end{align*}
    which shows that $\lim_{k\to\infty}\|\x^{k+1} - \x^k\| = 0$.

    Next, we consider the first-order optimality condition of~\eqref{subprob:mmbpg}, that is,
    \begin{align*}
        \zero\in\nabla \hat{f}_k(\x^k) + \partial g(\x^{k+1}) + \frac{1}{\lambda_k}(\nabla\phi(\x^{k+1}) - \nabla\phi(\x^k)).
    \end{align*}
    We have
    \begin{align}
        \frac{1}{\lambda_k}(\nabla\phi(\x^k) - \nabla\phi(\x^{k+1}))\in\nabla f(\x^k) + \partial g(\x^{k+1}).\label{eq:first-order-opt-phi}
    \end{align}
    Because $\{\x^k\}_k$ is bounded, there exist an accumulation point $\tilde{\x}$ of $\{\x^k\}_k$ and a subsequence $\{\x^{k_j}\}_{j}$ such that $\lim_{j\to\infty}\x^{k_j} = \tilde{\x}$.
    Substituting $k_j$ for $k$ in~\eqref{eq:first-order-opt-phi}, we have $\frac{1}{\lambda_k}(\nabla\phi(\x^{k_j}) - \nabla\phi(\x^{k_j+1})) \to \zero$ because there exists $A_0 > 0$ such that $\frac{1}{\lambda_{k_j}}\|\nabla\phi(\x^{k_j}) - \nabla\phi(\x^{k_j+1})\| \leq A_0(c + L)\|\x^{k_j} - \x^{k_j+1}\|\to0$ from $\lim_{j\to\infty}\|\x^{k_j+1} - \x^{k_j}\| = 0$, $\frac{1}{\lambda_{k_j}} \leq c + L$, and Assumption~\ref{assumption:global-convergence}\ref{assumption:lip-compact} ($\{\x^{k_j}\}_{j} \cup \{\tilde{x}\}$ is compact).
    Because of Assumption~\ref{assumption:setting} and $\lim_{j\to\infty}\|\x^{k_j+1} - \x^{k_j}\| = 0$, we have $\zero\in\nabla f(\tilde{\x}) + \partial g(\tilde{\x})$, which shows $\tilde{\x}$ is a stationary point of~\eqref{prob:general}.

    (ii) In the same way as~\cite[Theorems 4.1 and 6.2]{Bolte2018-zt} and~\cite[Theorem 2]{Takahashi2022-ml}, we can establish global convergence.
    There exists $k_0 > 0$ such that $\x^k$ is sufficiently close to the set of accumulation points due to (i) (see also the proofs of~\cite[Theorem 6.2]{Bolte2018-zt} and~\cite[Theorem 2]{Takahashi2022-ml}).
    From the first-order optimality condition of~\eqref{subprob:mmbpg} and $\nabla f(\x^k) = \nabla \hat{f}_k(\x^k)$, we have
    \begin{align*}
        \nabla f(\x^{k+1}) - \nabla f(\x^k) + \frac{1}{\lambda_k}(\nabla\phi(\x^k) - \nabla\phi(\x^{k+1}))\in\nabla f(\x^{k+1}) + \partial g(\x^{k+1}),
    \end{align*}
    which implies $\dist(\zero, \nabla f(\x^k) + \partial g(\x^k))\leq A_1\|\x^k - \x^{k-1}\|$ for $k > k_0$.
    Therefore, $\lim_{k\to\infty}\dist(\zero, \nabla f(\x^k) + \partial g(\x^k)) = 0$ due to $\lim_{k\to\infty}\|\x^k - \x^{k-1}\| = 0$.
    The rest of the proof is the same to~\cite[Theorems 4.1 and 6.2]{Bolte2018-zt} and~\cite[Theorem 2]{Takahashi2022-ml}.
    \qed
\end{proof}

\subsection{Majorization-Minimization Bregman Proximal Gradient Algorithm with Extrapolation} 

BPG is known to be slow in applications.
Zhang \etal~\cite{Zhang2019-ia} proposed the Bregman proximal gradient algorithm with extrapolation (BPGe).
We use the momentum parameter $\beta_k$ defined by~\cite{Nesterov1983-xn} and the restart condition proposed in~\cite{Takahashi2022-ml}.
Takahashi~\etal~\cite{Takahashi2022-ml} proposed the adaptive restart scheme, \ie, it resets $\theta_k$ defined in line~\ref{alg:theta-beta} of Algorithm~\ref{alg:mmbpge} whenever it holds that
\begin{align}
    D_\phi(\x^k, \y^k) > \rho D_\phi(\x^{k-1}, \x^k).\label{ineq:adaptive-restart}
\end{align}
This inequality requires less computation than the line search procedure proposed by~\cite{Zhang2019-ia}.
Now we describe the majorization-minimization Bregman proximal gradient algorithm with extrapolation (MMBPGe) for solving~\eqref{prob:klnmf} as Algorithm~\ref{alg:mmbpge}.

\begin{algorithm}[!t]
    \caption{Majorization-minimization Bregman proximal gradient algorithm with extrapolation (MMBPGe)}
    \label{alg:mmbpge}
    \Input{$\phi\in\mathcal{G}(C)$ such that the smooth adaptable property for the pair $(\hat{f}_k,\phi)$ holds on $C$.
    An initial point $\x^0 = \x^{-1}\in\interior\dom\phi$.
    $\theta_0 = \theta_{-1} = 1, \rho\in(0,1]$
    }
    \For{$k = 0, 1, 2, \ldots,$}{
        Take $\lambda_k >0$ and compute \label{alg:theta-beta}
        \begin{align*}
            \beta_k &= \frac{\theta_{k-1} - 1}{\theta_k} \quad \text{with} \quad \theta_k = \frac{1 + \sqrt{1 + 4\theta_{k-1}^2}}{2},\\
            \y^k &= \x^k + \beta_k(\x^k - \x^{k-1}).
        \end{align*}
        
        \If{$\y^k\notin C$ or $D_\phi(\x^k, \y^k) > \rho D_\phi(\x^{k-1}, \x^k)$}{
            Set $\y^k = \x^k$ and $\theta_k = \theta_{k-1} = 1$.
        }
        Compute
        \begin{align}
            \x^{k+1}
            = \argmin_{\y\in\cl{C}}&
            \left\{\langle\nabla \hat{f}_k(\y^k), \y - \y^k\rangle + g(\y) + \frac{1}{\lambda_k} D_\phi(\y,\y^k)\right\}.\label{subprob:mmbpge}
        \end{align}
    }
\end{algorithm}

It is difficult to show the monotonically decreasing property of the objective function by algorithms with extrapolation.
Instead, we establish the monotonically decreasing property of the potential function value by MMBPGe.
Let the potential function $\Phi_M$ be defined by
\begin{align*}
    \Phi_M(\x) := \Psi(\x) + MD_\phi(\y, \x), \quad M > 0.
\end{align*}
We also assume that $g$ is convex.
\begin{assumption}
    The function $g$ is convex.\label{assumption:g-convex}
\end{assumption}
We obtain the monotonically decreasing property of the potential function value as follows.
\begin{lemma}
\label{lemma:decreasing-bpge}
Suppose that Assumptions~{\rm\ref{assumption:setting},~\ref{assumption:bregman-mapping-well-defined},~\ref{assumption:l-smad}}, and~{\rm\ref{assumption:g-convex}} hold.
For any $\x^k\in C$ and $\y^k\in C$, let $\y^k := \x^k + \beta_k(\x^k - \x^{k-1})\in C$ with $\{\beta_k\}_k\subset[0,1)$ and $\x^{k+1}\in C$ be defined by~\eqref{subprob:mmbpge}.
Then, it holds that
\begin{align}
    \lambda_k\Psi(\x^{k+1}) &\leq \lambda_k\Psi(\x^k) + \left(1 + \lambda_k L_k\right)D_\phi(\x^k,\y^k)\nonumber\\
    &\quad- (1 - \lambda_k L_k)D_\phi(\x^{k+1}, \y^k) - D_\phi(\x^k, \x^{k+1}).\label{ineq:decreasing-property-bpge-obj}
\end{align}
Furthermore, when $\{\beta_k\}_k\subset[0,1)$ is given by the adaptive restart scheme~\eqref{ineq:adaptive-restart}, it holds that
\begin{align}
    \lambda_k\Phi_M(\x^{k+1}) &\leq \lambda_k\Phi_M(\x^k)- (\lambda_k M - \rho\left(1 + \lambda_kL_k\right))D_\phi(\x^{k-1},\x^k)\nonumber\\
    &\quad-\left(1 - \lambda_k L_k\right)D_\phi(\x^{k+1},\y^k) - \left(1 - \lambda_k M\right) D_\phi(\x^k,\x^{k+1}).\label{ineq:decreasing-property-bpge-lyap}
\end{align}
In addition, when $0 < \lambda_k L_k < 1$ and $\rho(1 + \lambda_k L_k) < \lambda_k M < 1$ for all $k \geq 0$, the potential function value $\Phi_M$ is ensured to be decreasing.
\end{lemma}
\begin{proof}
From the first-order optimality condition of~\eqref{subprob:mmbpge}, we have
\begin{align*}
    \zero\in\nabla \hat{f}_k(\y^k) + \partial g(\x^{k+1}) + \frac{1}{\lambda_k}\left(\nabla\phi(\x^{k+1}) - \nabla\phi(\y^k)\right).
\end{align*}
Using the above and the convexity of $g$, we obtain
\begin{align}
    g(\x^k) - g(\x^{k+1}) \geq -\left\langle\nabla \hat{f}_k(\y^k) + \frac{1}{\lambda_k}\left(\nabla\phi(\x^{k+1}) - \nabla\phi(\y^k)\right), \x^k - \x^{k+1}\right\rangle.\label{ineq:decreasing-g}
\end{align}
Applying the three-point identity~\eqref{eq:three-point},
\begin{align*}
    &\left\langle\nabla\phi(\x^{k+1}) - \nabla\phi(\y^k), \x^k - \x^{k+1}\right\rangle\\&\quad= D_\phi(\x^k,\y^k)-D_\phi(\x^k,\x^{k+1}) - D_\phi(\x^{k+1}, \y^k),
\end{align*}
to~\eqref{ineq:decreasing-g}, we have
\begin{align}
    g(\x^k) - g(\x^{k+1}) &\geq -\left\langle\nabla \hat{f}_k(\y^k) , \x^k - \x^{k+1}\right\rangle\nonumber\\
    &\quad- \frac{1}{\lambda_k} \left(D_\phi(\x^k,\y^k)-D_\phi(\x^k,\x^{k+1}) - D_\phi(\x^{k+1}, \y^k)\right).\label{ineq:decreasing-bpge-g}
\end{align}
It follows from Lemma~\ref{lemma:extended-descent} that
\begin{align}
    &\hat{f}_k(\x^k) - \hat{f}_k(\x^{k+1}) - \langle\nabla\hat{f}_k(\y^k),\x^k -\x^{k+1}\rangle\nonumber\\
    &\quad=\hat{f}_k(\x^k) - \hat{f}_k(\y^k) - \langle\nabla\hat{f}_k(\y^k),\x^k -\y^k\rangle - \hat{f}_k(\x^{k+1})\nonumber\\
    &\qquad+ \hat{f}_k(\y^k) + \langle\nabla\hat{f}_k(\y^k),\x^{k+1} - \y^k\rangle\nonumber\\
    &\quad\geq -L_kD_\phi(\x^k,\y^k)-L_kD_\phi(\x^{k+1},\y^k).\label{ineq:decreasing-bpge-f}
\end{align}
We combine~\eqref{ineq:decreasing-bpge-g} and~\eqref{ineq:decreasing-bpge-f} into
\begin{align*}
    &\hat{f}_k(\x^{k+1}) + g(\x^{k+1})\\
    &\quad\leq\hat{f}_k(\x^k) + g(\x^{k+1}) - \langle\nabla\hat{f}_k(\y^k),\x^k -\x^{k+1}\rangle\\
    &\qquad+ L_kD_\phi(\x^k,\y^k)+L_kD_\phi(\x^{k+1},\y^k)\\
    &\quad\leq\hat{f}_k(\x^k) + g(\x^k) + \left(L_k + \frac{1}{\lambda_k}\right)D_\phi(\x^k,\y^k)\\
    &\qquad-\left(\frac{1}{\lambda_k} - L_k\right)D_\phi(\x^{k+1},\y^k) - \frac{1}{\lambda_k} D_\phi(\x^k,\x^{k+1}),
\end{align*}
which proves the desired inequality~\eqref{ineq:decreasing-property-bpge-obj} due to $\hat{f}_k(\x^k) = f(\x^k)$ and $f(\x) \leq \hat{f}_k(\x)$ for any $\x\in C$.

The adaptive restart scheme~\eqref{ineq:adaptive-restart} implies that there exists $\rho\in[0,1)$ such that $D_\phi(\x^k, \y^k) \leq \rho D_\phi(\x^{k-1}, \x^k)$.
It holds that
\begin{align*}
    \lambda_k\left(\Psi(\x^{k+1}) + MD_\phi(\x^k, \x^{k+1})\right)&\leq\lambda_k\left(\Psi(\x^k) + MD_\phi(\x^{k-1}, \x^k)\right)\\
    &\quad- (\lambda_k M - \rho\left(1 + \lambda_kL_k\right))D_\phi(\x^{k-1},\x^k)\\
    &\quad-\left(1 - \lambda_k L_k\right)D_\phi(\x^{k+1},\y^k)\\
    &\quad-\left(1 - \lambda_k M\right) D_\phi(\x^k,\x^{k+1}).
\end{align*}
When $0 < \lambda_k L_k < 1$ and $\rho(1 + \lambda_k L_k) < \lambda_k M < 1$ for all $k \geq 0$, $\lambda_k \Phi_M(\x^{k+1}) \leq \lambda_k \Phi_M(\x^k)$, that is, the potential function value decreases.
\qed
\end{proof}

\begin{remark}
In Lemma~\ref{lemma:decreasing-bpge}, $\rho(1 + \lambda_k L_k) < 1$ always holds if $\rho < 1/2$ and $\lambda_k L_k < 1$.
However, $\rho < 1/2$ may weaken the acceleration by the extrapolation technique.
In practice, \eqref{ineq:decreasing-property-bpge-lyap} holds even if $1/2 \leq \rho < 1$.
We will use $\rho = 0.999$ in numerical experiments.
\end{remark}

We also have the global convergence of MMBPGe in the same way as~Theorem~\ref{theorem:covergence-mmbpg} (see also the proof of~\cite[Theorems 6 and 7, and Appendix]{Takahashi2022-ml}).
In the same way as~\cite[Assumption 8]{Takahashi2022-ml} and~\cite[Assumption 3(iii)]{Zhang2019-ia}, we suppose that there is a bounded subdifferential of $\nabla\phi$.

\begin{assumption}\label{assumption:second-order-subdiff}
    For any compact subset $S\subset\interior\dom\phi$ and any $\x \in S$, there exists $A > 0$ such that $\|\bm{\xi}\|\leq A\|\u\|$ for any $\bm{\xi}\in\partial(\nabla\phi(\x))(\u)$ and any $\u\in\interior\dom\phi$.
\end{assumption}

The second-order subdifferential $\partial(\nabla\phi(\x))(\u)$ is defined by~\cite[Section 1.3.5]{Mordukhovich2006-xd} and~\cite[Section 1.2.1]{Mordukhovich2024-wa}.
When $\phi$ is $\mathcal{C}^2$, $\partial(\nabla\phi(\x))(\u) = \{\nabla^2\phi(\x)\u\}$.
Supposing $\Phi_M$, instead of $\Psi$, is a Kurdyka--\L{}ojasiewicz function, we have the global convergence of MMBPGe.

\begin{theorem}[Convergence of MMBPGe]\label{theorem:convergence-mmbpge}
    Suppose that Assumptions~{\rm\ref{assumption:setting},~\ref{assumption:bregman-mapping-well-defined},~\ref{assumption:l-smad}},~{\rm\ref{assumption:global-convergence}},~{\rm\ref{assumption:g-convex}}, and~{\rm\ref{assumption:second-order-subdiff}} hold.
    Let $\{\x^k\}_k$ be a sequence generated by MMBPGe with $0 < \lambda_k L_k < 1$ and $\rho(1 + \lambda_k L_k) < \lambda_k M < 1$ for solving~\eqref{prob:general}.
    Then, the following statements hold.
    \begin{enumerate}
        \item Global subsequential convergence: Any accumulation point of $\{\x^k\}_k$ is a KKT point of~\eqref{prob:general}.
        \item Global convergence: Suppose that $\Phi_M$ is a Kurdyka--\L{}ojasiewicz function.
        Then, the sequence $\{\x^k\}_k$ has finite length, \ie, $\sum_{k=1}^{\infty}\|\x^{k+1} - \x^k\| < \infty$, and converges to a stationary point of~\eqref{prob:general}.
    \end{enumerate}
\end{theorem}
\section{Majorization-Minimization Bregman Proximal Gradient Algorithms for KL-NMF}\label{sec:mmbpgs-kl-nmf}
\subsection{Majorization-Minimization Bregman Proximal Gradient Algorithm for KL-NMF}
Recalling the loss function,
\begin{align*}
    f(\W,\H) &:= D(\X, \W\H) = \sum_{i=1}^m\sum_{j=1}^n\left(X_{ij}\log\frac{X_{ij}}{(\W\H)_{ij}} - X_{ij} + (\W\H)_{ij}\right),
\end{align*}
we define the auxiliary function $\hat{f}$.
From the convexity of the negative log function, it holds that
\begin{align*}
    -\log(\W\H)_{ij} = -\log\sum_{l=1}^rW_{il}H_{lj}\leq-\sum_{l=1}^r\alpha_{ilj}\log \frac{W_{il}H_{lj}}{\alpha_{ilj}},
\end{align*}
where $\alpha_{ilj} \geq 0$ and $\sum_l\alpha_{ilj} = 1$.
Setting
\begin{align*}
    \alpha_{ilj} = \frac{W_{il}H_{lj}}{\sum_{l=1}^r W_{il}H_{lj}},
\end{align*}
we obtain the auxiliary function $\hat{f}$~\cite{Lee2000-nl,Marmin2023-jh} given by
\begin{align*}
    \hat{f}(\W,\H) := \sum_{i,j}\left(X_{ij}\log X_{ij} - X_{ij}\sum_{l=1}^r \alpha_{ilj}\log \frac{W_{il}H_{lj}}{\alpha_{ilj}} - X_{ij} + (\W\H)_{ij}\right).
\end{align*}
The multiplicative updates (MU)~\cite{Lee2000-nl} is known to be the alternating minimization of $\hat{f}$.
Given 
\begin{align}
    \alpha_{ilj}^k := \frac{W_{il}^kH_{lj}^k}{\sum_{l=1}^r W_{il}^kH_{lj}^k} \label{def:auxiliary-para}
\end{align}
at the $k$th iteration, we define the auxiliary function $\hat{f}_k$ at the $k$th iteration,
\begin{align*}
    \hat{f}_k(\W,\H) := \sum_{i,j}\left(X_{ij}\log X_{ij} - X_{ij}\sum_{l=1}^r \alpha^k_{ilj}\log \frac{W_{il}H_{lj}}{\alpha^k_{ilj}} - X_{ij} + (\W\H)_{ij}\right).
\end{align*}
Note that $f(\W^k,\H^k) = \hat{f}_k(\W^k, \H^k)$ and $f(\W,\H) \leq \hat{f}_k(\W, \H)$ hold for any $(\W, \H)\in\R_{++}^{m\times r}\times\R_{++}^{r \times n}$.
We see that the objective function is majorized by $\hat{f}_k$, that is, for any $(\W, \H)\in\R_{++}^{m\times r}\times\R_{++}^{r \times n}$, $\Psi(\W,\H) = f(\W,\H) + g(\W,\H) \leq \hat{f}_k(\W, \H) + g(\W,\H)$.

Although $\nabla\hat{f}_k$ is not globally Lipschitz continuous, MMBPG does not require the Lipschitz continuity of $\nabla\hat{f}_k$ to guarantee its global convergence.
Instead, MMBPG requires the smooth adaptable property for $(\hat{f}_k, \phi)$.
Now we are ready to describe MMBPG for solving~\eqref{prob:klnmf} as Algorithm~\ref{alg:mmbpg-kl-nmf}.

\begin{algorithm}[!t]
    \caption{Majorization-minimization Bregman proximal gradient algorithm for KL-NMF}
    \label{alg:mmbpg-kl-nmf}
    \Input{$\phi\in\mathcal{G}(\R_{++}^{m\times r}\times\R_{++}^{r \times n})$ such that the smooth adaptable property for the pair $(\hat{f}_k,\phi)$ holds on $\R_{++}^{m\times r}\times\R_{++}^{r \times n}$.
    An initial point $(\W^0, \H^0)\in\R_{++}^{m\times r}\times\R_{++}^{r \times n}$.
    }
    \For{$k = 0, 1, 2, \ldots,$}{
        Compute
        \begin{align*}
            \alpha_{ilj}^k = \frac{W_{il}^kH_{lj}^k}{\sum_{l=1}^r W_{il}^kH_{lj}^k}, \quad \forall i\in[m],l\in[r],j\in[n].
        \end{align*}
        Take $\lambda_k > 0$ and compute
        \begin{align}
            \Z^{k+1}
            = \argmin_{\Z\in\R_+^{m\times r}\times\R_+^{r \times n}}&
            \left\{\langle\nabla \hat{f}_k(\Z^k), \Z - \Z^k\rangle + g(\Z) + \frac{1}{\lambda_k} D_\phi(\Z,\Z^k)\right\},\label{subprob:mmbpg-kl-nmf}
        \end{align}
        where $\Z = (\W, \H)$ and $\Z^k = (\W^k, \H^k)$.
    }
\end{algorithm}

\subsection{Majorization-Minimization Bregman Proximal Gradient Algorithm with Extrapolation for KL-NMF}
To guarantee $f(\W^k,\H^k) = \hat{f}_k(\W^k, \H^k)$ for MMBPGe, we also use~\eqref{def:auxiliary-para} without $\beta_k$.
Hien~\etal~\cite{Hien2025-vc} used $\Y^k = \Z^k + \beta_k[\Z^k - \Z^{k-1}]_+$, where $[\cdot]_+ := \max(\cdot, 0)$ and we use the same notation when applying $[\cdot]_+$ elementwise.
Although this extrapolation works well for MU, it does not work well for MMBPGe in practice.
Therefore, we adopt $\Y^k = \Z^k + \beta_k(\Z^k - \Z^{k-1})$ and reset $\beta_k$ when $\Y^k\notin\R_{++}^{m\times r}\times\R_{++}^{r \times n}$.
Now we are ready to describe MMBPGe for solving~\eqref{prob:klnmf} as Algorithm~\ref{alg:mmbpge-kl-nmf}.

\begin{algorithm}[!t]
    \caption{Majorization-minimization Bregman proximal gradient algorithm with extrapolation for KL-NMF}
    \label{alg:mmbpge-kl-nmf}
    \Input{$\phi\in\mathcal{G}(\R_{++}^{m\times r}\times\R_{++}^{r \times n})$ such that the smooth adaptable property for the pair $(\hat{f}_k,\phi)$ holds on $\R_{++}^{m\times r}\times\R_{++}^{r \times n}$.
    An initial point $(\W^0, \H^0) = (\W^{-1}, \H^{-1})\in\R_{++}^{m\times r}\times\R_{++}^{r \times n}$.
    $\theta_0 = \theta_{-1} = 1, \rho\in(0,1]$
    }
    \For{$k = 0, 1, 2, \ldots,$}{
        Compute
        \begin{align*}
            \alpha_{ilj}^k = \frac{W_{il}^kH_{lj}^k}{\sum_{l=1}^r W_{il}^kH_{lj}^k}, \quad \forall i\in[m],l\in[r],j\in[n].
        \end{align*}
        Take $\lambda_k >0$ and compute
        \begin{align*}
            \beta_k &= \frac{\theta_{k-1} - 1}{\theta_k} \quad \text{with} \quad \theta_k = \frac{1 + \sqrt{1 + 4\theta_{k-1}^2}}{2},\\
            \Y^k &= \Z^k + \beta_k(\Z^k - \Z^{k-1}),
        \end{align*}
        where $\Z^k = (\W^k, \H^k)$.
        
        \If{$\Y^k\notin\R_{++}^{m\times r}\times\R_{++}^{r \times n}$ or $D_\phi(\Z^k, \Y^k) > \rho D_\phi(\Z^{k-1}, \Z^k)$}{
            Set $\Y^k = \Z^k$ and $\theta_k = \theta_{k-1} = 1$.
        }
        Compute
        \begin{align}
            \Z^{k+1}
            = \argmin_{\Z\in\R_+^{m\times r}\times\R_+^{r \times n}}&
            \left\{\langle\nabla \hat{f}_k(\Y^k), \Z - \Y^k\rangle + g(\Z) + \frac{1}{\lambda_k} D_\phi(\Z,\Y^k)\right\},\label{subprob:mmbpge-kl-nmf}
        \end{align}
        where $\Z = (\W, \H)$.
    }
\end{algorithm}
\subsection{Smooth Adaptable Property for KL-NMF}
To apply MMBPG and MMBPGe to KL-NMF~\eqref{prob:klnmf}, we need appropriate $\phi$ and $L$.
The pair $(\hat{f}_k, \phi)$ should be $L$-smad.
In addition, a separable $\phi$ is preferable to solve subproblems~\eqref{subprob:mmbpg} and~\eqref{subprob:mmbpge} in closed forms.
In Theorem~\ref{theorem:l-smad}, we have the smooth adaptable property for $(\hat{f}_k, \phi)$ with a separable $\phi$.
In Propositions~\ref{prop:closed-form-nmf-l1} and~\ref{prop:closed-form-nmf-fro}, by using a separable $\phi$, we obtain closed-form solutions of subproblems~\eqref{subprob:mmbpg} and~\eqref{subprob:mmbpge}.
\begin{theorem}
\label{theorem:l-smad}
Let
\begin{align}
    \phi(\W, \H) = \sum_{i=1}^m\sum_{l=1}^r\left(-\log W_{il}+\frac{1}{2}W_{il}^2\right) + \sum_{l=1}^r\sum_{j=1}^n\left(-\log H_{lj}+\frac{1}{2}H_{lj}^2\right). \label{eq:kernel-nmf}
\end{align}
Then, for any $L_k > 0$ satisfying
\begin{align}
    L_k \geq \max\left\{\max_{i\in[m],l\in[r]}\left\{\sum_{j=1}^n\alpha^k_{ilj}X_{ij}\right\},\max_{l\in[r],j\in[n]}\left\{\sum_{i=1}^m\alpha^k_{ilj}X_{ij}\right\}, m, n\right\},\label{ineq:l-smad-nmf}
\end{align}
the pair $(\hat{f}_k,\phi)$ is $L_k$-smad on $\R_{++}^{m\times r}\times\R_{++}^{r \times n}$ for all $k \geq 0$. Moreover, the pair $(\hat{f}, \phi)$ is $L$-smad for any $L \geq \max_{i, j}\left\{\sum_{j}X_{ij}, \sum_{i}X_{ij}, m, n\right\}$ on $\R_{++}^{m\times r}\times\R_{++}^{r \times n}$.
\end{theorem}
\begin{proof}
Note that $\alpha_{ijl}^k$ is a fixed parameter at the $k$th iteration.
We have
\begin{align*}
    \hat{f}_k(\W,\H) = \sum_{i=1}^m\sum_{j=1}^n\left(X_{ij}\sum_{l=1}^r \alpha^k_{ilj}\left(-\log W_{il} - \log H_{lj}\right) + (\W\H)_{ij}\right) + C_k,
\end{align*}
where $C_k = \sum_{i,j}(X_{ij}\log X_{ij} - X_{ij} + X_{ij}\sum_{l} \alpha^k_{ilj}\log\alpha^k_{ilj})$ is a constant term at the $k$th iteration.
Since the auxiliary function $\hat{f}_k$ is separable, we can consider the smooth adaptable property separately.
We see that
\begin{align*}
    \left(L_k - \sum_{j=1}^n\alpha^k_{ilj}X_{ij}\right)(-\log W_{il})\quad\text{and}\quad \left(L_k - \sum_{i=1}^m\alpha^k_{ilj}X_{ij}\right)(-\log H_{lj})
\end{align*}
are convex for all $i\in[m],l\in[r]$, and $j\in[n]$.
In addition,
\begin{align*}
    \frac{n}{2}\sum_{i=1}^m\sum_{l=1}^rW_{il}^2 + \frac{m}{2}\sum_{l=1}^r\sum_{j=1}^nH_{lj}^2 - \sum_{i=1}^m\sum_{j=1}^n(\W\H)_{ij} = \frac{1}{2} \sum_{i=1}^m\sum_{j=1}^n\sum_{l=1}^r\left(W_{il} - H_{lj}\right)^2
\end{align*}
is convex because the composition of a convex function and an affine transformation is convex~\cite[Theorem 2.17]{Beck2017-qc}.
Using the above, we find that $L_k\phi - \hat{f}_k$ is convex on $\R_{++}^{m\times r}\times\R_{++}^{r \times n}$.
On the other hand, $L_k\phi + \hat{f}_k$ is also convex by a similar discussion.
Therefore, $(\hat{f}_k,\phi)$ is $L_k$-smad on $\R_{++}^{m\times r}\times\R_{++}^{r \times n}$.
The pair $(\hat{f}, \phi)$ is $L$-smad for any $L \geq \max_{i, j}\left\{\sum_{j}X_{ij}, \sum_{i}X_{ij}, m, n\right\}$ because the right-hand side of~\eqref{ineq:l-smad-nmf} is bounded above by $L$ due to $\alpha_{ilj}^k \in[0,1]$.
\qed
\end{proof}

We use $\phi$ so that it is easy to verify the smooth adaptable property and to provide a closed-form solution of subproblems. An easy way to find $\phi$ is to use elementary functions similar to $f$ or $\hat{f}_k$. In this case, the logarithm is suitable.
In practice,~\eqref{ineq:l-smad-nmf} is not tight.
We have $\hat{f}_k(\W, \H) = \hat{f}_{k,1}(\W) + \hat{f}_{k,2}(\H) + \hat{f}_{k,3}(\W, \H) + C_k$, where $\hat{f}_{k,1}(\W) = -\sum_{i, j, l}X_{ij} \alpha^k_{ilj}\log W_{il}$, $\hat{f}_{k,2}(\H) = -\sum_{i, j, l} X_{ij}\alpha^k_{ilj}\log H_{il}$, and $\hat{f}_{k,3}(\W,\H) = \sum_{i, j}(\W\H)_{ij}$.
Let $\phi_1(\W) = -\sum_{i,l}\log W_{il} + \frac{1}{2}\|\W\|_F^2$ and $\phi_2(\H) = -\sum_{l,j}\log H_{lj} + \frac{1}{2}\|\H\|_F^2$.
In the same way as the proof of Theorem~\ref{theorem:l-smad}, for any $L_{k,i} > 0$, $i = 1,2$ satisfying $L_{k,1} \geq \max_{i,l}\left\{\sum_{j=1}^n\alpha^k_{ilj}X_{ij}, m, n\right\}$ and $L_{k,2} \geq\max_{l,j}\left\{\sum_{i=1}^m\alpha^k_{ilj}X_{ij}, m, n\right\}$, $L_{k,1} \phi_1 - \hat{f}_{k,1}$ and $L_{k,2} \phi_2 - \hat{f}_{k,2}$ are convex.
Obviously, $\min\{L_{k,1}, L_{k,2}\}(\phi_1(\W) + \phi_2(\H)) - \hat{f}_{k,3}$ is convex because of $\min\{L_{k,1}, L_{k,2}\} \geq \max\{m, n\}$.
Therefore, when $g = g_1 + g_2$ for $g_1:\R^{m\times r}\to(-\infty,+\infty]$ and $g_2:\R^{r\times n}\to(-\infty,+\infty]$, instead of~\eqref{subprob:mmbpg-kl-nmf},  we can use
\begin{align}
    \W^{k+1}
    &= \argmin_{\W\in\R_+^{m\times r}}
    \left\{\langle\nabla_{\W} \hat{f}_k(\Z^k), \W - \W^k\rangle + g_1(\W) + \frac{1}{\lambda_{k,1}} D_{\phi_1}(\W,\W^k)\right\},\label{subprob:mmbpg-kl-nmf-h}\\
    \H^{k+1}
    &= \argmin_{\H\in\R_+^{r \times n}}
    \left\{\langle\nabla_{\H} \hat{f}_k(\Z^k), \H - \H^k\rangle + g_2(\H) + \frac{1}{\lambda_{k,2}} D_{\phi_2}(\H,\H^k)\right\},\label{subprob:mmbpg-kl-nmf-w}
\end{align}
where $0 < \lambda_{k,i} L_{k,i} < 1$, $i = 1,2$ hold and $\nabla_{\W} \hat{f}_k$ and $\nabla_{\H} \hat{f}_k$ are gradients with respect to $\W$ and $\H$, respectively.

When $g$ and $\phi$ are separable, it would be easy to solve~\eqref{subprob:mmbpg} in closed form.
Even if the subproblem is not solved in closed form, it is equivalent to independent $(n + m)r$ one-dimensional convex optimization problems.
We provide closed-form solutions of subproblem~\eqref{subprob:mmbpg-kl-nmf} with $\ell_1$ regularization or squared Frobenius norm regularization.
Propositions~\ref{prop:closed-form-nmf-l1} and~\ref{prop:closed-form-nmf-fro} are also true for subproblem~\eqref{subprob:mmbpge-kl-nmf} by replacing $(\W^k, \H^k)$ with $\Y^k$.
\begin{proposition}[$\ell_1$ regularization]
    \label{prop:closed-form-nmf-l1}
    Let $g(\W,\H) = \mu_{\W}\sum_{i,l}|W_{il}| + \mu_{\H}\sum_{l,j}|H_{lj}|$ with $\mu_{\W} > 0$ and $\mu_{\H} > 0$ and $\phi$ be defined by~\eqref{eq:kernel-nmf}.
    Then, $\Z^{k+1} = (\W^{k+1}, \H^{k+1})$ in~\eqref{subprob:mmbpg} is given by
    \begin{align*}
        W_{il}^{k+1} &= \frac{-P_{il} - \mu_{\W}\lambda_k + \sqrt{\left(P_{il}+\mu_{\W}\lambda_k\right)^2 + 4}}{2} > 0,\quad\forall i\in[m],l\in[r],\\
        H_{lj}^{k+1} &= \frac{-Q_{lj} - \mu_{\H}\lambda_k + \sqrt{\left(Q_{lj}+\mu_{\H}\lambda_k\right)^2 + 4}}{2} > 0,\quad\forall l\in[r],j\in[n],
    \end{align*}
    where
    \begin{align*}
        \P &= \lambda_k\nabla_{\W}\hat{f}_k(\W^k, \H^k) - \nabla_{\W}\phi(\W^k, \H^k),\\
        \Q &= \lambda_k\nabla_{\H}\hat{f}_k(\W^k, \H^k) - \nabla_{\H}\phi(\W^k, \H^k).
    \end{align*}
\end{proposition}
\begin{proof}
    Using $\P$ and $\Q$, we find that $(\W^{k+1}, \H^{k+1})$ is the optimal solution of the following strongly convex optimization problem:
    \begin{align*}
        \min_{(\W,\H)\in\R_+^{m\times r}\times\R_+^{r \times n}}\quad
        \left\langle\P, \W\right\rangle + \left\langle\Q, \H\right\rangle + \lambda_kg(\W, \H) + \phi(\W,\H).
    \end{align*}
    Assume that $W_{il}^{k+1} > 0$ and $H_{lj}^{k+1} > 0$.
    By the first-order optimality condition, we obtain
    \begin{align*}
        P_{il} + \mu_{\W}\lambda_k - \frac{1}{W_{il}^{k+1}} + W_{il}^{k+1} = 0.
    \end{align*}
    We solve the above equation and have
    \begin{align*}
        W_{il}^{k+1} = \frac{-P_{il} - \mu_{\W}\lambda_k + \sqrt{\left(P_{il}+\mu_{\W}\lambda_k\right)^2 + 4}}{2} > 0.
    \end{align*}
    We see $W_{il}^{k+1} = 0$ only when $P_{il} \to +\infty$.
    Note that if we allow $P_{il}$ to include infinity, $W_{il}^{k+1}$ is nonnegative.
    We also obtain $\H^{k+1}$ in the same way.
    \qed
\end{proof}
\begin{proposition}[Squared Frobenius norm regularization]
    \label{prop:closed-form-nmf-fro}
    Let $g(\W,\H) = \frac{\mu_{\W}}{2}\|\W\|_F^2 + \frac{\mu_{\H}}{2}\|\H\|_F^2$ with $\mu_{\W} > 0$ and $\mu_{\H} > 0$ and $\phi$ be defined by~\eqref{eq:kernel-nmf}.
    Then, $\Z^{k+1} = (\W^{k+1}, \H^{k+1})$ in~\eqref{subprob:mmbpg} is given by
    \begin{align*}
        W_{il}^{k+1} &= \frac{-P_{il} + \sqrt{P_{il}^2 + 4(1 + \mu_{\W}\lambda_k)}}{2(1 + \mu_{\W}\lambda_k)} > 0,\quad\forall i\in[m],l\in[r],\\
        H_{lj}^{k+1} &= \frac{-Q_{lj} + \sqrt{Q_{lj}^2 + 4(1 + \mu_{\H}\lambda_k)}}{2(1 + \mu_{\H}\lambda_k)} > 0,\quad\forall l\in[r],j\in[n],
    \end{align*}
    where
    \begin{align*}
        \P &= \lambda_k\nabla_{\W}\hat{f}_k(\W^k, \H^k) - \nabla_{\W}\phi(\W^k, \H^k),\\
        \Q &= \lambda_k\nabla_{\H}\hat{f}_k(\W^k, \H^k) - \nabla_{\H}\phi(\W^k, \H^k).
    \end{align*}
\end{proposition}
\begin{proof}
    Using $\P$ and $\Q$, we find that $(\W^{k+1}, \H^{k+1})$ is the optimal solution of the following strongly convex optimization problem:
    \begin{align*}
        \min_{(\W,\H)\in\R_+^{m\times r}\times\R_+^{r \times n}}\quad
        \left\langle\P, \W\right\rangle + \left\langle\Q, \H\right\rangle + \lambda_kg(\W, \H) + \phi(\W,\H).
    \end{align*}
    Assume that $W_{il}^{k+1} > 0$ and $H_{lj}^{k+1} > 0$.
    By the first-order optimality condition, we obtain
    \begin{align*}
        P_{il} - \frac{1}{W_{il}^{k+1}} + (1 + \mu_{\W}\lambda_k)W_{il}^{k+1} = 0.
    \end{align*}
    We solve the above equation and have
    \begin{align*}
        W_{il}^{k+1} = \frac{-P_{il} + \sqrt{P_{il}^2 + 4(1 + \mu_{\W}\lambda_k)}}{2(1 + \mu_{\W}\lambda_k)} > 0.
    \end{align*}
    We see $W_{il}^{k+1} = 0$ only when $P_{il} \to +\infty$ in the same discussion as Proposition~\ref{prop:closed-form-nmf-l1}.
    \qed
\end{proof}

When $g\equiv0$, by setting $\mu_{\W}$ and $\mu_{\H}$ to 0, we have
\begin{align*}
    W_{il}^{k+1} = \frac{-P_{il} + \sqrt{P_{il}^2 + 4}}{2} > 0 \quad\text{and}\quad H_{lj}^{k+1} = \frac{-Q_{lj} + \sqrt{Q_{lj}^2 + 4}}{2} > 0,
\end{align*}
which are positive.

\begin{corollary}
    Let $g$ be defined by Proposition~{\rm\ref{prop:closed-form-nmf-l1}} or~{\rm\ref{prop:closed-form-nmf-fro}} and $\phi$ be defined by~\eqref{eq:kernel-nmf}.
    Let $\{\Z^k\}_k$ be a sequence generated by MMBPG with $0 < \lambda_k L_k < 1$ or MMBPGe with $\rho(1 +\lambda_k L_k) < \lambda_kM < 1$ for solving~\eqref{prob:klnmf}.
    Then, the sequence $\{\Z^k\}_k$ has finite length, \ie, $\sum_{k=1}^{\infty}\|\Z^{k+1} - \Z^k\|_F < \infty$, and converges to a KKT point of~\eqref{prob:klnmf}.
\end{corollary}
\begin{proof}
We choose $\phi$ in~\eqref{eq:kernel-nmf}, \ie, $\interior\dom \phi = \R_{++}^{m\times r}\times\R_{++}^{r\times n}$ and $\cl{\dom\phi} = \R_+^{m\times r}\times\R_+^{r\times n}$.
This implies Assumption~\ref{assumption:setting}(i) holds.
Obviously, Assumption~\ref{assumption:setting}(ii) holds because $f$ and $\hat{f}_k$ are proper and lower semicontinuous functions, which are $\mathcal{C}^1$ on $\dom f = \dom \hat{f} = \R_{++}^{m\times r}\times\R_{++}^{r\times n}$.
In addition, $g$ in Propositions~\ref{prop:closed-form-nmf-l1} and~\ref{prop:closed-form-nmf-fro} satisfies Assumption~\ref{assumption:setting}(iii) because $g$ is a proper and lower semicontinuous function with $\dom g = \R^{m\times r}\times\R^{r\times n}$.
Because Proposition~\ref{prop:global-solution} implies $\inf\Psi(\W,\H) > -\infty$, Assumption~\ref{assumption:setting}(iv) holds.
The function $\phi + \lambda g$ is supercoercive for any $\lambda > 0$ due to the strong convexity of $\phi$ and~\cite[Corollary 11.17]{Bauschke2017-jp}, \ie,~Assumption~\ref{assumption:setting}(v) holds.

From Propositions~\ref{prop:closed-form-nmf-l1} and~\ref{prop:closed-form-nmf-fro}, we have $\mathcal{T}_{\lambda_k}(\Z^k)\subset\interior\dom\phi = \R_{++}^{m\times r}\times\R_{++}^{r\times n}$, which implies that Assumption~\ref{assumption:bregman-mapping-well-defined} holds.
Assumption~\ref{assumption:l-smad} holds because of Theorem~\ref{theorem:l-smad}.
Moreover, Assumption~\ref{assumption:global-convergence}(i) holds because $\phi$ is strongly convex and $\nabla\phi$ and $\nabla \hat{f}_k$ are Lipschitz continuous on any compact subset of $\R_{++}^{m\times r}\times\R_{++}^{r\times n}$.
The objective function $\Psi = f + g$ is coercive, \ie, level-bounded~\cite[Proposition 11.12]{Bauschke2017-jp} because $\Psi \geq g$ holds from $f \geq 0$ and $g$ is coercive, \ie, $\lim_{\|(\W, \H)\|\to+\infty}g(\W, \H) = +\infty$.
Thus, Assumption~\ref{assumption:global-convergence}(ii) holds.
Obviously, Assumption~\ref{assumption:g-convex} holds.
Because $\phi$ is $\mathcal{C}^2$ and $\Z \in S$, where $S$ is a compact subset of $\interior\dom\phi$, from the extreme value theorem, there exists $A > 0$ such that $\|\nabla^2\phi(\Z)\U\| \leq \|\nabla^2\phi(\Z)\|\|\U\| \leq A\|\U\|$ for all $\U \in \interior\dom\phi$, \ie, Assumption~\ref{assumption:second-order-subdiff} holds.

In addition, logarithmic and polynomial functions are real-analytic on their domain.
Although absolute functions are not real-analytic on their domain, $|x| = x$ is real-analytic on $x \geq 0$.
Therefore, the objective function and the potential function are Kurdyka--\L{}ojasiewicz functions because real-analytic functions satisfy the Kurdyka--\L{}ojasiewicz property on their domain~\cite{Bolte2007-xr,Lojasiewicz1963-kf}.
Therefore, $\{\Z^k\}_k$ converges to a KKT point from Theorems~\ref{theorem:covergence-mmbpg} and~\ref{theorem:convergence-mmbpge}.\qed
\end{proof}

\section{Numerical Experiments}\label{sec:numerical-experiments}
In this section, we conducted numerical experiments to examine the performance of MMBPG and MMBPGe for KL-NMF~\eqref{prob:klnmf}.
All numerical experiments were performed in Python 3.11 on a MacBook Pro with an Apple M2 Max 64GB LPDDR5 memory.

We compared MMBPG and MMBPGe with the multiplicative update method (MU)~\cite{Lee2000-nl}, the MU with extrapolation (MUe)~\cite{Hien2025-vc}, the cyclic coordinate descent algorithm (CCD)~\cite{Hsieh2011-ik}, and the alternating gradient descent method with line search (AGD).
MMBPG(e) used $\lambda_k = \frac{1}{L_k}$ satisfying~\eqref{ineq:l-smad-nmf}.
MMBPGe set $\rho = 0.999$.
The number of inner iterations of CCD is 100.
AGD minimized $f$ alternately and set the initial step-size as $(c L_0)^{-1}$ satisfying~\eqref{ineq:l-smad-nmf} with a constant $c > 1$.
The termination criteria is $\|(\W^{k+1}, \H^{k+1}) - (\W^k, \H^k)\|_F/\max\{1, \|(\W^{k+1}, \H^{k+1})\|_F\}\leq 10^{-9}$.

We define the relative error~\cite{Hien2021-sh} to be $f(\W, \H)$ divided by $\sum_{i,j}X_{ij}\log\frac{nX_{ij}}{\sum_{j}X_{ij}}$.
Let $\W_{\bullet,l}$ and $\H_{l, \bullet}$ be the $l$th column vector of $\W$ and the $l$th row vector of $\H$. The notation $(\W_{\bullet,l})_l$ denotes the matrix whose columns are the vectors $\W_{\bullet,l}$. Likewise, $(\H_{l,\bullet})_l$ is the matrix whose rows are the vectors $\H_{l, \bullet}$. When the context is clear, we omit the subscripts and write $\frac{\W_{\bullet,l}}{\|\W_{\bullet,l}\|}$ and $\frac{\H_{l, \bullet}}{\|\H_{l, \bullet}\|}$ in figures simply as $\left(\frac{\W_{\bullet,l}}{\|\W_{\bullet,l}\|}\right)_l$ and $\left(\frac{\H_{l, \bullet}}{\|\H_{l, \bullet}\|}\right)_l$.
We also use the norm of the normalized KKT conditions $\left\|\left(\frac{\W_{\bullet,l}}{\|\W_{\bullet,l}\|}\right)_l\odot\nabla_{\W} f(\W^k,\H^k)\right\|$ and $\left\|\left(\frac{\H_{l, \bullet}}{\|\H_{l, \bullet}\|}\right)_l\odot\nabla_{\H} f(\W^k,\H^k)\right\|$, where $\odot$ denotes Hadamard (elementwise) product. The norm of the normalized KKT conditions implies the performance of KKT conditions~\eqref{cond:kkt_w} and~\eqref{cond:kkt_h}, respectively.

\subsection{Synthetic Data}
We consider synthetic data.
We generate the ground truth $\W^*$ from i.i.d. uniform distribution and $\H^*$ from the Dirichlet distribution and then set $\X = \W^*\H^*$.
Figure~\ref{fig:nmf-heatmap} shows heat maps of $\H^*(\H^*)^{\mathsf{T}}$ for the ground truth $\H^*$.
We generate the initial point $(\W^0, \H^0)$ from i.i.d. uniform distribution and scaled it by $(\alpha\W^0, \alpha\H^0)$ for different $\alpha > 0$, where we use $\alpha = 1$ as the unscaled initial point or $\alpha = \sqrt{\frac{\sum_{i,j}X_{i,j}}{\sum_{i,j}(\W^0\H^0)_{ij}}}$ as the scaled initial point~\cite{Hien2021-sh}.
The maximum iteration is 3000 for algorithms except for CCD.

\begin{figure}[!tphb]
    \centering
    \begin{minipage}[b]{0.49\linewidth}
        \centering
        \includegraphics[width=\textwidth]{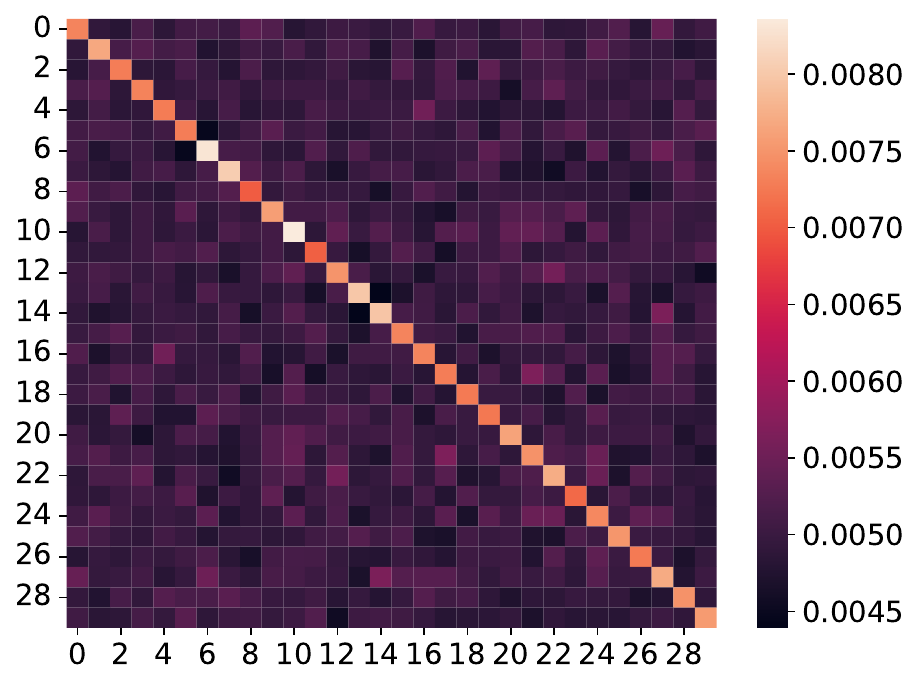}
        \subcaption{The heat map of $\H^*(\H^*)^{\mathsf{T}}$ for $\H^*\in\R^{30\times200}$.}
        \label{fig:nmf-200-200-30-heatmap}
    \end{minipage}
    \hfill
    \begin{minipage}[b]{0.49\linewidth}
        \centering
        \includegraphics[width=\textwidth]{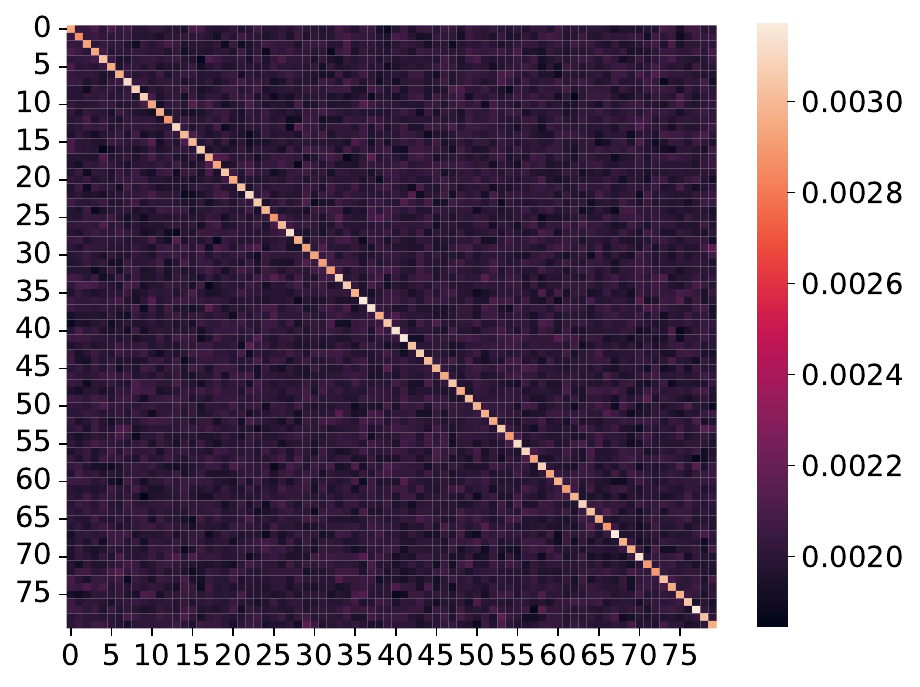}
        \subcaption{The heat map of $\H^*(\H^*)^{\mathsf{T}}$ for $\H^*\in\R^{80\times500}$.}
        \label{fig:nmf-500-500-80-heatmap}
    \end{minipage}
    \caption{The heat maps of $\H^*(\H^*)^{\mathsf{T}}$ for the ground truth $\H^*$.}
    \label{fig:nmf-heatmap}
\end{figure}

We first consider $g \equiv 0$.
The performance results are shown in Figures~\ref{fig:nmf-x0unscaled-200-200-30} and~\ref{fig:nmf-x0scaled-200-200-30} for $(m, n, r) = (200, 200, 30)$ and Figures~\ref{fig:nmf-x0unscaled-500-500-80} and~\ref{fig:nmf-x0scaled-500-500-80} for $(m, n, r) = (500, 500, 80)$ with different $\alpha$.
The maximum iterations of CCD are 600 for $(m, n, r) = (200, 200, 30)$ and 100 for $(m, n, r) = (500, 500, 80)$.
Figures~\ref{fig:nmf-x0unscaled-200-200-30-rel-iter} and~\ref{fig:nmf-x0unscaled-500-500-80-rel-iter} show that the relative error value of MMBPGe at the 3000th iteration is better than that of other algorithms, while Figures~\ref{fig:nmf-x0scaled-200-200-30-rel-iter} and~\ref{fig:nmf-x0scaled-500-500-80-rel-iter} show that that of MUe is better than that of other algorithms.
On the other hand, Figures~\ref{fig:nmf-x0scaled-200-200-30-kktw-iter},~\ref{fig:nmf-x0scaled-200-200-30-kkth-iter},~\ref{fig:nmf-x0scaled-500-500-80-kktw-iter}, and~\ref{fig:nmf-x0scaled-500-500-80-kkth-iter} show that MMBPGe recovered $\W$ and $\H$ with the minimum norm value of the normalized KKT condition.
For the computation time in 3000 iterations, MMBPGe is almost as fast as or faster than MUe.
Figure~\ref{fig:nmf-x0unscaled-200-200-30} shows that AGD terminated before the 3000th iteration because KL-NMF~\eqref{prob:klnmf} lacks Lipschitz continuity of $\nabla f$ and step-sizes would be very small due to many line search procedures for this experiment.
This implies that $\nabla f$ would not be Lipschitz continuous around the unscaled initial point.
MMBPGe performs well even if $\nabla f$ is not Lipschitz continuous.
Table~\ref{tab:resetting-g0} shows the number of times $\theta_k$ and $\beta_k$ were reset in MMBPGe of Figures~\ref{fig:nmf-x0unscaled-200-200-30},~\ref{fig:nmf-x0scaled-200-200-30},~\ref{fig:nmf-x0unscaled-500-500-80}, and~\ref{fig:nmf-x0scaled-500-500-80}. The restart frequency is low at approximately 0.2\% in Table~\ref{tab:resetting-g0}.

\begin{table}[!tpb]
\caption{The frequency of occurrence of $\Y^k \not\in\R_{++}^{m \times r} \times \R_{++}^{r \times n}$ and \eqref{ineq:adaptive-restart} for MMBPGe in Figures~\ref{fig:nmf-x0unscaled-200-200-30},~\ref{fig:nmf-x0scaled-200-200-30},~\ref{fig:nmf-x0unscaled-500-500-80}, and~\ref{fig:nmf-x0scaled-500-500-80}.}
\label{tab:resetting-g0}
\centering
\begin{tabular}{lllcrr}
\toprule
$m$ & $n$ & $r$ & scaled & nonpositive & \eqref{ineq:adaptive-restart}\\ \midrule
200 & 200 & 30 & \xmarkr & 0 & 7 \\
200 & 200 & 30 & \cmarkg & 0 & 7 \\\midrule
500 & 500 & 80 & \xmarkr & 1 & 6 \\
500 & 500 & 80 & \cmarkg & 0 & 5 \\
\bottomrule
\end{tabular}
\end{table}

\begin{figure}[!pth]
    \centering
    \begin{minipage}[b]{0.49\linewidth}
        \centering
        \includegraphics[width=\textwidth]{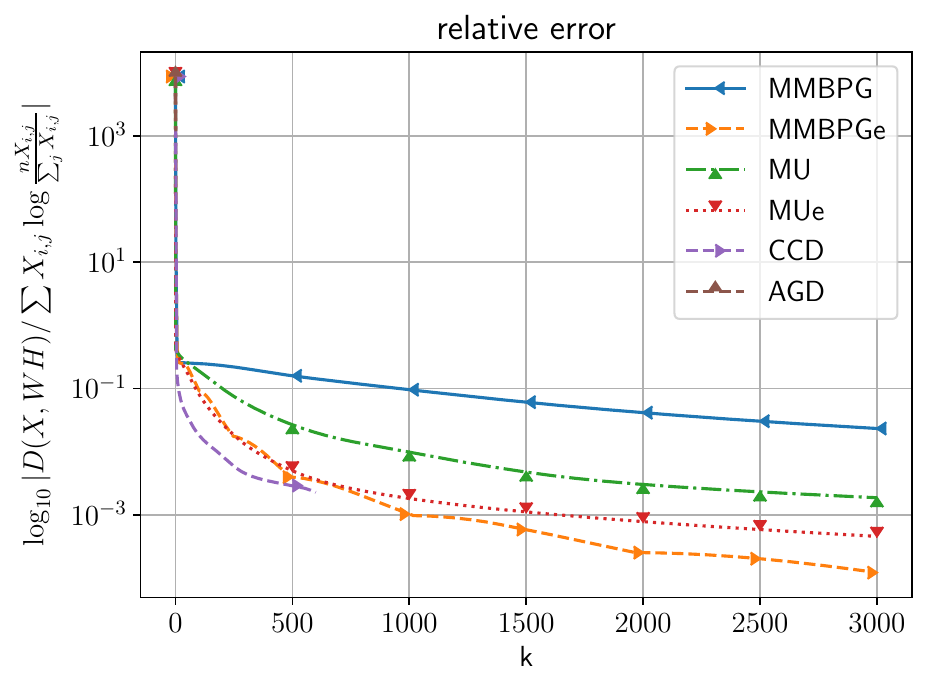}
        \subcaption{The relative error values with the iteration axis.}
        \label{fig:nmf-x0unscaled-200-200-30-rel-iter}
    \end{minipage}
    \hfill
    \begin{minipage}[b]{0.49\linewidth}
        \centering
        \includegraphics[width=\textwidth]{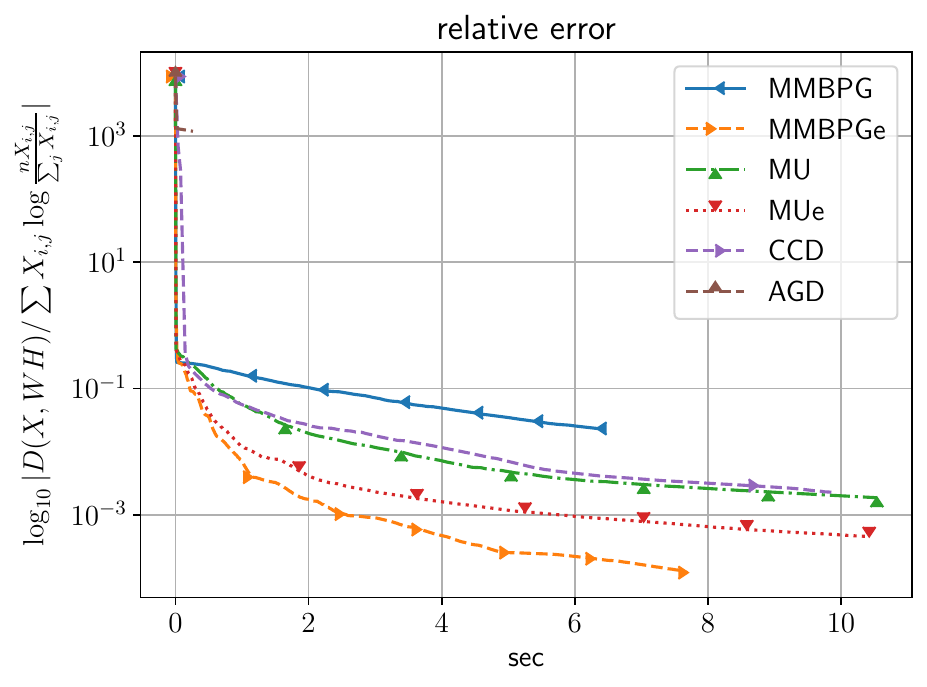}
        \subcaption{The relative error values with the time axis.}
        \label{fig:nmf-x0unscaled-200-200-30-rel-time}
    \end{minipage}
    \begin{minipage}[b]{0.49\linewidth}
        \centering
        \includegraphics[width=\textwidth]{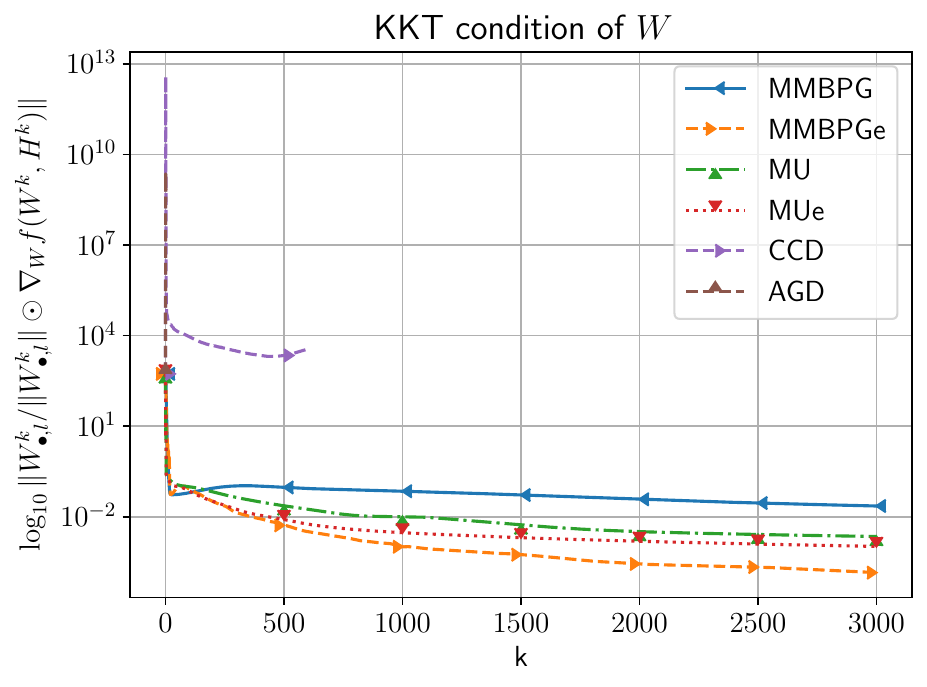}
        \subcaption{The KKT condition with respect to $\W$ with the iteration axis.}
        \label{fig:nmf-x0unscaled-200-200-30-kktw-iter}
    \end{minipage}
    \hfill
    \begin{minipage}[b]{0.49\linewidth}
        \centering
        \includegraphics[width=\textwidth]{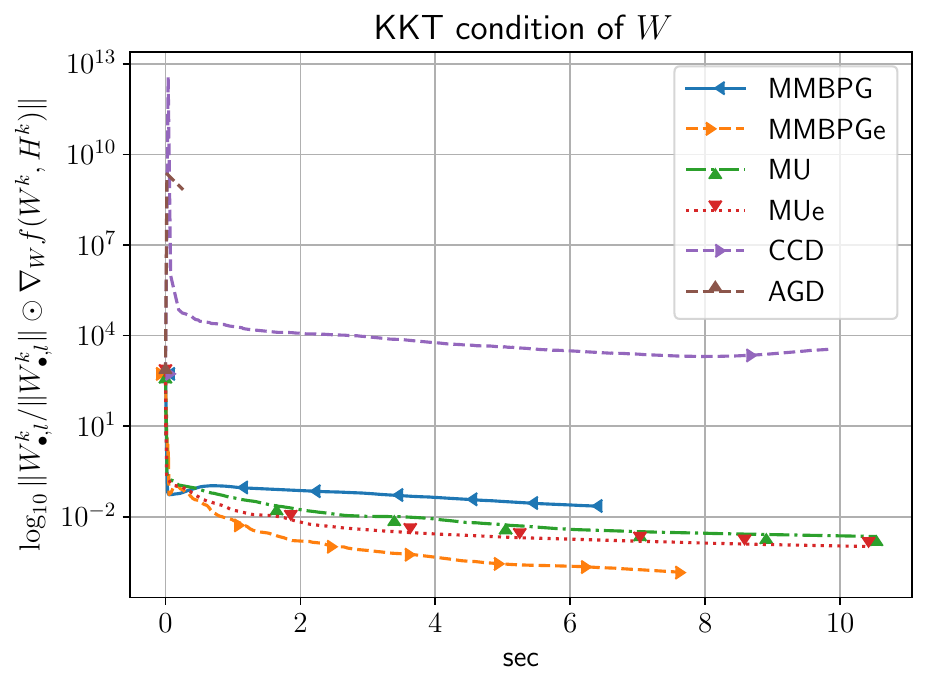}
        \subcaption{The KKT condition with respect to $\W$ with the time axis.}
        \label{fig:nmf-x0unscaled-200-200-30-kktw-time}
    \end{minipage}
    \begin{minipage}[b]{0.49\linewidth}
        \centering
        \includegraphics[width=\textwidth]{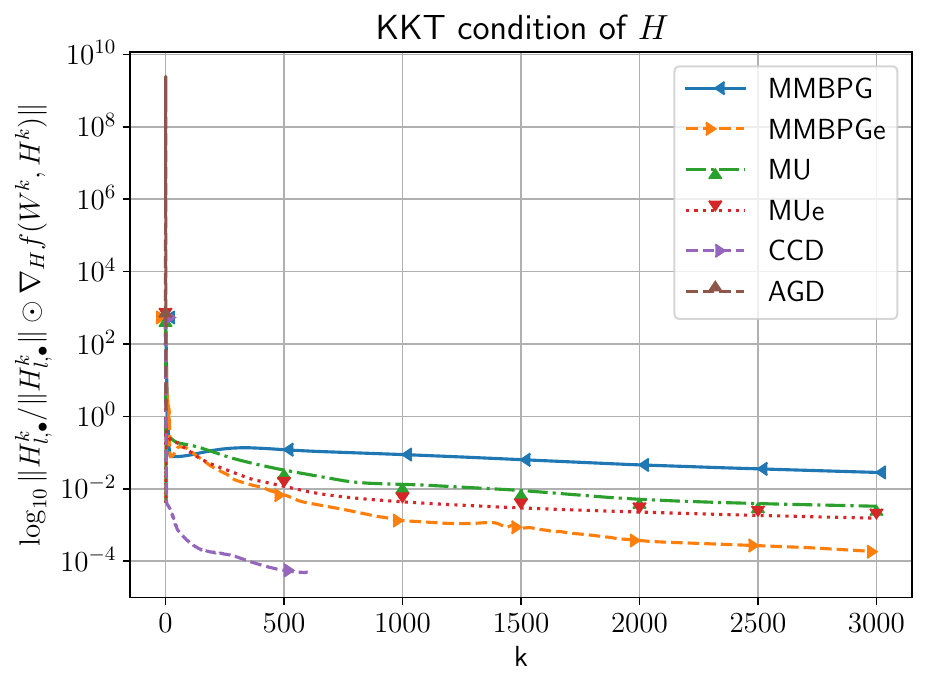}
        \subcaption{The KKT condition with respect to $\H$ with the iteration axis.}
        \label{fig:nmf-x0unscaled-200-200-30-kkth-iter}
    \end{minipage}
    \hfill
    \begin{minipage}[b]{0.49\linewidth}
        \centering
        \includegraphics[width=\textwidth]{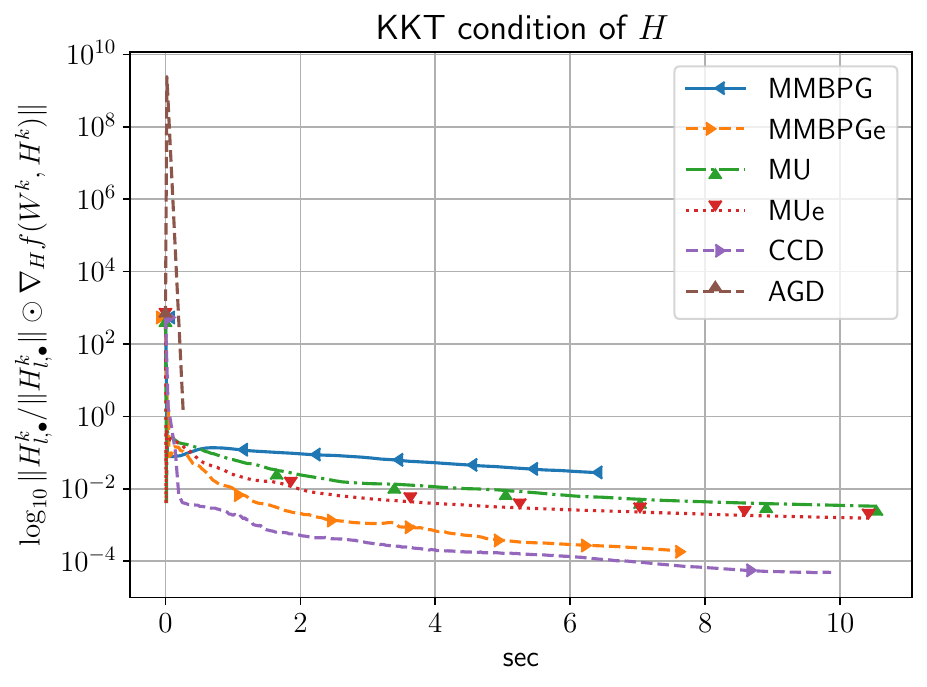}
        \subcaption{The KKT condition with respect to $\H$ with the time axis.}
        \label{fig:nmf-x0unscaled-200-200-30-kkth-time}
    \end{minipage}
    \caption{Comparison with MMBPG (blue), MMBPGe (orange), MU (green), MUe (red), CCD (purple), and AGD (brown) on KL-NMF ($(m,n,r) = (200, 200, 30)$) from the unscaled initial point.
    The left column corresponds to the iteration axis, and the right column corresponds to the time axis.}
    \label{fig:nmf-x0unscaled-200-200-30}
\end{figure}

\begin{figure}[!pth]
    \centering
    \begin{minipage}[b]{0.49\linewidth}
        \centering
        \includegraphics[width=\textwidth]{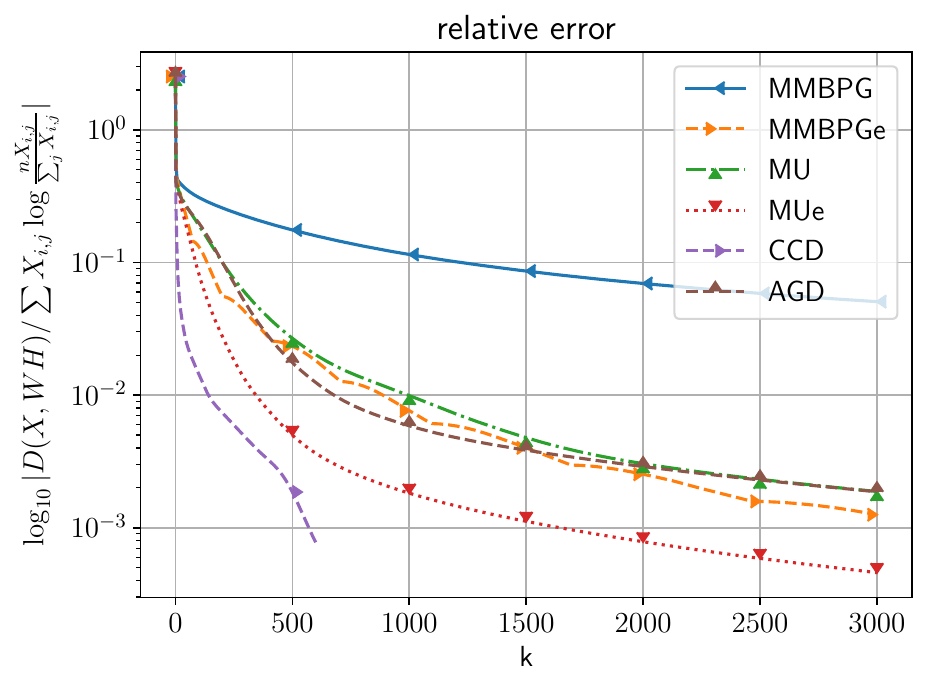}
        \subcaption{The relative error values with the iteration axis.}
        \label{fig:nmf-x0scaled-200-200-30-rel-iter}
    \end{minipage}
    \hfill
    \begin{minipage}[b]{0.49\linewidth}
        \centering
        \includegraphics[width=\textwidth]{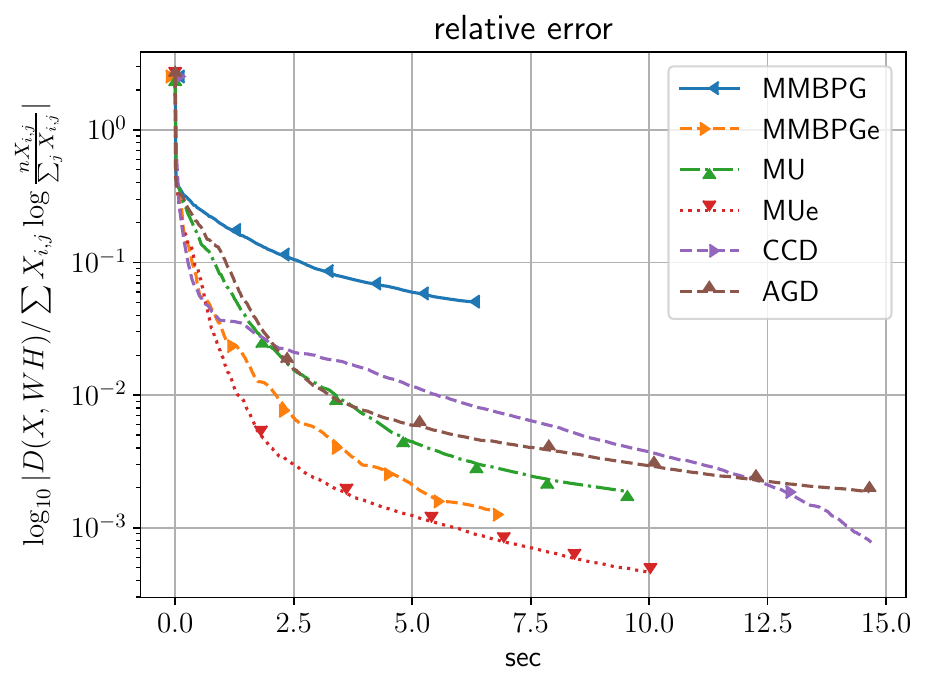}
        \subcaption{The relative error values with the time axis.}
        \label{fig:nmf-x0scaled-200-200-30-rel-time}
    \end{minipage}
    \begin{minipage}[b]{0.49\linewidth}
        \centering
        \includegraphics[width=\textwidth]{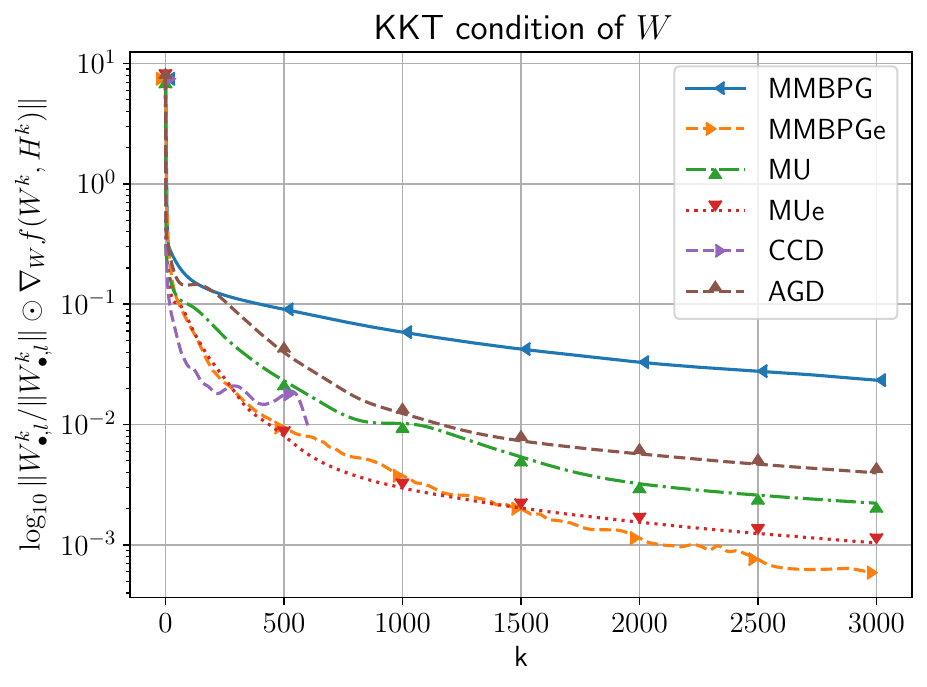}
        \subcaption{The KKT condition with respect to $\W$ with the iteration axis.}
        \label{fig:nmf-x0scaled-200-200-30-kktw-iter}
    \end{minipage}
    \hfill
    \begin{minipage}[b]{0.49\linewidth}
        \centering
        \includegraphics[width=\textwidth]{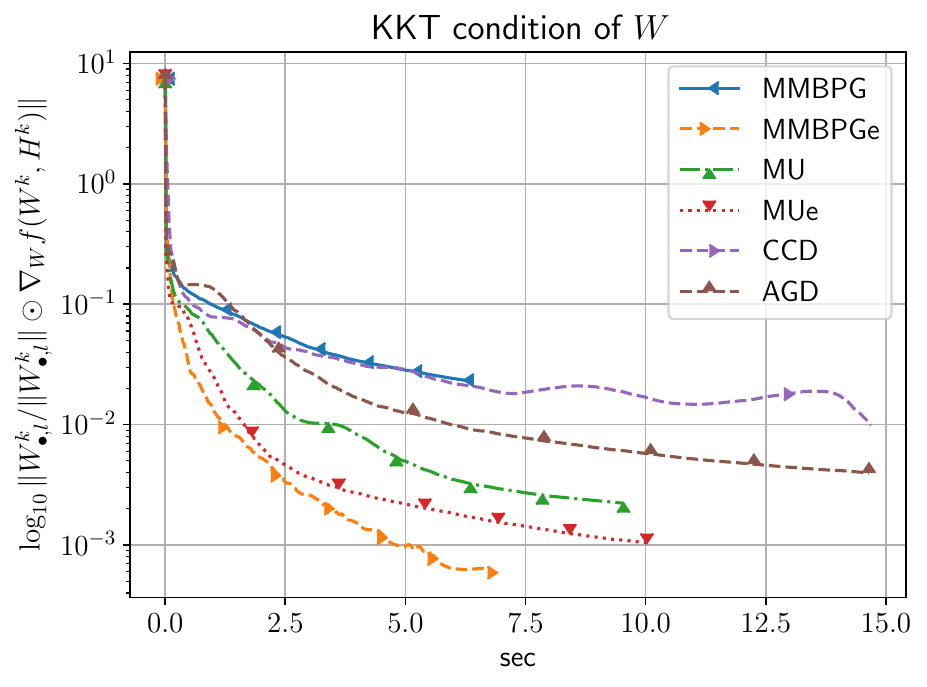}
        \subcaption{The KKT condition with respect to $\W$ with the time axis.}
        \label{fig:nmf-x0scaled-200-200-30-kktw-time}
    \end{minipage}
    \begin{minipage}[b]{0.49\linewidth}
        \centering
        \includegraphics[width=\textwidth]{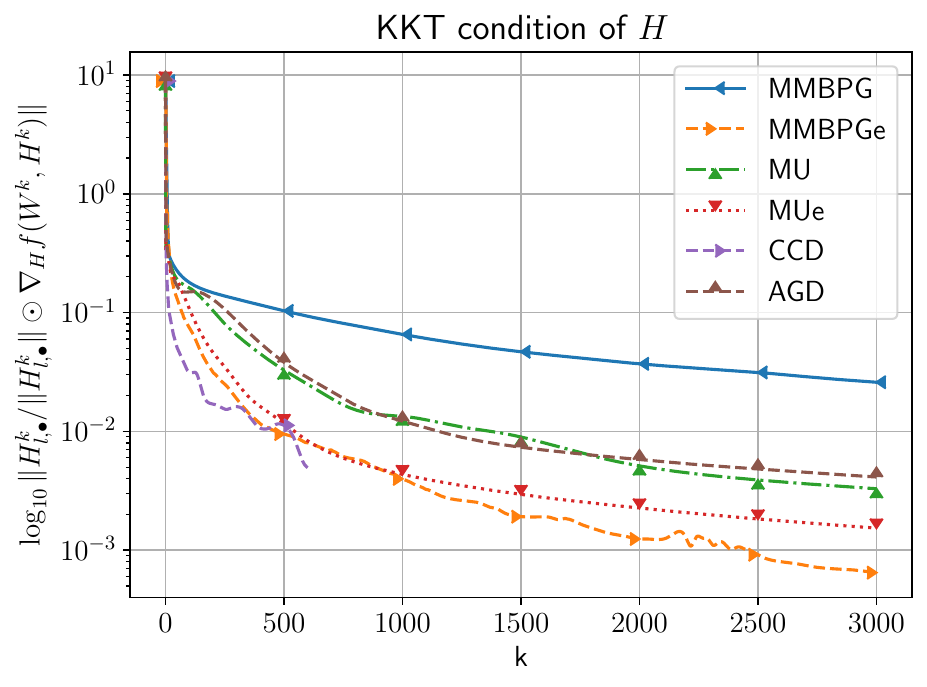}
        \subcaption{The KKT condition with respect to $\H$ with the iteration axis.}
        \label{fig:nmf-x0scaled-200-200-30-kkth-iter}
    \end{minipage}
    \hfill
    \begin{minipage}[b]{0.49\linewidth}
        \centering
        \includegraphics[width=\textwidth]{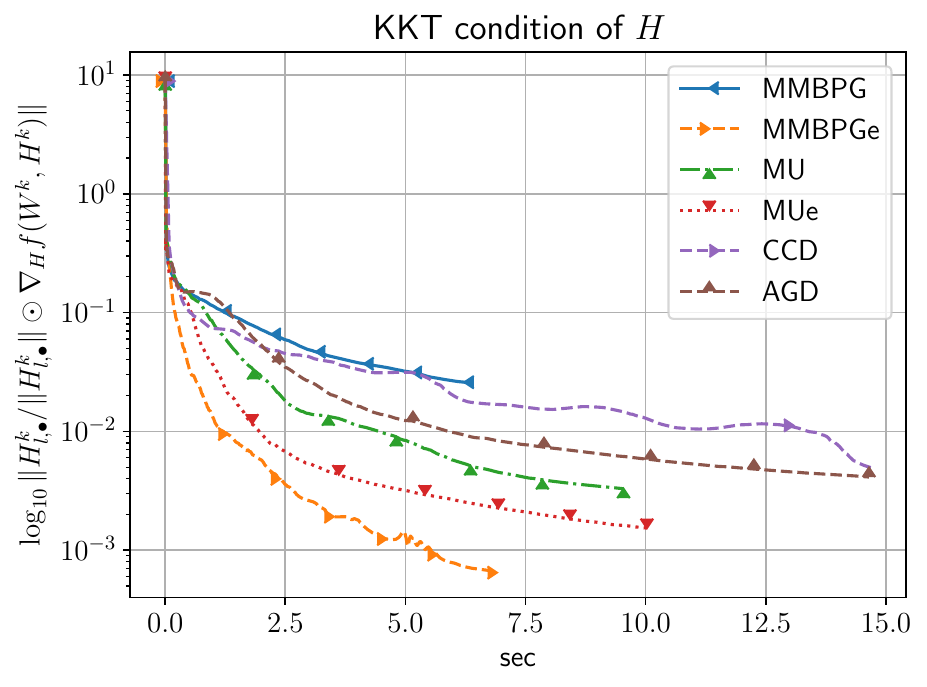}
        \subcaption{The KKT condition with respect to $\H$ with the time axis.}
        \label{fig:nmf-x0scaled-200-200-30-kkth-time}
    \end{minipage}
    \caption{Comparison with MMBPG (blue), MMBPGe (orange), MU (green), MUe (red), CCD (purple), and AGD (brown) on KL-NMF ($(m,n,r) = (200, 200, 30)$) from the scaled initial point.
    The left column corresponds to the iteration axis, and the right column corresponds to the time axis.}
    \label{fig:nmf-x0scaled-200-200-30}
\end{figure}

\begin{figure}[!pth]
    \centering
    \begin{minipage}[b]{0.49\linewidth}
        \centering
        \includegraphics[width=\textwidth]{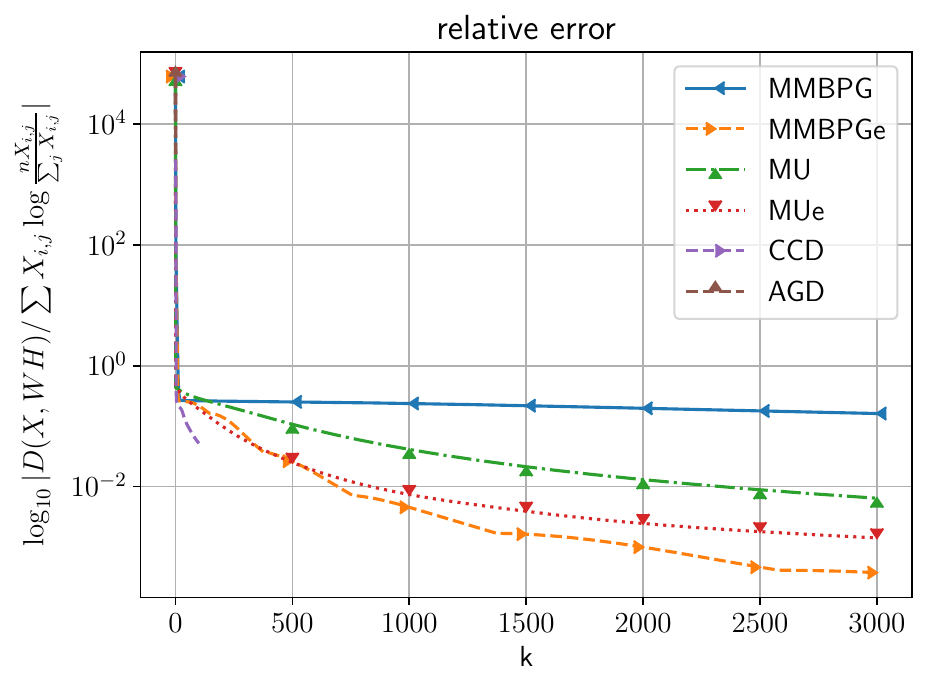}
        \subcaption{The relative error values with the iteration axis.}
        \label{fig:nmf-x0unscaled-500-500-80-rel-iter}
    \end{minipage}
    \hfill
    \begin{minipage}[b]{0.49\linewidth}
        \centering
        \includegraphics[width=\textwidth]{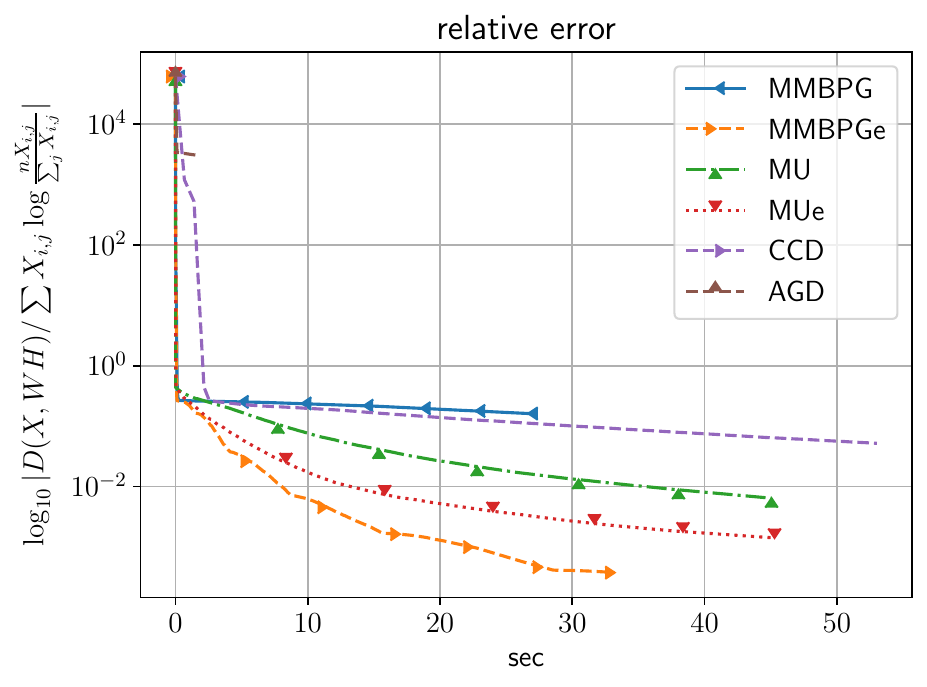}
        \subcaption{The relative error values with the time axis.}
        \label{fig:nmf-x0unscaled-500-500-80-rel-time}
    \end{minipage}
    \begin{minipage}[b]{0.49\linewidth}
        \centering
        \includegraphics[width=\textwidth]{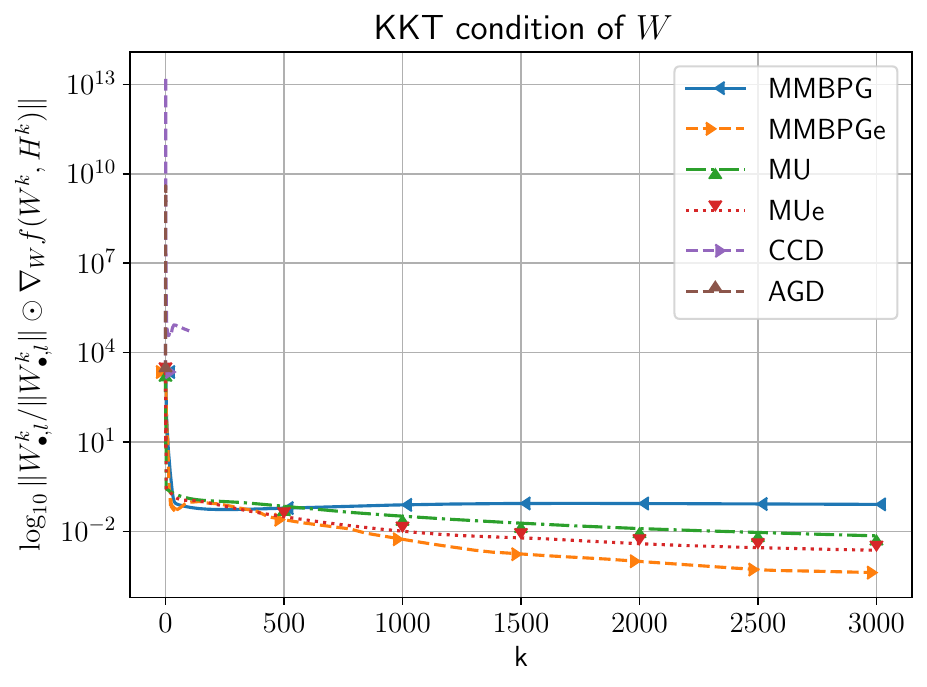}
        \subcaption{The KKT condition with respect to $\W$ with the iteration axis.}
        \label{fig:nmf-x0unscaled-500-500-80-kktw-iter}
    \end{minipage}
    \hfill
    \begin{minipage}[b]{0.49\linewidth}
        \centering
        \includegraphics[width=\textwidth]{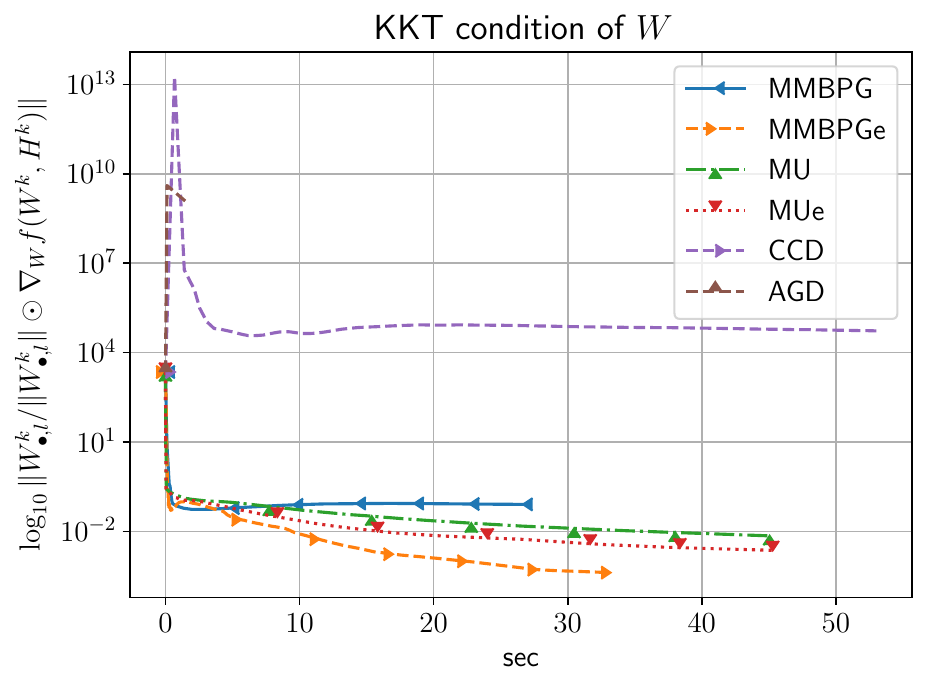}
        \subcaption{The KKT condition with respect to $\W$ with the time axis.}
        \label{fig:nmf-x0unscaled-500-500-80-kktw-time}
    \end{minipage}
    \begin{minipage}[b]{0.49\linewidth}
        \centering
        \includegraphics[width=\textwidth]{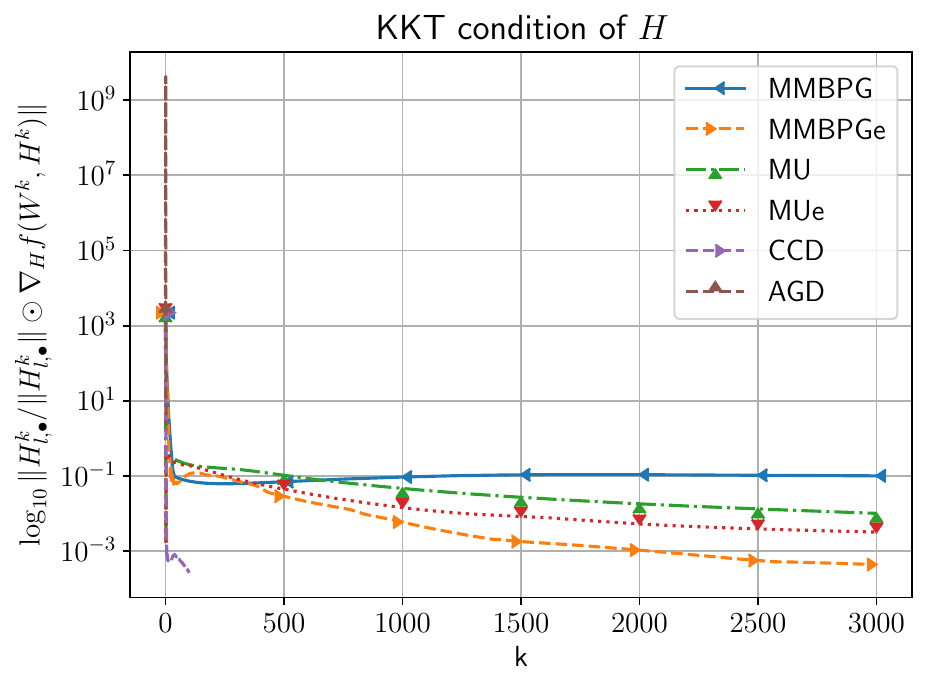}
        \subcaption{The KKT condition with respect to $\H$ with the iteration axis.}
        \label{fig:nmf-x0unscaled-500-500-80-kkth-iter}
    \end{minipage}
    \hfill
    \begin{minipage}[b]{0.49\linewidth}
        \centering
        \includegraphics[width=\textwidth]{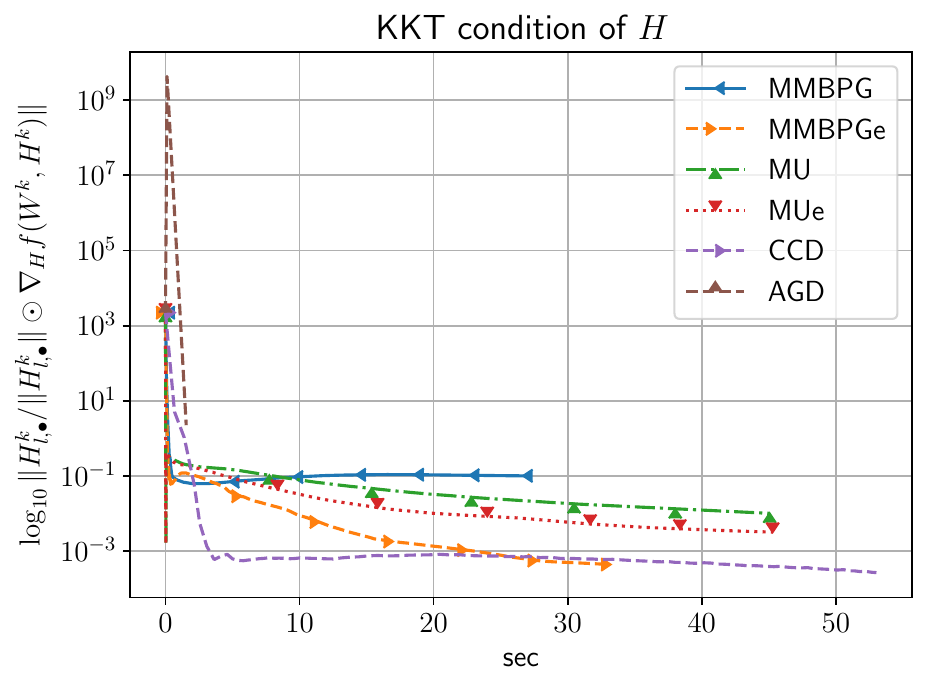}
        \subcaption{The KKT condition with respect to $\H$ with the time axis.}
        \label{fig:nmf-x0unscaled-500-500-80-kkth-time}
    \end{minipage}
    \caption{Comparison with MMBPG (blue), MMBPGe (orange), MU (green), MUe (red), CCD (purple), and AGD (brown) on KL-NMF ($(m,n,r) = (500, 500, 80)$) from the unscaled initial point.
    The left column corresponds to the iteration axis, and the right column corresponds to the time axis.}
    \label{fig:nmf-x0unscaled-500-500-80}
\end{figure}

\begin{figure}[!pth]
    \centering
    \begin{minipage}[b]{0.49\linewidth}
        \centering
        \includegraphics[width=\textwidth]{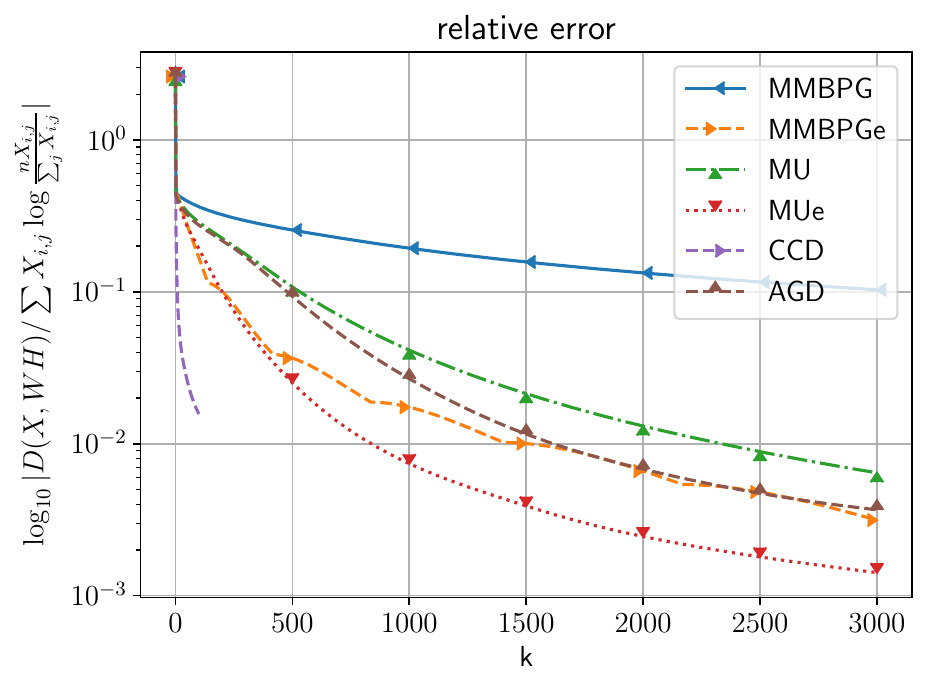}
        \subcaption{The relative error values with the iteration axis.}
        \label{fig:nmf-x0scaled-500-500-80-rel-iter}
    \end{minipage}
    \hfill
    \begin{minipage}[b]{0.49\linewidth}
        \centering
        \includegraphics[width=\textwidth]{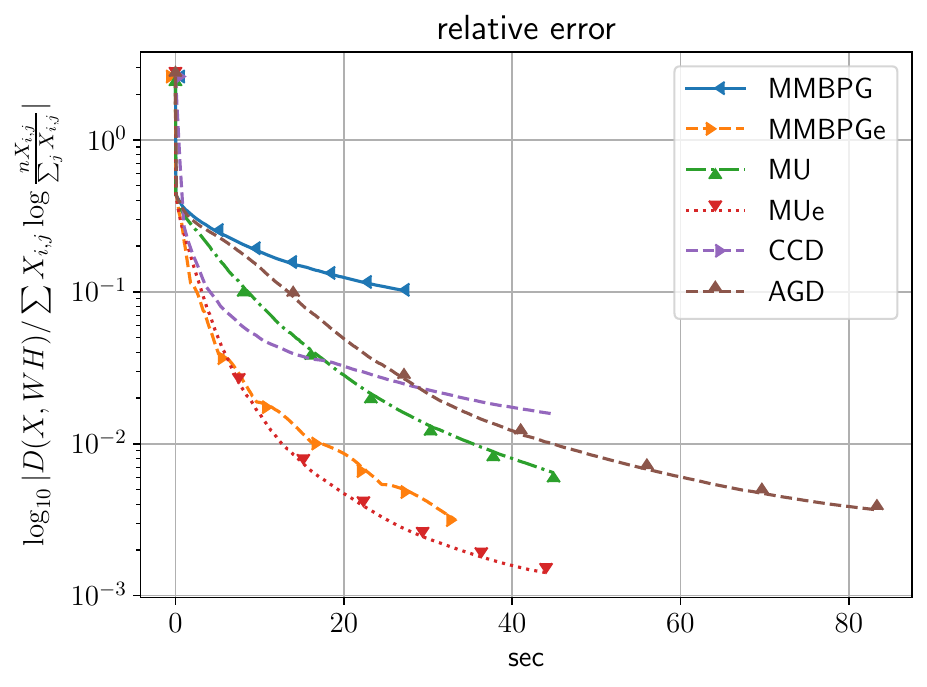}
        \subcaption{The relative error values with the time axis.}
        \label{fig:nmf-x0scaled-500-500-80-rel-time}
    \end{minipage}
    \begin{minipage}[b]{0.49\linewidth}
        \centering
        \includegraphics[width=\textwidth]{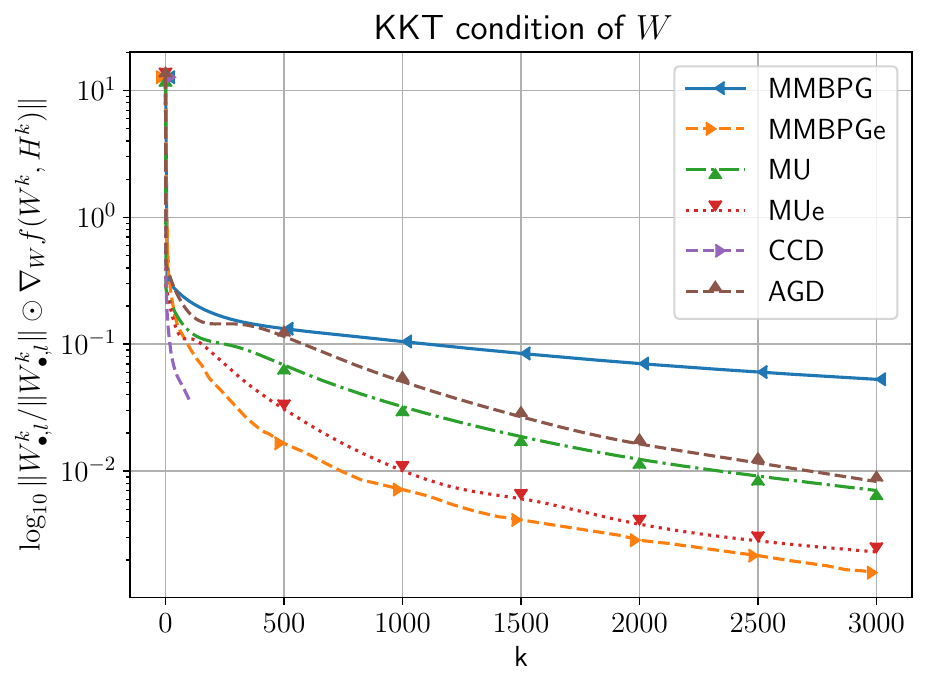}
        \subcaption{The KKT condition with respect to $\W$ with the iteration axis.}
        \label{fig:nmf-x0scaled-500-500-80-kktw-iter}
    \end{minipage}
    \hfill
    \begin{minipage}[b]{0.49\linewidth}
        \centering
        \includegraphics[width=\textwidth]{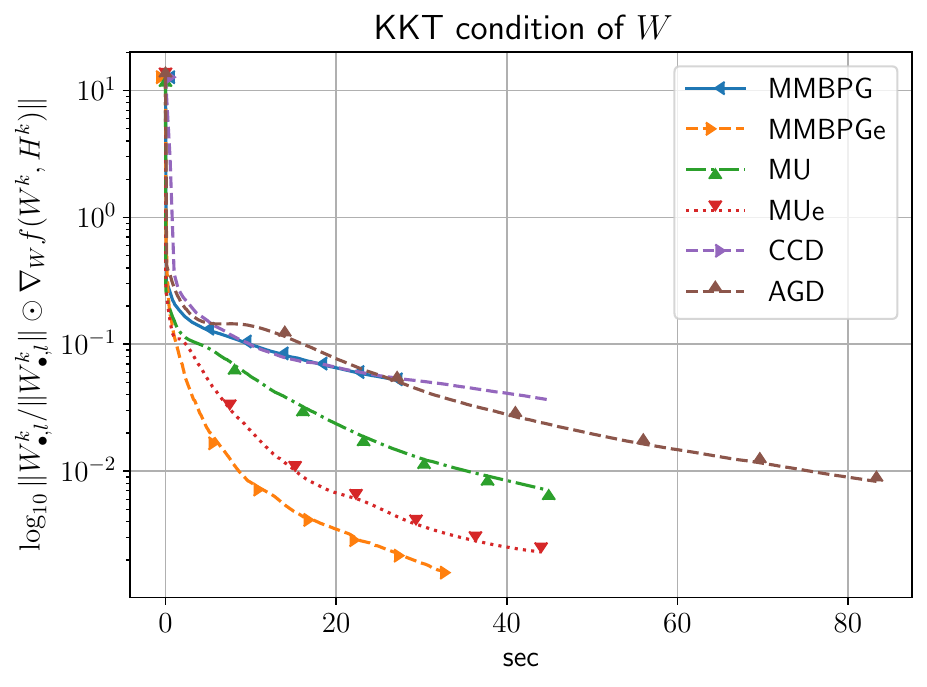}
        \subcaption{The KKT condition with respect to $\W$ with the time axis.}
        \label{fig:nmf-x0scaled-500-500-80-kktw-time}
    \end{minipage}
    \begin{minipage}[b]{0.49\linewidth}
        \centering
        \includegraphics[width=\textwidth]{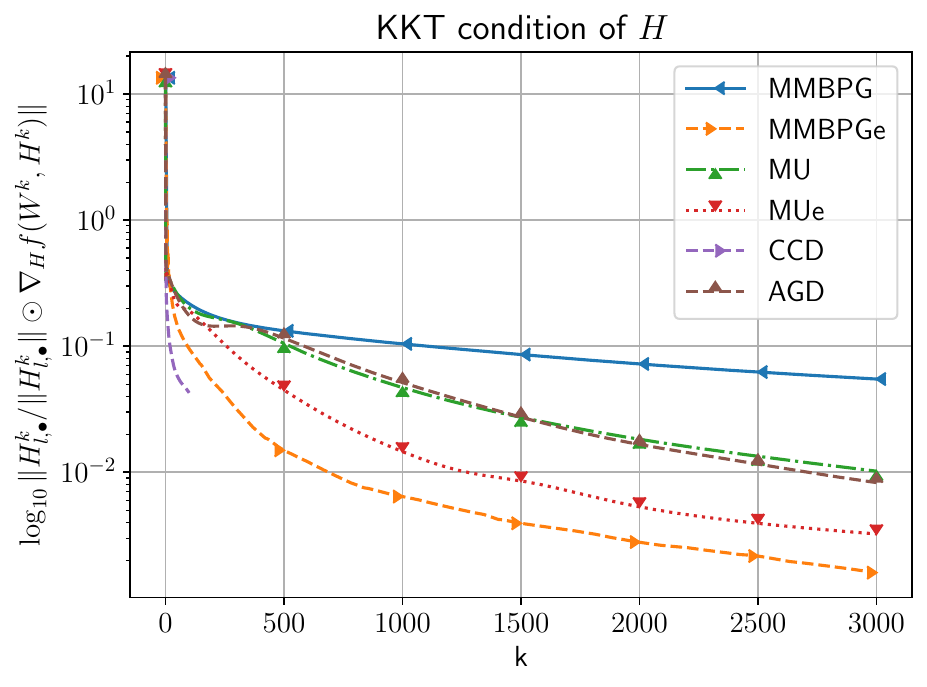}
        \subcaption{The KKT condition with respect to $\H$ with the iteration axis.}
        \label{fig:nmf-x0scaled-500-500-80-kkth-iter}
    \end{minipage}
    \hfill
    \begin{minipage}[b]{0.49\linewidth}
        \centering
        \includegraphics[width=\textwidth]{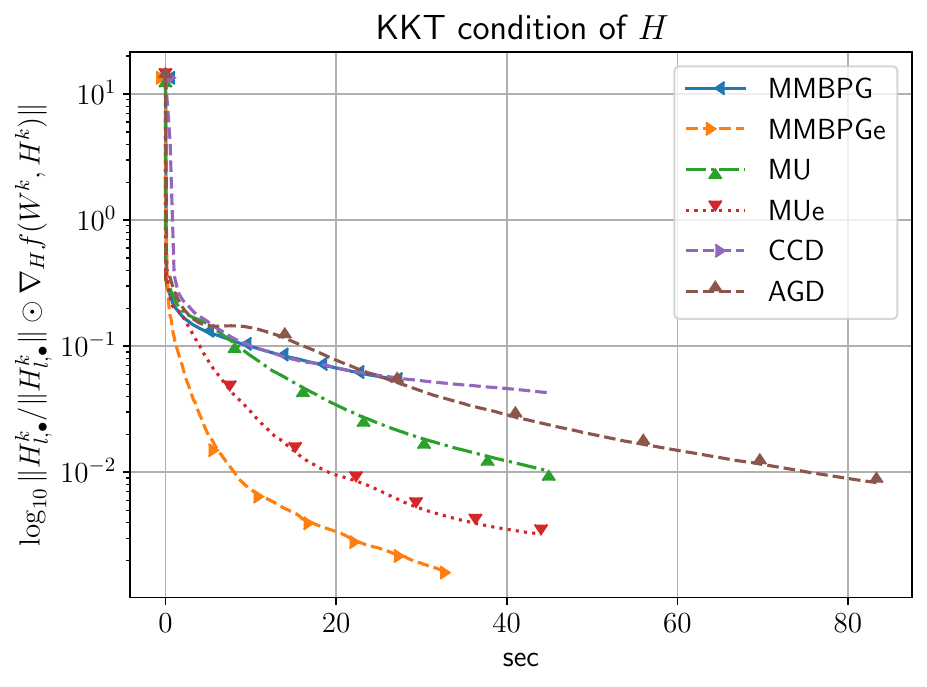}
        \subcaption{The KKT condition with respect to $\H$ with the time axis.}
        \label{fig:nmf-x0scaled-500-500-80-kkth-time}
    \end{minipage}
    \caption{Comparison with MMBPG (blue), MMBPGe (orange), MU (green), MUe (red), CCD (purple), and AGD (brown) on KL-NMF ($(m,n,r) = (500, 500, 80)$) from the scaled initial point.
    The left column corresponds to the iteration axis, and the right column corresponds to the time axis.}
    \label{fig:nmf-x0scaled-500-500-80}
\end{figure}

In addition, Table~\ref{tab:unscaled} shows average performance over 20 different instances with the unscaled initial point, and Table~\ref{tab:scaled} shows it with the scaled initial point.
When algorithms start with the unscaled initial point, MMBPGe has the best performance (Table~\ref{tab:unscaled}).
On the other hand, when algorithms start with the scaled initial point, MUe has the best performance on the relative error, while MMBPGe has the best performance on the norm of the normalized KKT conditions. 
AGD with the unscaled initial point always terminated at the 2nd iteration because of the lack of Lipschitz continuity of $\nabla f$ for all instances.
Although CCD has good performance, it requires more computational time than MMBPG(e) and MU(e) due to nested iterations.
AGD also requires computational time because the line search procedure requires additional evaluations.

\begin{table}[!tpb]
\caption{Average number of iterations, relative error, norm value of KKT conditions, and computational time (sec) over 20 different instances for each size from the unscaled initial points.}
\label{tab:unscaled}
\centering
\begin{tabular}{llllrllll}
\toprule
$m$ & $n$ & $r$ & algorithm & iter & rel & KKT ($\W$) & KKT ($\H$) & time \\
\midrule
\multirow[c]{6}{*}{200} & \multirow[c]{6}{*}{200} & \multirow[c]{6}{*}{30} & MMBPG & 3000 & 3.70e-02 & 5.92e-03 & 7.15e-03 & 3.71e+00 \\
 &  &  & MMBPGe & 3000 & \textbf{1.06e-04} & \textbf{2.27e-05} & 2.91e-05 & 4.81e+00 \\
 &  &  & MU & 3000 & 1.59e-03 & 3.77e-04 & 5.36e-04 & 4.74e+00 \\
 &  &  & MUe & 3000 & 4.76e-04 & 2.04e-04 & 2.97e-04 & 4.78e+00 \\
 &  &  & CCD & 600 & 1.58e-03 & 3.36e+02 & \textbf{1.03e-05} & 8.74e+00 \\
 &  &  & AGD & 2 & 1.25e+03 & 1.28e+08 & 9.74e-01 & 2.48e-01 \\\midrule
\multirow[c]{6}{*}{500} & \multirow[c]{6}{*}{500} & \multirow[c]{6}{*}{80} & MMBPG & 3000 & 1.92e-01 & 8.00e-02 & 8.96e-02 & 1.47e+01 \\
 &  &  & MMBPGe & 3000 & \textbf{3.29e-04} & \textbf{3.79e-04} & \textbf{4.59e-04} & 1.79e+01 \\
 &  &  & MU & 3000 & 6.08e-03 & 7.01e-03 & 9.84e-03 & 2.01e+01 \\
 &  &  & MUe & 3000 & 1.33e-03 & 2.29e-03 & 3.16e-03 & 2.04e+01 \\
 &  &  & CCD & 100 & 4.93e-02 & 5.39e+04 & 8.45e-04 & 3.97e+01 \\
 &  &  & AGD & 2 & 2.94e+03 & 1.21e+09 & 8.68e+00 & 1.36e+00 \\
\bottomrule
\end{tabular}
\end{table}
\begin{table}[!tpb]
\caption{Average number of iterations, relative error, norm value of KKT conditions, and computational time (sec) over 20 different instances for each size from the scaled initial points.}
\label{tab:scaled}
\centering
\begin{tabular}{llllrllll}
\toprule
$m$ & $n$ & $r$ & algorithm & iter & rel & KKT ($\W$) & KKT ($\H$) & time \\
\midrule
\multirow[c]{6}{*}{200} & \multirow[c]{6}{*}{200} & \multirow[c]{6}{*}{30} & MMBPG & 3000 & 5.07e-02 & 2.51e-02 & 2.76e-02 & 3.40e+00 \\
 &  &  & MMBPGe & 3000 & 1.26e-03 & \textbf{7.10e-04} & \textbf{7.69e-04} & 4.25e+00 \\
 &  &  & MU & 3000 & 1.43e-03 & 2.05e-03 & 2.94e-03 & 4.03e+00 \\
 &  &  & MUe & 3000 & \textbf{4.06e-04} & 1.06e-03 & 1.53e-03 & 4.05e+00 \\
 &  &  & CCD & 600 & 5.28e-04 & 6.42e-03 & 5.66e-03 & 7.20e+00 \\
 &  &  & AGD & 3000 & 1.30e-03 & 3.12e-03 & 3.32e-03 & 9.35e+00 \\\midrule
\multirow[c]{6}{*}{500} & \multirow[c]{6}{*}{500} & \multirow[c]{6}{*}{80} & MMBPG & 3000 & 1.02e-01 & 5.43e-02 & 5.74e-02 & 1.28e+01 \\
 &  &  & MMBPGe & 3000 & 3.01e-03 & \textbf{1.53e-03} & \textbf{1.61e-03} & 1.69e+01 \\
 &  &  & MU & 3000 & 6.11e-03 & 6.90e-03 & 9.81e-03 & 1.88e+01 \\
 &  &  & MUe & 3000 & \textbf{1.35e-03} & 2.29e-03 & 3.15e-03 & 1.93e+01 \\
 &  &  & CCD & 100 & 1.56e-02 & 3.97e-02 & 3.71e-02 & 2.47e+01 \\
 &  &  & AGD & 3000 & 3.94e-03 & 9.11e-03 & 9.18e-03 & 5.42e+01 \\
\bottomrule
\end{tabular}
\end{table}

Next, we consider $g(\W, \H) = \mu_{\W}\|\W\|_1 + \mu_{\H}\|\H\|_1$ setting $\mu_{\W} = \mu_{\H} = 10^{-4}$.
We compare MMBPG and MMBPGe with MU (also called MM-SNMF-$\ell_1$)~\cite{Marmin2023-kw}, MUe, and AGD.
Because CCD requires the gradient and the Hessian matrix of the objective function for its update, CCD was not applied to this problem.
MUe was implemented as the MU~(MM-SNMF-$\ell_1$)~\cite{Marmin2023-kw} with extrapolation, which is the same as~\cite{Hien2025-vc}.
Note that MUe is not guaranteed to converge for $g \not \equiv 0$.
In this case, AGD is the alternating proximal gradient algorithm with line search.
The number of nonzero elements of $(\W^*, \H^*)$ is 10\% of all elements.
The initial point is scaled.
Other settings are the same.
We use, instead of the relative error, the objective function value because the objective function includes $g$.
Figures~\ref{fig:nmf-x0unscaled-200-200-30-sp} and~\ref{fig:nmf-x0scaled-200-200-30-sp} show that MMBPGe outperforms other algorithms on the objective function value and the norm value of the KKT conditions.
Table~\ref{tab:resetting-l1} shows the number of times $\theta_k$ and $\beta_k$ were reset in MMBPGe. The restart frequency is low at approximately 0.8\% for the unscaled initial point case and 0.6\% for the scaled initial point case in Table~\ref{tab:resetting-l1}.
Moreover, Table~\ref{tab:unscaled-sp} shows average performance over 20 different instances with the unscaled initial point, and Table~\ref{tab:scaled-sp} shows it with the scaled initial point.
In both cases, MMBPGe has the best performance.
Therefore, for $g(\W, \H) = \mu_{\W}\|\W\|_1 + \mu_{\H}\|\H\|_1$, MMBPGe outperforms other existing algorithms.

\begin{table}[!tpb]
\caption{The frequency of occurrence of $\Y^k \not\in\R_{++}^{m \times r} \times \R_{++}^{r \times n}$ and \eqref{ineq:adaptive-restart} for MMBPGe in Figures~\ref{fig:nmf-x0unscaled-200-200-30-sp} and~\ref{fig:nmf-x0scaled-200-200-30-sp}.}
\label{tab:resetting-l1}
\centering
\begin{tabular}{lllcrr}
\toprule
$m$ & $n$ & $r$ & scaled & nonpositive & \eqref{ineq:adaptive-restart}\\ \midrule
200 & 200 & 30 & \xmarkr & 1 & 26 \\
200 & 200 & 30 & \cmarkg & 0 & 18 \\
\bottomrule
\end{tabular}
\end{table}

\begin{figure}[!pth]
    \centering
    \begin{minipage}[b]{0.49\linewidth}
        \centering
        \includegraphics[width=\textwidth]{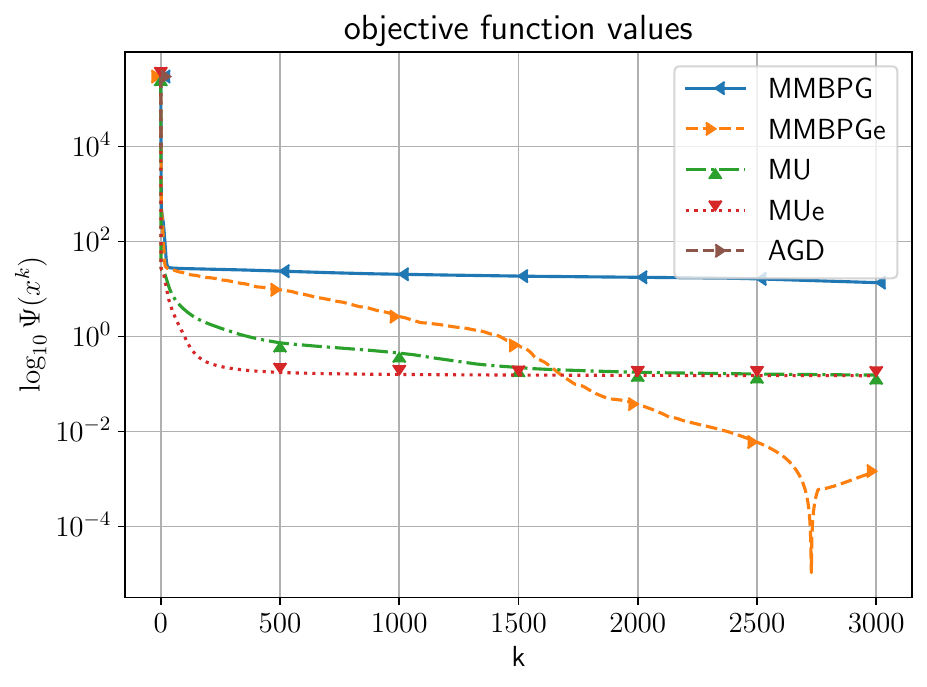}
        \subcaption{The objective function error values with the iteration axis.}
        \label{fig:nmf-x0unscaled-200-200-30-obj-iter-sp}
    \end{minipage}
    \hfill
    \begin{minipage}[b]{0.49\linewidth}
        \centering
        \includegraphics[width=\textwidth]{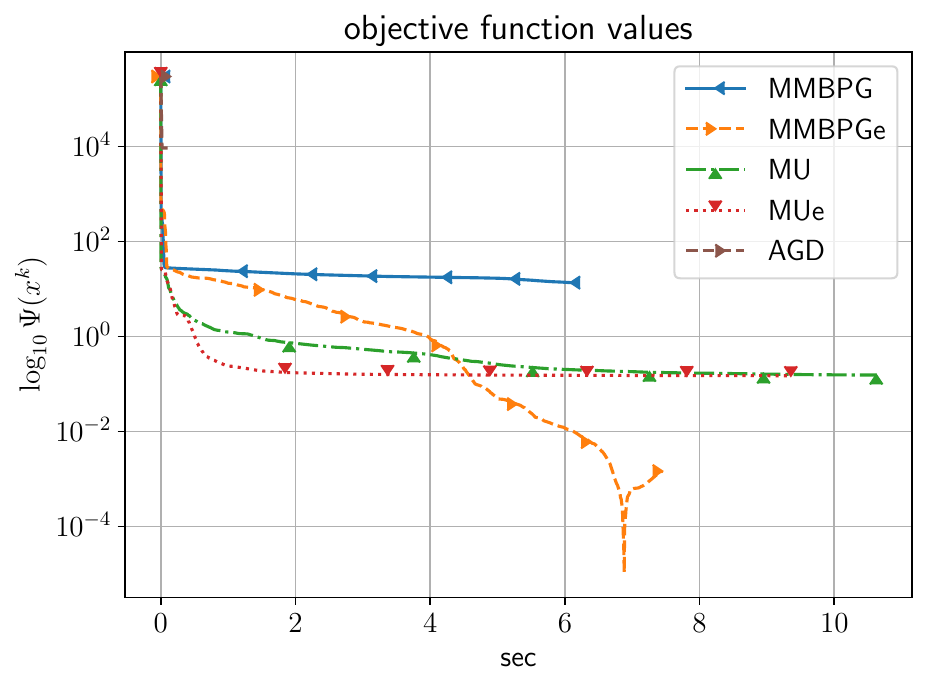}
        \subcaption{The objective function values with the time axis.}
        \label{fig:nmf-x0unscaled-200-200-30-rel-time-sp}
    \end{minipage}
    \begin{minipage}[b]{0.49\linewidth}
        \centering
        \includegraphics[width=\textwidth]{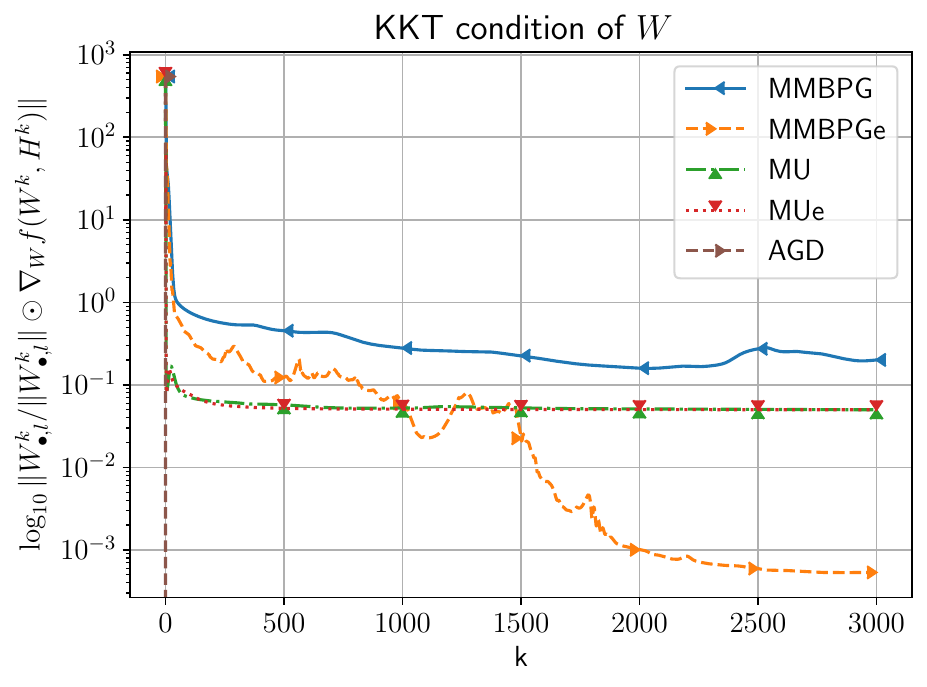}
        \subcaption{The KKT condition with respect to $\W$ with the iteration axis.}
        \label{fig:nmf-x0unscaled-200-200-30-kktw-iter-sp}
    \end{minipage}
    \hfill
    \begin{minipage}[b]{0.49\linewidth}
        \centering
        \includegraphics[width=\textwidth]{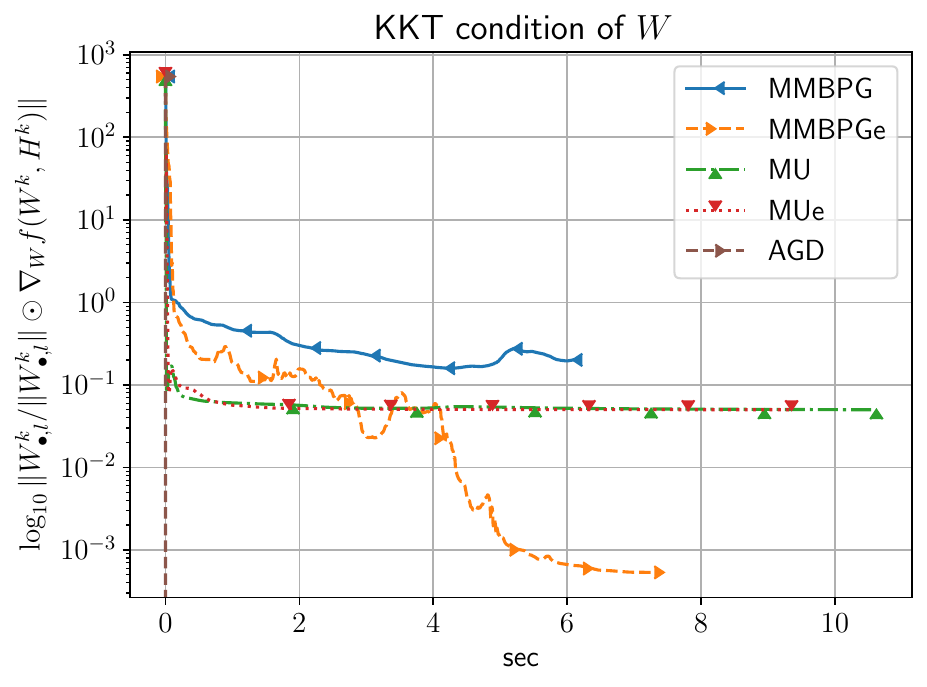}
        \subcaption{The KKT condition with respect to $\W$ with the time axis.}
        \label{fig:nmf-x0unscaled-200-200-30-kktw-time-sp}
    \end{minipage}
    \begin{minipage}[b]{0.49\linewidth}
        \centering
        \includegraphics[width=\textwidth]{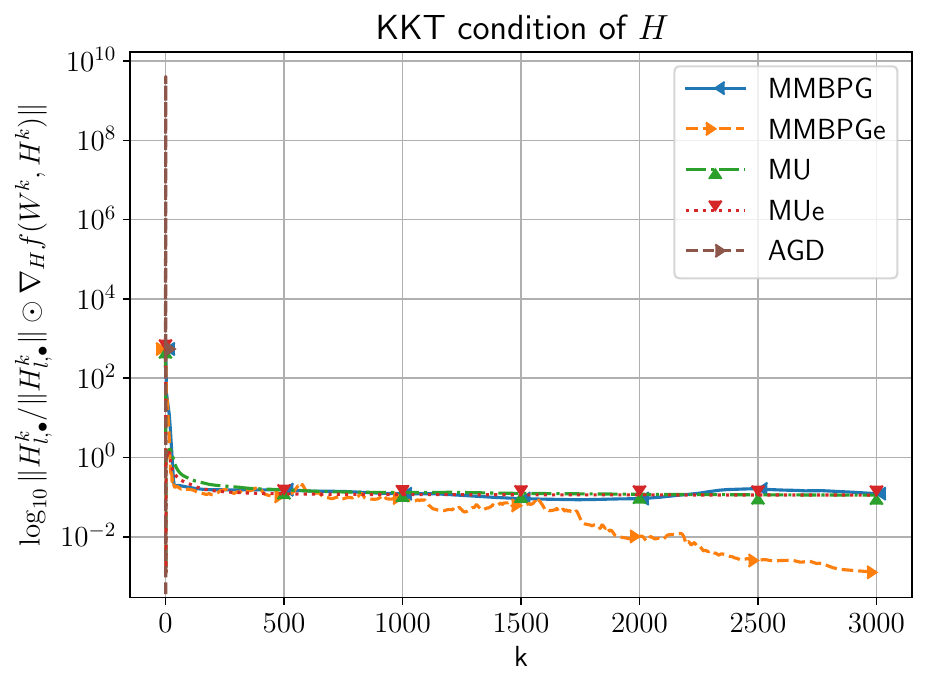}
        \subcaption{The KKT condition with respect to $\H$ with the iteration axis.}
        \label{fig:nmf-x0unscaled-200-200-30-kkth-iter-sp}
    \end{minipage}
    \hfill
    \begin{minipage}[b]{0.49\linewidth}
        \centering
        \includegraphics[width=\textwidth]{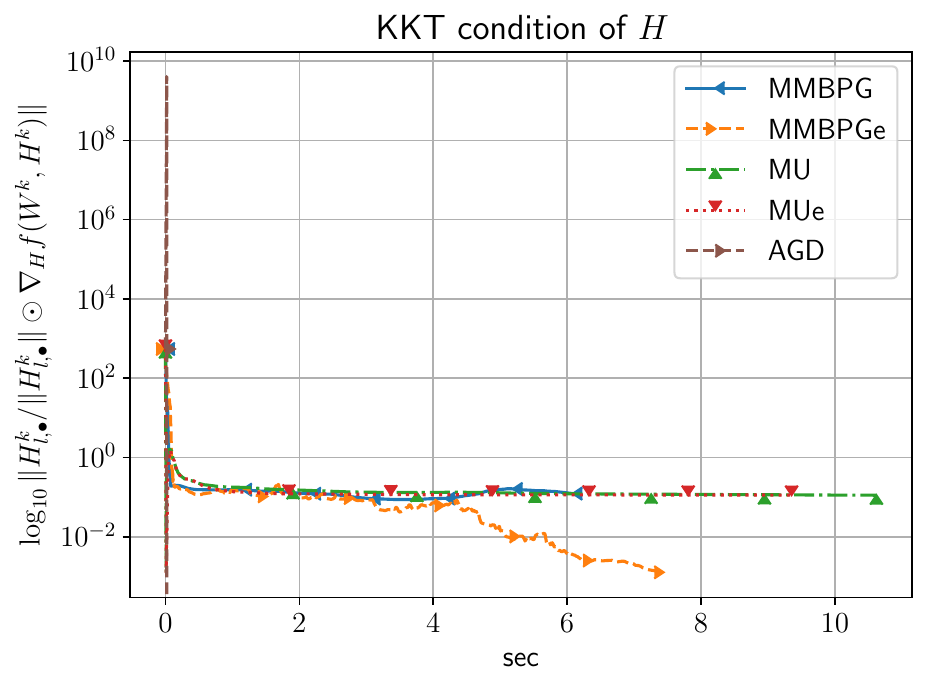}
        \subcaption{The KKT condition with respect to $\H$ with the time axis.}
        \label{fig:nmf-x0unscaled-200-200-30-kkth-time-sp}
    \end{minipage}
    \caption{Comparison with MMBPG (blue), MMBPGe (orange), MU (green), MUe (red), and AGD (brown) on KL-NMF ($(m,n,r) = (200, 200, 30)$) from the unscaled initial point.
    The left column corresponds to the iteration axis, and the right column corresponds to the time axis.}
    \label{fig:nmf-x0unscaled-200-200-30-sp}
\end{figure}

\begin{figure}[!pth]
    \centering
    \begin{minipage}[b]{0.49\linewidth}
        \centering
        \includegraphics[width=\textwidth]{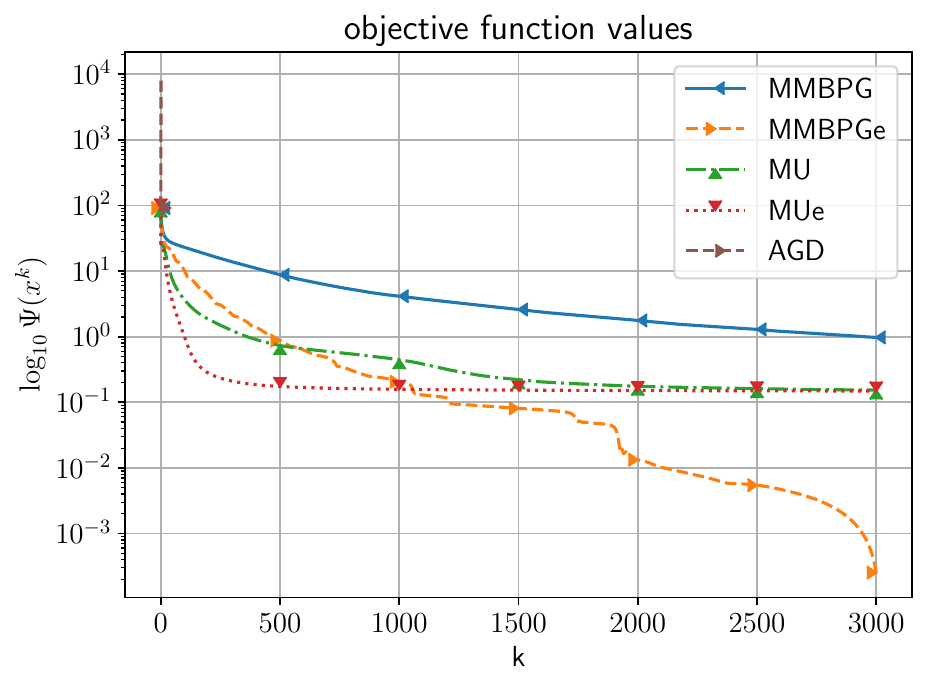}
        \subcaption{The objective function error values with the iteration axis.}
        \label{fig:nmf-x0scaled-200-200-30-obj-iter-sp}
    \end{minipage}
    \hfill
    \begin{minipage}[b]{0.49\linewidth}
        \centering
        \includegraphics[width=\textwidth]{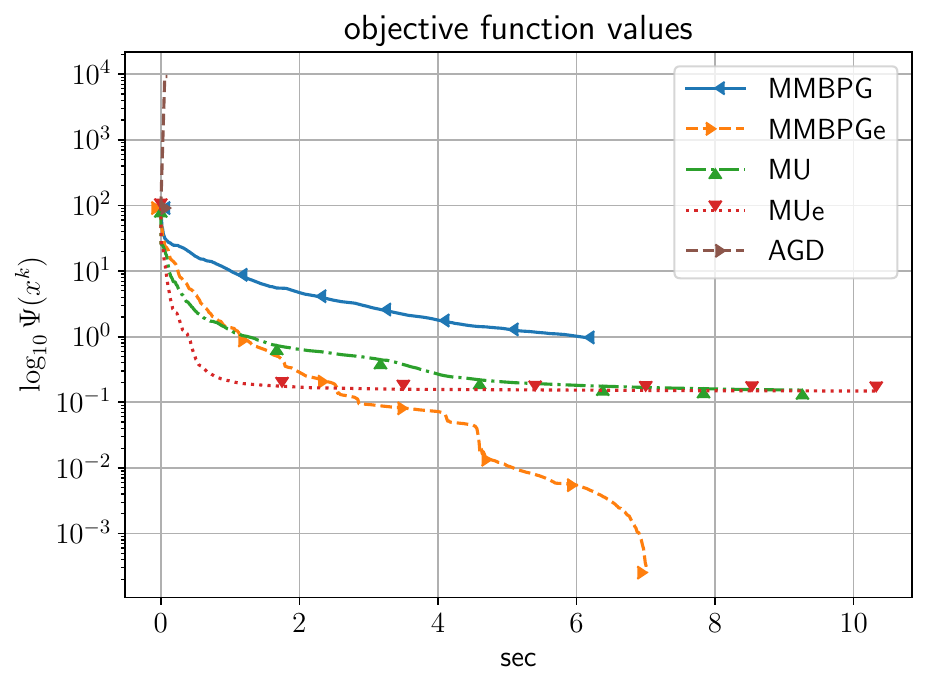}
        \subcaption{The objective function values with the time axis.}
        \label{fig:nmf-x0scaled-200-200-30-rel-time-sp}
    \end{minipage}
    \begin{minipage}[b]{0.49\linewidth}
        \centering
        \includegraphics[width=\textwidth]{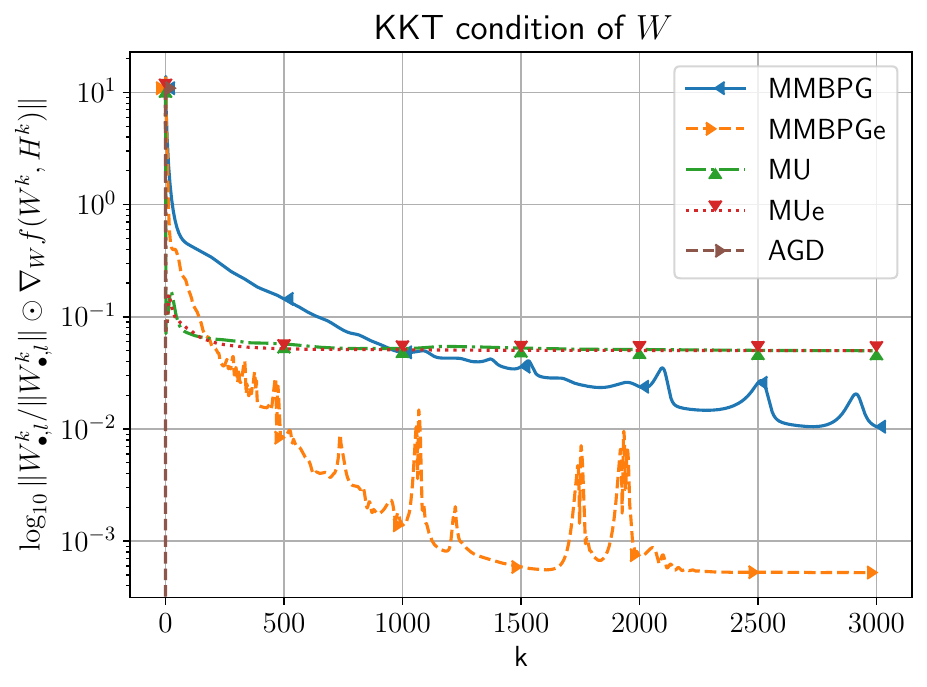}
        \subcaption{The KKT condition with respect to $\W$ with the iteration axis.}
        \label{fig:nmf-x0scaled-200-200-30-kktw-iter-sp}
    \end{minipage}
    \hfill
    \begin{minipage}[b]{0.49\linewidth}
        \centering
        \includegraphics[width=\textwidth]{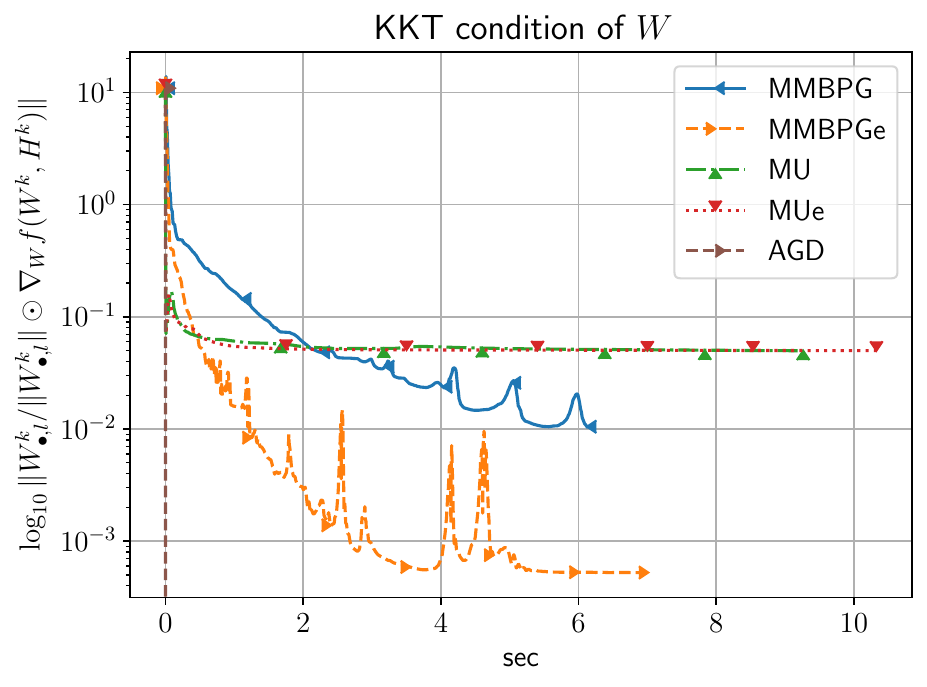}
        \subcaption{The KKT condition with respect to $\W$ with the time axis.}
        \label{fig:nmf-x0scaled-200-200-30-kktw-time-sp}
    \end{minipage}
    \begin{minipage}[b]{0.49\linewidth}
        \centering
        \includegraphics[width=\textwidth]{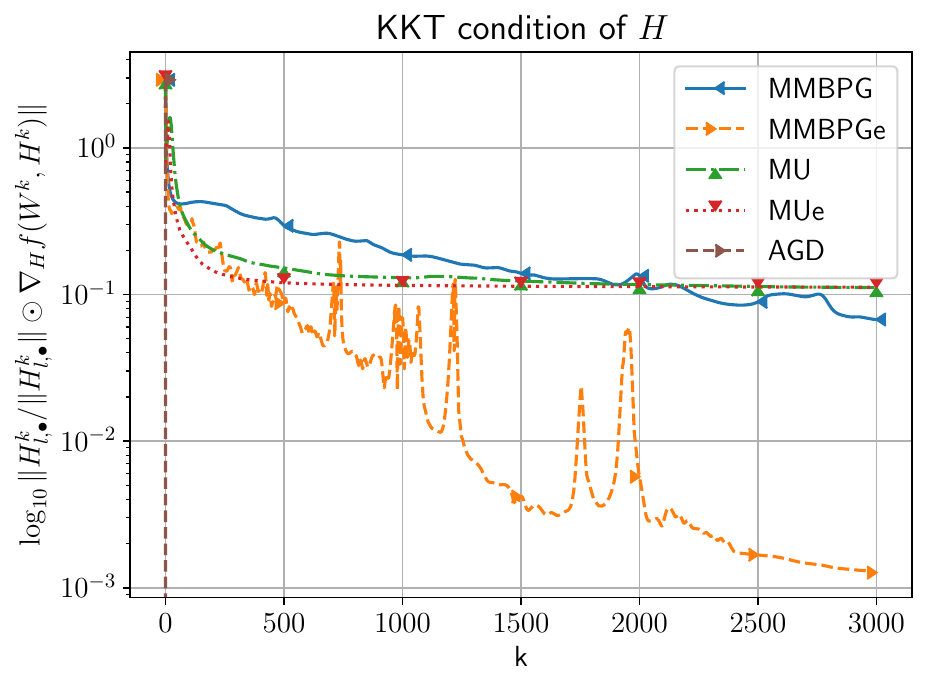}
        \subcaption{The KKT condition with respect to $\H$ with the iteration axis.}
        \label{fig:nmf-x0scaled-200-200-30-kkth-iter-sp}
    \end{minipage}
    \hfill
    \begin{minipage}[b]{0.49\linewidth}
        \centering
        \includegraphics[width=\textwidth]{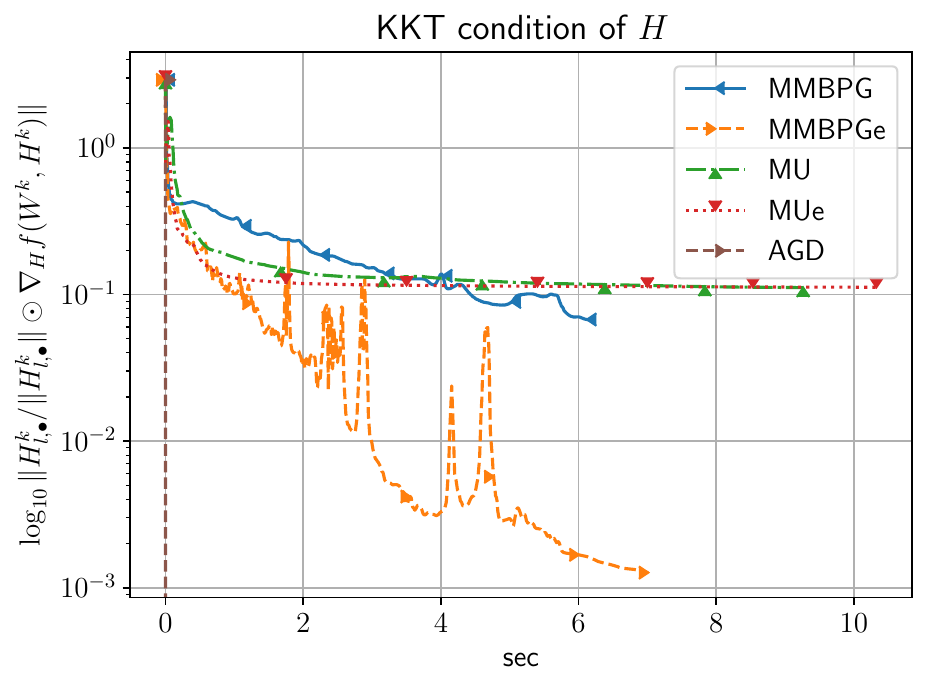}
        \subcaption{The KKT condition with respect to $\H$ with the time axis.}
        \label{fig:nmf-x0scaled-200-200-30-kkth-time-sp}
    \end{minipage}
    \caption{Comparison with MMBPG (blue), MMBPGe (orange), MU (green), MUe (red), and AGD (brown) on KL-NMF ($(m,n,r) = (200, 200, 30)$) from the scaled initial point.
    The left column corresponds to the iteration axis, and the right column corresponds to the time axis.}
    \label{fig:nmf-x0scaled-200-200-30-sp}
\end{figure}

\begin{table}[!tpb]
\caption{Average number of iterations, relative error, norm value of KKT conditions, and computational time (sec) over 20 different instances for each size from the unscaled initial points when $g(\W, \H) = \mu_{\W}\|\W\|_1 + \mu_{\H}\|\H\|_1$. The norm of the KKT conditions of AGD is NaN.}
\label{tab:unscaled-sp}
\centering
\begin{tabular}{llllrllll}
\toprule
$m$ & $n$ & $r$ & algorithm & iter & rel & KKT ($\W$) & KKT ($\H$) & time \\
\midrule
\multirow[c]{5}{*}{200} & \multirow[c]{5}{*}{200} & \multirow[c]{5}{*}{30} & MMBPG & 3000 & 6.17e-01 & 1.79e-01 & 1.21e-01 & 3.22e+00 \\
     &  &  & MMBPGe & 3000 & \textbf{1.82e-03} & \textbf{6.91e-04} & \textbf{5.38e-03} & 4.14e+00 \\
 &  &  & MU & 3000 & 7.03e-03 & 4.42e-02 & 9.98e-02 & 3.84e+00 \\
 &  &  & MUe & 3000 & 6.01e-03 & 4.65e-02 & 1.05e-01 & 3.89e+00 \\
 &  &  & AGD & 3 & 3.43e+02 & -- & -- & 8.67e-02 \\
\bottomrule
\end{tabular}
\end{table}
\begin{table}[!tpb]
\caption{Average number of iterations, relative error, norm value of KKT conditions, and computational time (sec) over 20 different instances for each size from the scaled initial points when $g(\W, \H) = \mu_{\W}\|\W\|_1 + \mu_{\H}\|\H\|_1$. The norm of the KKT conditions of AGD is NaN.}
\label{tab:scaled-sp}
\centering
\begin{tabular}{llllrllll}
\toprule
$m$ & $n$ & $r$ & algorithm & iter & rel & KKT ($\W$) & KKT ($\H$) & time \\\midrule
\multirow[c]{5}{*}{200} & \multirow[c]{5}{*}{200} & \multirow[c]{5}{*}{30} & MMBPG & 3000 & 5.39e-02 & 1.43e-02 & 7.37e-02 & 3.18e+00 \\
 &  &  & MMBPGe & 3000 & \textbf{1.81e-03} & \textbf{6.99e-04} & \textbf{2.96e-03} & 4.14e+00 \\
 &  &  & MU & 3000 & 4.83e-03 & 4.04e-02 & 8.78e-02 & 3.84e+00 \\
 &  &  & MUe & 3000 & 4.57e-03 & 4.00e-02 & 8.69e-02 & 4.07e+00 \\
 &  &  & AGD & 2 & 3.45e+02 & -- & -- & 5.53e-02 \\
\bottomrule
\end{tabular}
\end{table}

\subsection{MovieLens}
Next, we consider real-world data, the MovieLens 100K Dataset (\url{https://grouplens.org/datasets/movielens/}).
The size of KL-NMF~\eqref{prob:klnmf} is $(m, n, r) = (9724, 610, 20)$.
We generate the initial point $(\W^0, \H^0)$ from i.i.d. uniform distribution and scaled it by $(\alpha\W^0, \alpha\H^0)$ with $\alpha = \sqrt{\frac{\sum_{i,j}X_{i,j}}{\sum_{i,j}(\W^0\H^0)_{ij}}}$~\cite{Hien2021-sh}.
We set $g(\W, \H) = \mu_{\W}\|\W\|_1 + \mu_{\H}\|\H\|_1$ with $\mu_{\W} = \mu_{\H} = 0.5$.
We compared MMBPG and MMBPGe with MU and MUe because CCD requires the gradient and the Hessian matrix of the objective function for its update, and AGD did not converge.
To accelerate MMBPG and MMBPGe, we use $(\lambda_{k,1}, \lambda_{k,2}) = (10/3L_{k,1}, 10/L_{k,2})$ from~\eqref{subprob:mmbpg-kl-nmf-h} and~\eqref{subprob:mmbpg-kl-nmf-w}, \ie, $\lambda_{k,i} L_{k,i} < 1$, $i = 1,2$ do not hold.
The maximum iteration is 6000.
Figure~\ref{fig:nmfmovielens-x0scaled-obj-iter} shows that the objective function value of MMBPGe at the 6000th iteration is the smallest.
Moreover, MMBPGe recovered $\W$ and $\H$ with the minimum norm value of the KKT condition (Figures~\ref{fig:nmfmovielens-x0scaled-kktw-iter} and~\ref{fig:nmfmovielens-x0scaled-kkth-iter}). 
Overall, MMBPGe shows stable and better performance.
Note that the frequency of occurrence of $\Y^k \not\in\R_{++}^{m \times r} \times \R_{++}^{r \times n}$ is 4 times and that of~\eqref{ineq:adaptive-restart} is 60 times for MMBPGe. The restart frequency is low at approximately 1\%.

\begin{figure}[!pth]
    \centering
    \begin{minipage}[b]{0.49\linewidth}
        \centering
        \includegraphics[width=\textwidth]{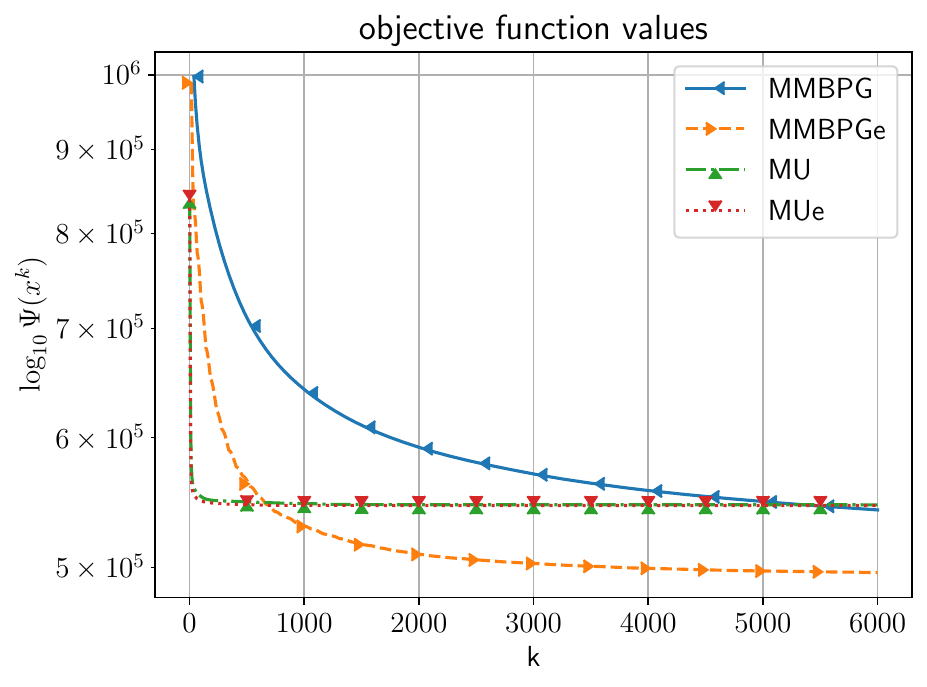}
        \subcaption{The objective function values with the iteration axis.}
        \label{fig:nmfmovielens-x0scaled-obj-iter}
    \end{minipage}
    \hfill
    \begin{minipage}[b]{0.49\linewidth}
        \centering
        \includegraphics[width=\textwidth]{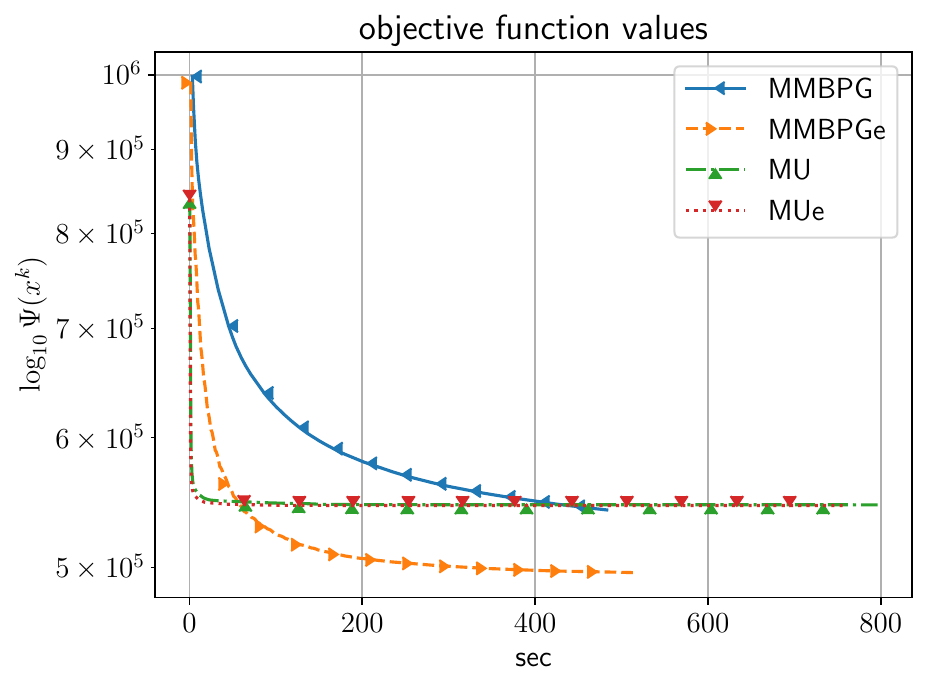}
        \subcaption{The objective function values with the time axis.}
        \label{fig:nmfmovielens-x0scaled-obj-time}
    \end{minipage}
    \begin{minipage}[b]{0.49\linewidth}
        \centering
        \includegraphics[width=\textwidth]{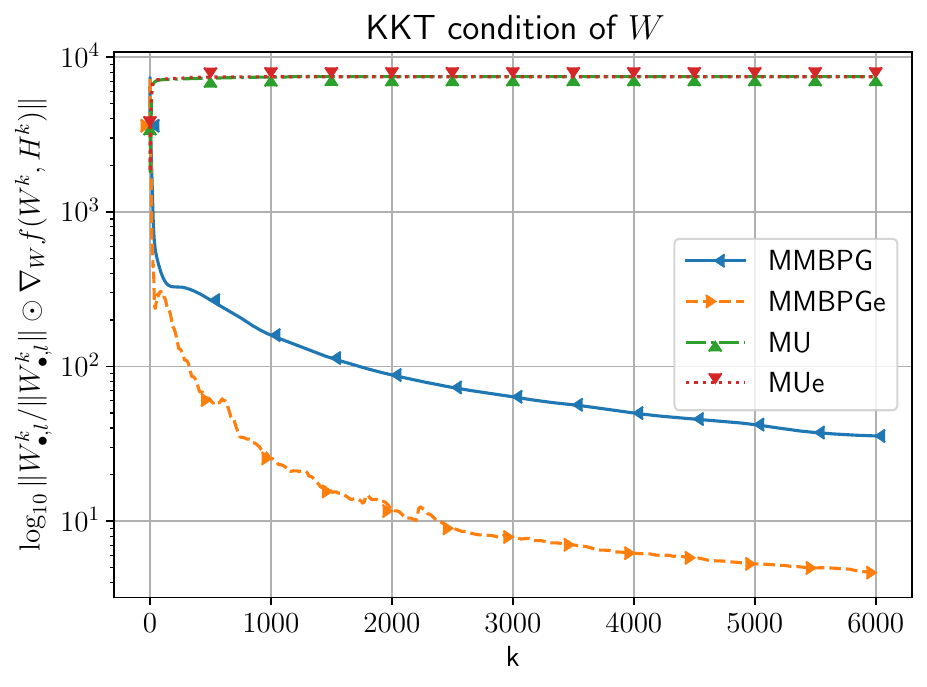}
        \subcaption{The KKT condition with respect to $\W$ with the iteration axis.}
        \label{fig:nmfmovielens-x0scaled-kktw-iter}
    \end{minipage}
    \hfill
    \begin{minipage}[b]{0.49\linewidth}
        \centering
        \includegraphics[width=\textwidth]{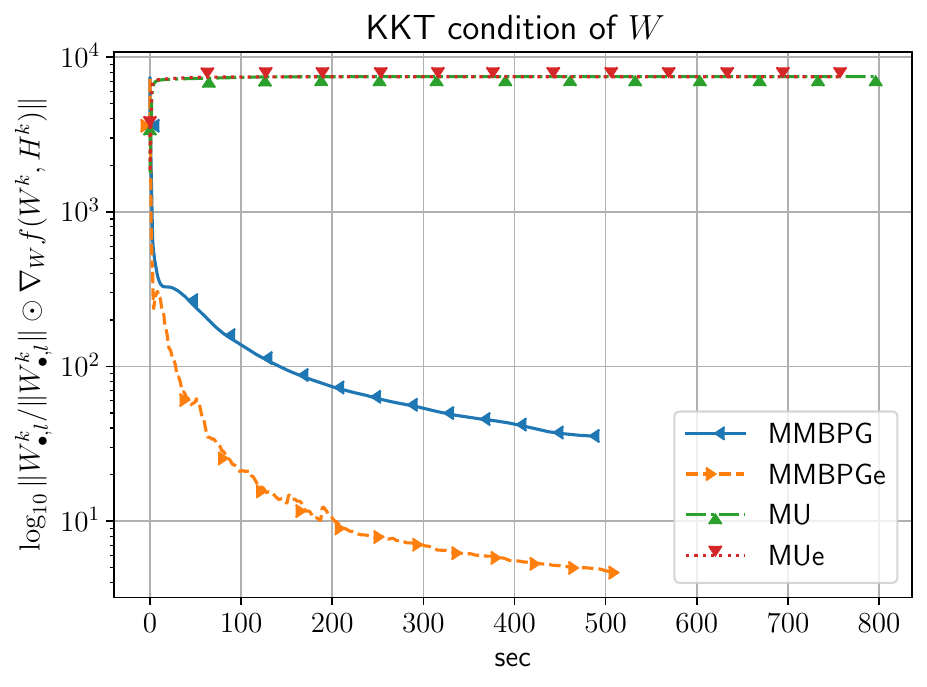}
        \subcaption{The KKT condition with respect to $\W$ with the time axis.}
        \label{fig:nmfmovielens-x0scaled-kktw-time}
    \end{minipage}
    \begin{minipage}[b]{0.49\linewidth}
        \centering
        \includegraphics[width=\textwidth]{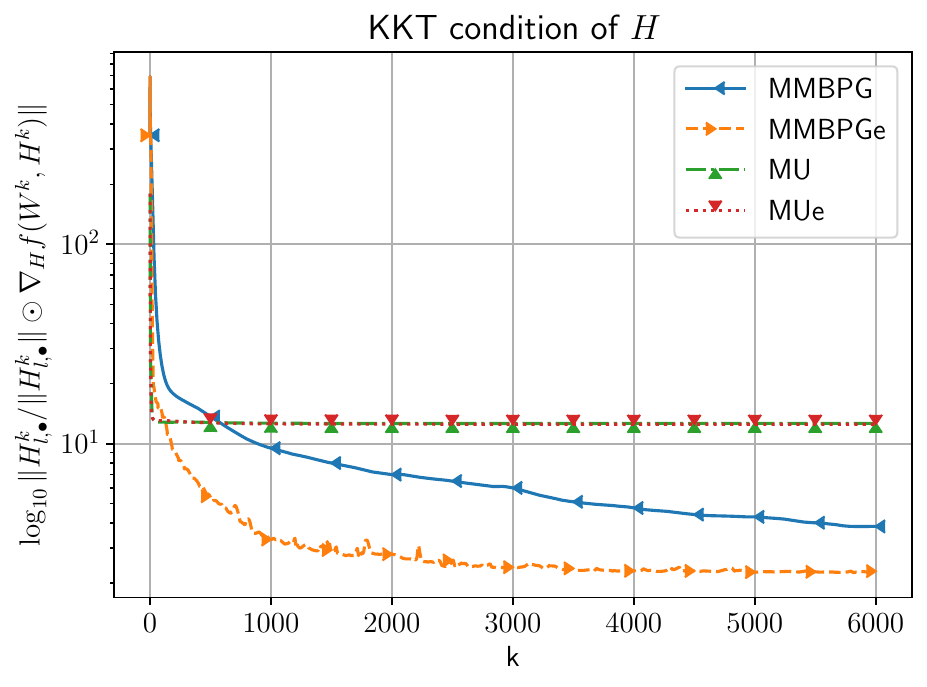}
        \subcaption{The KKT condition with respect to $\H$ with the iteration axis.}
        \label{fig:nmfmovielens-x0scaled-kkth-iter}
    \end{minipage}
    \hfill
    \begin{minipage}[b]{0.49\linewidth}
        \centering
        \includegraphics[width=\textwidth]{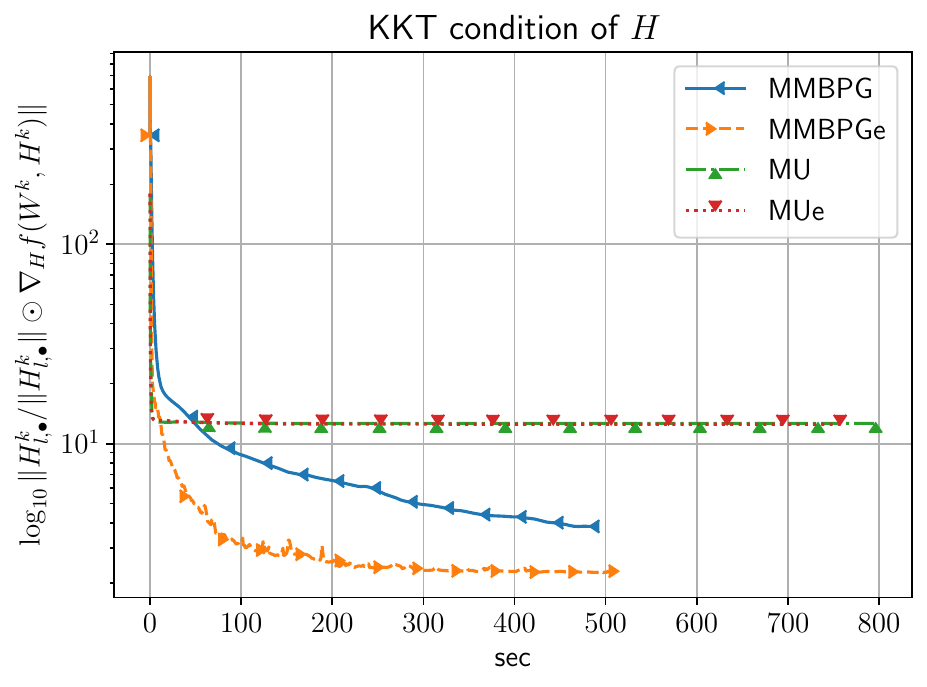}
        \subcaption{The KKT condition with respect to $\H$ with the time axis.}
        \label{fig:nmfmovielens-x0scaled-kkth-time}
    \end{minipage}
    \caption{Comparison with MMBPG (blue), MMBPGe (orange), MU (green), and MUe (red) on KL-NMF (MovieLens 100K Dataset) from the scaled initial point.
    The left column corresponds to the iteration axis, and the right column corresponds to the time axis. The plot of objective function values shows only values below $10^6$.}
    \label{fig:nmfmovielens-x0scaled}
\end{figure}

\section{Conclusions}\label{sec:conclusion}
In this paper, we propose the variants of BPG and BPGe for KL-NMF, called the majorization-minimization Bregman proximal gradient algorithm (MMBPG) and the MMBPG with extrapolation (MMBPGe).
Theoretically, we establish the monotonically decreasing properties of the objective function and the potential function by MMBPG and MMBPGe, respectively.
Moreover, assuming Kurdyka--\L{}ojasiewicz property, we establish that a sequence generated by MMBPG(e) globally converges to a stationary point.
To apply MMBPG(e) to KL-NMF, we define a separable Bregman distance that satisfies the smooth adaptable property, which makes the subproblem of MMBPG(e) solvable in closed form.
MMBPG(e) updates $\W$ and $\H$ simultaneously, unlike existing NMF methods.
In numerical experiments, we have applied MMBPG(e) to KL-NMF on synthetic data and real-world data. The results are promising for real-world applications.

\section*{Acknowledgments}
We thank C\'{e}dric F\'{e}votte and Shotaro Yagishita for their kind comments.

\section*{Declarations}
\textbf{Funding}:
This work was supported by JSPS KAKENHI Grant Number JP23K19953 and JP25K21156, and the Nakajima foundation; JSPS KAKENHI Grant Number JP23H01642; JSPS KAKENHI Grant Number JP20H01951.\\
\textbf{Conflict of interest}:
The authors have no competing interests to declare that are relevant to the content of this article.\\
\textbf{Data availability}:
The datasets generated during and/or analyzed during the current study are available in the GitHub repository, \url{https://github.com/ShotaTakahashi/mmbpg-kl-nmf}.

\addcontentsline{toc}{section}{References}
\bibliographystyle{spmpsci}
\bibliography{main}

\end{document}